\documentclass[10pt,reqno]{amsart}
\usepackage{amssymb,mathrsfs,color}
\usepackage{tikz-cd}
\usepackage{tikz}
\usetikzlibrary{matrix,arrows,decorations.pathmorphing}
\allowdisplaybreaks
\usepackage{enumerate}
\usepackage{graphicx}
\usepackage{xcolor} 
\usepackage{geometry}
\usepackage{amsfonts,amssymb,amsmath}
\usepackage{slashed}
\usepackage{tikz}
\usepackage{ mathrsfs }
\usetikzlibrary{patterns}
\usepackage[titletoc,title]{appendix}
\usepackage{wrapfig}
\usetikzlibrary{arrows,calc,decorations.pathreplacing}

\usepackage{cite}

\definecolor{green}{rgb}{0,0.8,0} 

\renewcommand{\Re}{\mathrm{Re}}
\renewcommand{\Im}{\mathrm{Im}}

\newcommand{\bfa}{{\bf a}}

\newcommand{\bfn}{{\bf n}}

\newcommand{\bfv}{{\bf v}}

\newcommand{\bfx}{{\bf x}}
\newcommand{\bfy}{{\bf y}}

\newcommand{\bbC}{\mathbb C}
\newcommand{\bbD}{\mathbb D}

\newcommand{\bbH}{\mathbb H}

\newcommand{\bbR}{\mathbb R}

\newcommand{\calH}{\mathcal H}

\newcommand{\calN}{\mathcal N}

\newcommand{\calP}{\mathcal P}

\newcommand{\calR}{\mathcal R}
\newcommand{\calS}{\mathcal S}
\newcommand{\calT}{\mathcal T}
\newcommand{\calU}{\mathcal U}

\newcommand{\zed}{\mathfrak{z}}

\newcommand{\zbar}{{\con z}}
\newcommand{\Zbar}{\overline{Z}}

\definecolor{deepgreen}{cmyk}{1,0,1,0.5}

\newcommand{\A}{\mathcal{A}}

\newcommand{\E}{\mathcal{E}}
\newcommand{\F}{\mathcal{F}}

\newcommand{\C}{\mathbb{C}}
\newcommand{\Hp}{\mathbb{H}}
\newcommand{\N}{\mathbb{N}}
\newcommand{\R}{\mathbb{R}}

\newcommand{\D}{\mathbb{D}}

\newcommand{\Lbar}{\bar L}

\newcommand{\al}{\alpha}

\newcommand{\y}{\eta}

\newcommand{\La}{\Lambda}

\makeatletter

\newcommand{\Rmnum}[1]{\expandafter\@slowromancap\romannumeral #1@}
\makeatother

\newcommand{\con}{\overline}

\newcommand{\phibar}{\overline{\phi}}

\newcommand{\abs}[1]{\left\lvert{#1}\right\rvert}

\newcommand{\Align}[1]{\begin{align}\begin{split} #1 \end{split}\end{align}}
\newcommand{\Aligns}[1]{\begin{align*}\begin{split} #1 \end{split}\end{align*}}

\newcommand{\pmat}[1]{\begin{pmatrix} #1 \end{pmatrix}}

\newcommand{\Del}[1]{}

\newcommand{\pt}{&}

\numberwithin{equation}{section}

\newtheorem{theorem}{Theorem}[section]
\newtheorem{corollary}[theorem]{Corollary}
\newtheorem{lemma}[theorem]{Lemma}
\newtheorem{proposition}[theorem]{Proposition}
\newtheorem{claim}[theorem]{Claim}
\newtheorem{remark}[theorem]{Remark}

\newcommand{\pa}{\triangleright}

\renewcommand\Re{\mathrm{Re}\,}
\renewcommand\Im{\mathrm{Im}\,}

\newcommand{\ep}{\varepsilon}

\newcommand{\minus}{\backslash}
\renewcommand{\div}{\mathrm{div}\,}
\newcommand{\AV}{\mathcal{AV}}
\newcommand{\Av}{\mathrm{Av}}

\newcommand{\alphap}{{\alpha^\prime}}
\newcommand{\betap}{{\beta^\prime}}
\newcommand{\tH}{\widetilde{\Hp}}

\newcommand{\wbar}{\overline{w}}

\newcommand{\cH}{\mathcal{H}}
\newcommand{\Hbar}{{\overline{H}}}
\newcommand{\pv}{\mathrm{p.v.}}

\newcommand{\deltabar}{{\overline{\delta}}}
\newcommand{\dt}{\partial_t}
\newcommand{\da}{\partial_\alpha}
\newcommand{\zetab}{\overline{\zeta}}
\newcommand{\ubar}{\overline{u}}
\renewcommand{\pt}{\partial_t}
\newcommand{\Thetabar}{{\overline{\Theta}}}
\renewcommand{\E}{\mathcal{E}}
\renewcommand{\F}{\mathcal{F}}
\newcommand{\thetabar}{{\overline{\theta}}}
\newcommand{\CH}{\calH}
\newcommand{\CHbar}{{\overline{\CH}}}
\newcommand{\zetabar}{{\overline{\zeta}}}
\newcommand{\etabar}{{\overline{\eta}}}
\renewcommand{\pa}{\partial_\alpha}
\renewcommand{\La}{L_\alpha}
\renewcommand{\P}{\mathcal{P}}

\newcommand{\gbar}{{\overline{g}}}
\newcommand{\fibar}{{\overline{f_i}}}
\newcommand{\gibar}{{\overline{g_i}}}

\newcommand{\fbar}{{\overline{f}}}
\newcommand{\pb}{\partial_\beta}

\newcommand{\CK}{\mathcal{K}}
\renewcommand{\H}{\mathbb{H}}
\newcommand{\pap}{\partial_\alphap}

\newcommand{\Gbar}{{\overline{G}}}
\newcommand{\tA}{{\tilde{A}}}

\newcommand{\zedbar}{{\overline{\zed}}}
\newcommand{\Lap}{L_{\alphap}}
\newcommand{\Hcbar}{{\overline{\H}}}
\newcommand{\PH}{\left(\frac{I+\H}{2}\right)}

\newcommand{\PHbar}{\left(\frac{I+\Hcbar}{2}\right)}
\newcommand{\OPHbar}{\left(\frac{I-\Hcbar}{2}\right)}
\newcommand{\Vbar}{{\overline{V}}}
\newcommand{\Wbar}{{\overline{W}}}

\renewcommand{\Lbar}{\overline{L}}
\newcommand{\TPhi}{\widetilde{\Phi}}
\newcommand{\tz}{\widetilde{z}}
\newcommand{\tcalH}{\widetilde{\cH}}
\newcommand{\tL}{\widetilde{L}}
\newcommand{\tcalU}{\widetilde{\calU}}
\newcommand{\tv}{\tilde{v}}
\newcommand{\tvbar}{{\overline{\tilde{v}}}}

\begin{document}
\title[Lifespan of Solutions to the Euler-Poisson System]{On the Motion of a Self-Gravitating Incompressible Fluid with Free Boundary}

\author{Lydia Bieri}
\author{Shuang Miao}
\author{Sohrab Shahshahani}
\author{Sijue Wu}

\begin{abstract}
We consider the motion of the interface separating a vacuum from an inviscid, incompressible, and irrotational fluid, subject to the self-gravitational force and neglecting surface tension, in two space dimensions.  The  fluid motion is described by the Euler-Poission system in moving bounded simply connected domains. 
A family of equilibrium solutions of the system are the perfect balls moving at constant velocity. 
We show that for smooth data which are small perturbations of size $\epsilon$ of these static states, measured in appropriate Sobolev spaces, the solution exists and remains of size $\epsilon$ on a time interval of length at least $c\epsilon^{-2},$ where $c$ is a constant independent of $\epsilon.$ This should be compared with the lifespan $O(\epsilon^{-1})$ provided by local well-posdness. The key ingredient of our proof is finding a nonlinear
transformation  which removes quadratic terms from the nonlinearity. An important difference with the related gravity water waves problem is that unlike the constant gravity for water waves, the self-gravity in the Euler-Poisson system is \emph{nonlinear}. As a first step in our analysis we also show that the Taylor sign condition always holds and establish local well-posedness for this system.
\end{abstract}

\thanks{Support of the National Science Foundation grants  DMS-1253149 for the first and second, NSF-1045119 for the third, and DMS-1361791 for the fourth authors is gratefully acknowledged. The third author was also supported by the NSF under Grant~No.0932078000 while in residence at the MSRI in Berkeley, CA during Fall 2015.}

\maketitle

\section{Introduction}

We consider the motion of the interface separating a vacuum from an inviscid, incompressible, and irrotational fluid, subject to the self-gravitational force in two dimensional spaces. We assume that the fluid domain is bounded and simply connected and the surface tension is zero. Denoting the fluid domain by $\Omega(t)\subset\R^2,$ the fluid velocity by $\bfv$, and the pressure  by $P,$ the evolution is described by the system

\begin{align}\label{main eq}
\begin{cases}
\bfv_{t}+(\bfv\cdot\nabla)\bfv=-\nabla P-\nabla\phi\quad&\mbox{in}\quad \Omega(t),t\geq0,\\
\textrm{div}\,\bfv=0,\quad\textrm{curl}\,\bfv=0\quad &\mbox{in}\quad \Omega(t), t\geq0,\\
P=0\quad &\textrm{on}\quad \partial\Omega(t),
\end{cases}
\end{align}
where the self-gravity Newtonian potential $\phi$ satisfies

\Aligns{
\begin{cases}
\Delta\phi=2\pi\chi_{\Omega(t)},\\
\nabla\phi=\iint_{\Omega(t)}\frac{\bfx-\bfy}{\left|\bfx-\bfy\right|^{2}}d\bfy.
\end{cases}
}
This system is commonly referred to as the incompressible and irrotational Euler-Poisson system. 
 In the equilibrium case where the total force from the pressure and self-gravity are balanced, a ball in $\R^2,$ possibly moving with constant velocity, gives a static solution of the system \eqref{main eq}. 
 An important stability condition for this problem is the Taylor sign condition $$\frac{\partial P}{\partial \bfn}<0,$$ where $\bfn$ is the unit outward pointing normal to the boundary of the fluid region. 
 
 In the three dimensional version of this problem Nordgren \cite{Nor1} proved local well-posedness without the irrotationality assumption, but assuming that initially  the Taylor sign condition holds.  A priori estimates were consequently obtained by Lindblad and Nordgren \cite{LinNor1} in the two dimensional case, but well-posedness was not proved. In our case where the fluid is incompressible and irrotational, the Taylor sign condition holds automatically. 
Indeed by taking divergence of the first equation in \eqref{main eq} and using the fact that $\Delta\phi=2\pi$ in $\Omega(t)$ we see that in $\Omega(t)$
$$\Delta P=-\Delta \phi-|\nabla v|^2 =-2\pi-|\nabla v|^2 <0,$$
so by the Hopf's Maximum principle
$${\frac{\partial P}{\partial \bfn}<0}.$$
The objective of this paper is to show that if  $\epsilon\ll1$ is the size of the difference of the smooth initial data from one of the equilibrium states above, measured in various Sobolev spaces, a unique solution exists and its lifespan has a lower bound of order $O(\epsilon^{-2})$. This should be compared with the $O(\epsilon^{-1})$ estimate from local well-posedness. As a first step in the proof of this result we 
establish local well-posedness for data of arbitrary size. The key to obtaining our long-time $O(\epsilon^{-2})$ estimate  is a fully nonlinear `normal form' transformation which removes quadratic terms from the nonlinearity in the equation. More precisely we find a new unknown and a coordinate change such that in the new coordinates the new unknown satisfies an equation with only cubic and higher order nonlinearity. The use of normal form transformations for evolution PDEs goes back to \cite{ShatahNF, SimonNF} where bilinear transformations of the unknown are used to study nonlinear Klein-Gordon equations. For the gravity water wave problem a fully nonlinear transformation was discovered by the last author in \cite{Wu2, Wu4}. An important difference of the Euler-Poisson system with the gravity water wave problem is that in the Euler-Poisson system the contribution of gravity is \emph{nonlinear}. Indeed in the two dimensional gravity water wave equation the gravity is given by the constant vector $\pmat{0\\-1}.$ Finally note that since the fluid domain $\Omega(t)$ is bounded, dispersive tools are not available to prove global well-posedness at this point. The precise statement of our result is given in Theorem \ref{thm: main} below.

To state our result we first discuss the reduction of the system \eqref{main eq} to a system on the boundary $\partial\Omega(t).$ We occasionally use the notation $\Omega_t:=\Omega(t)$. When there is no risk of confusion we simply write $\Omega$; similarly we occasionally write the parametrization of $\partial\Omega:=\partial\Omega(t)$ as $z=z(\cdot)$ instead of $z=z(t,\cdot)$. Moreover, we use the usual identification $\pmat{x\\y}\mapsto z=x+iy$ of $\R^2$ with $\bbC$ to identify $\Omega$ with a domain in the complex plane. 

Let $z(t,\alpha),~\alpha\in\R$ be a counterclockwise and $2\pi-$periodic Lagrangian parametrization of $\partial\Omega.$ By this we mean $$z_t(t,\alpha)=\bfv(t,z(t,\alpha))$$ so in particular $$z_{tt}(t,\alpha)=\bfv_t(t,z(t,\alpha))+(\bfv\cdot \nabla\bfv)(t,z(t,\alpha))$$ is the acceleration. The conditions $\textrm{div}\,\bfv=0,~\textrm{curl}\,\bfv=0$ now imply that $\overline{\bfv}$ is anti-holomorphic in $\Omega$ and therefore $\zbar_t$ is the boundary value of a holomorphic function in $\Omega.$ It then follows, cf. Proposition \ref{prop: hilbert} in Appendix  \ref{app: Hilbert transform}, that $$\zbar_t=H\zbar_t$$ where $H$ denotes the Hilbert transform associated to $\Omega$ defined by $$Hf(z_0):=\frac{\pv}{\pi i}\int_{\partial\Omega}\frac{f(z)}{z-z_0}dz:=\frac{\pv}{\pi i}\int_0^{2\pi}\frac{f(z(t,\beta))}{z(t,\beta)-z(t,\alpha)}z_\beta(t,\beta)d\beta$$ for $z_0=z(t,\alpha)\in\partial\Omega.$ Since $z$ is a counterclockwise parametrization of $\partial\Omega$ the unit exterior normal of this boundary is given by ${\bold n}:=\frac{-iz_\alpha}{|z_\alpha|},$ and since $P$ is constant on $\partial\Omega$ we can write $\nabla P(t,z)=iaz_\alpha$ for a real-valued function

\Aligns{
a:=-\frac{1}{|z_\alpha|}\frac{\partial P}{\partial {\bold n}}.
}
It follows from \eqref{main eq}, our identification of $\R^2$ with $\bbC,$ and these observations that $z$ satisfies the fully nonlinear system

\Align{\label{z eq temp}
\begin{cases}
z_{tt}+iaz_\alpha=-2\partial_\zbar\phi,\\
H\zbar_t=\zbar_t,
\end{cases}
}
or equivalently

\Align{\label{zbar eq temp}
\begin{cases}
\zbar_{tt}-ia\zbar_\alpha=-2\partial_z\phibar,\\
H\zbar_t=\zbar_t.
\end{cases}
}
The remainder of this paper is devoted to the study of this equation. Note that once a solution $z$ to \eqref{z eq temp} is found, one can recover $\bfv$ by solving the Dirichlet problem

\Aligns{
\begin{cases}
\Delta\bfv=0, \quad \mathrm{in~}\Omega\\
\bfv=z_t,\quad \mathrm{on~}\partial\Omega
\end{cases}.
}
We can now state the main result of this paper. See also Theorems \ref{thm: lwp} and \ref{long time existence theorem} for more quantitative formulations.

\begin{theorem}\label{thm: main}
Let $\Omega_0$ be a bounded simply-connected domain in $\bbC$ with smooth boundary $\partial\Omega_0$ satisfying $\abs{\Omega_0}=\pi,$ and denote the associated Hilbert transform by $H_0.$ Suppose $z_0(\alpha)=e^{i\alpha}+\epsilon f(\alpha)$ is a parametrization of $\partial\Omega_0$ and $z_1(\alpha)=v_0+\epsilon g(\alpha)$ where $f$ and $g$ are smooth and  $g$ satisfies $H_0\overline{g}=\overline{g},$ and $v_0\in\bbC$ is a constant. Then there is $T>0$ and a unique classical solution $z(t,\alpha)$ of \eqref{z eq temp} on $[0,T)$ satisfying $(z(0,\alpha),z_t(0,\alpha))=(z_0(\alpha),z_1(\alpha)).$  Moreover if $\epsilon>0$ is sufficiently small the solution can be extended at least to $T^*=c\epsilon^{-2}$ where $c$ is a constant independent of $\epsilon.$
\end{theorem}

\begin{remark}
The normalization $\abs{\Omega_0}=\pi$ is made only for notational convenience. By the incompressibility of the flow the area of $\Omega(t)$ remains constant during the evolution, and our proof goes through without this assumption by renormalizing the transformations in Section \ref{sec: normal form}.
\end{remark}

\begin{remark}
The constant $v_0\in\bbC$ corresponds to the fact that we consider the stability of the equilibrium solution $e^{i\alpha}+v_0t.$ In practice we work in the center of mass coordinates (see Section \ref{sec: normal form}) to reduce the analysis to the case $v_0=0.$
\end{remark}

We now continue with a brief historical survey of developments related to equation \eqref{main eq} followed by a discussion of the main difficulties in the proof of Theorem \ref{thm: main} and the ideas for resolving them. The mathematical study of the closely related water wave problem goes back to \cite{stokes,  Lev1, GTay1}. Numerous studies on local well-posedness for the water wave problem with or without surface tension, bottom,  and/or vorticity can be found in  \cite{Nalimov, Yos1, Cra1, Wu1, Wu2, ChrLin1, Lin1, Lin3, AmbMas1, AmbMas2, CouShk1, CouShk2, Igu1, Lan1, OgaTan1, ShaZen1, ZhaZha1, ABZ1, ABZ2, HIT1}, and works on water waves with angled crests can be found in \cite{KinWu1, Wu6, Wu5}.  Also as mentioned above, in the presence of self-gravity local well-posedness in dimension three and a-priori estimates in dimension two were obtained in \cite{LinNor1, Nor1}. For the gravity water wave problem, first Wu obtained almost global well-posedness in dimension two in \cite{Wu3}. Then, global well-posedness in three dimensions was solved independently by Wu in \cite{Wu4} and by Germain, Masmoudi, and Shatah in \cite{GMS1}. The $2$d result was later extended to global well-posedness in \cite{IonPus1, AlaDel1, IfrTat1}. See also \cite{GMS2, IfrTat3, IonPus2} for other related developments. As the literature on this subject is vast, we refer the reader to these articles and the references therein as well as the recent survey article \cite{Wunotes} for more comprehensive accounts. We mention that for the almost global and global existence results in \cite{Wu3, Wu4, GMS1, IonPus1, AlaDel1, IfrTat1}, besides the normal form transformations, the proofs rely crucially on the dispersive properties of the localized solutions. 

We now turn to the discussion of the proof of Theorem \ref{thm: main}. As mentioned earlier, our main idea is to find a new unknown and a coordinate change such that the new unknown satisfies an equation with only \emph{cubic} and higher order nonlinearity in the new coordinates. To understand what we mean by cubic we have to specify what kinds of terms are considered to be small. Recall that we are studying  the stability of the static solution\footnote{More precisely we consider the stability of the solutions $z(t,\alpha)=e^{i\alpha}+v_0t$ where $v_0\in\bbC$ is a constant initial velocity. However, by working in the center of mass frame we are able to reduce to the case $v_0=0.$ See Section \ref{subsec: center of mass} for more details.} $z(t,\alpha)\equiv e^{i\alpha},~z_t(t,\alpha)\equiv0.$ It therefore makes sense to consider a quantity depending on $z$ to be small if it is zero when $z$ is the static solution. For instance the quantities $|z|^2-1$ and $z_t$ are considered small, and to say that the nonlinearity is ``cubic and higher order" means that every term in the nonlinearity is the product of at least three small terms. Let $\ep:=|z|^2-1$ and denote by $\delta$ the projection of $\ep$ onto the space of functions which are holomorphic outside $\Omega$ (see Appendix \ref{app: Hilbert transform}), that is,
$$
\delta:=(I-H)\ep
$$
where $H$ is the Hilbert transform. Our first step is to prove that $\delta$ satisfies an equation of the form

$$
(\pt^2+ia\pa-\pi)\delta=\mathrm{cubic}.
$$
Note that to prove energy estimates for this equation we need to have control on the size of the coefficient $a$, and that the dependency of $a$ on $z$ is nonlinear. A careful computation then shows that the contribution of the term $ia\pa\delta$ to the nonlinearity is only quadratic. We remedy this problem by working in a different set of coordinates by introducing an appropriate coordinate transformation $k:\R\to\R.$ Given such $k$ and with the notation

$$
\chi:=\delta\circ k^{-1},\quad A:=(ak_\alpha)\circ k^{-1}, \quad b=k_t\circ k^{-1}
$$
we see that $\chi$ satisfies

$$
\left((\pt+b\pa)^2+iA\pb-\pi\right)\chi(t,\beta)=\mathrm{cubic}.
$$
The idea now is to choose $k$ in such a way that $b$ and $A-\pi$ are quadratic. Here in the static case the transformation $k$ is simply the identity and $A$ is the constant $\pi.$ We will show in Section \ref{sec: normal form} that these conditions will be satisfied if we choose $k$ such that $(I-H)(\log\zbar+ik)=0$ and $k$ satisfies an additional normalization. See Section \ref{sec: normal form} for more details. We refer the reader to \cite{Wu2} for the related transformation in the case of the water wave problem where the origin of these ideas can be found, and where a discussion on the relation with the bilinear normal form transformations of \cite{ShatahNF, SimonNF} is provided. We also emphasize that unlike the gravity water wave problem, the contribution of the self-gravity term on the right hand side of \eqref{z eq temp} is nonlinear. With this choice of $k$ we have obtained our cubic equation, and we focus on the energy estimates for the equation

$$
\left((\pt+b\pa)^2+iA\pb-\pi\right)\Theta=\mathrm{cubic}
$$
where $\Theta=(I-H)f$ for some $f.$ Unfortunately the operator $iA\pb-\pi$ is not positive even when restricted to the class of functions satisfying $\Theta=(I-H)f.$ On the other hand, we observe that if $z(t,e^{i\alpha})\equiv e^{i\alpha},$ so that $A=\pi,$ then a Fourier expansion shows that $i\pb-1$ is indeed positive on the class of functions with only negative frequencies, i.e. $\Theta=(I-\bbH)f$ where now $\bbH$ is the Hilbert transform associated to the unit circle. This suggests that the negative part of  $iA\pb-\pi$ should be higher order with respect to our energy, and in Section \ref{sec: energy estimates} we will show that this is indeed the case. The most natural way to see this structure will be to work with the quantity $(z\circ k^{-1})\Theta$ instead of $\Theta$ (see Lemma \ref{lem: energy positivity}), however to be able to control the negative part of the energy without loss of derivatives, a very careful choice of the energies will be needed. The detailed execution of these ideas is contained in Section \ref{sec: energy estimates}.

Finally we mention that our proof of long-time existence relies on the existence of a local-in-time solution. We have provided a proof of local well-posedness for the system \eqref{z eq temp} in Section \ref{sec: RM}. Here the Riemann mapping is used to transfer the analysis from the evolving domain $\Omega$ to the unit disk $\bbD\subset\bbC.$ A key step in this analysis is obtaining a lower bound on $a,$ which implies the Taylor sign condition. We refer the reader to \cite{Wu1, GTay1} for more details on the significance of this condition.
\subsection{Organization of the paper}
The rest of this paper is organized as follows. In Section \ref{sec: analysis} we collect some analytic tools which are used in the rest of the paper. The proof of the long-time existence statement of Theorem \ref{thm: main} is the content of Sections \ref{sec: normal form}--\ref{sec: long wp}. The proof relies on the existence of a local-in-time solution, but as local well-posedness is not the primary focus of this paper the proof of local well-posedness is postponed to Section \ref{sec: RM}, where Riemann mapping coordinates are introduced and the quasilinear structure of the equation revealed. In Section \ref{sec: normal form} we introduce the normal form transformation and obtain the desired cubic equation discussed above. In Section \ref{sec: relations} we investigate the relation between the original and transformed quantities, and how estimates on one set of quantities translate to estimates for the other set. In section \ref{sec: energy estimates} we introduce the energies and carry out the energy estimates, and finally in Section \ref{sec: long wp} we combine the results from the previous three sections to conclude the proof of long-time existence. Appendix \ref{app: Hilbert transform} contains a review of facts that are used about the Hilbert transform in this paper, and for the convenience of the reader we have provided a list of notations we use, before the references.

\section{Analysis Tools}\label{sec: analysis}
In this section we collect a number general estimates which will be used in the rest of the paper. Most notably we will provide classical estimates on certain singular integral operators adapted to our case. Throughout this section we let $$\zed:[0,2\pi]\to\partial \Omega\subseteq \C$$ be a parametrization of the (closed) boundary of a domain $\Omega$ in $\C.$ We require $\zed$ to be at least $C^1,$ but most of the results in this section hold under the weaker assumption that $\zed$ is Lipchitz. By abuse of notation, for a function $A:\partial\Omega\to\C,$ we write $ A(\alpha)$ instead of $A(\zed(\alpha)).$ In this context $A^\prime(\alpha)$ means $\pa A(\zed(\alpha)),$ so for instance if $A=|\zed|^2-1$ then $A^\prime=2\Re \zedbar \zed_\alpha.$  In the proof of local well-posedness $\zed$ will usually be chosen as $\zed(\alpha)=e^{i\alpha}$ or $\zed(t,\alpha)=Z(t,\alpha).$ For the long-time existence we will often consider $\zed(t,\alpha)=\zeta(t,\alpha)$ or $\zed(t,\alpha)=z(t,\alpha)$ (the definitions for $Z$ and $\zeta$ will be given in later sections).

Even though the functions in this section depend only on $\alpha$ and not on $t,$ we use the notations $L^p_\alpha$ and $H^s_\alpha$ for the Lebesgue and Sobolev spaces in the variable $\alpha$ to be consistent with the rest of the paper. The following standard Sobolev estimate will be used throughout this work, often without reference.

\begin{lemma}\label{lem: Sobolev}[Sobolev] There is a constant $C$ such that for all $f$ in the Sobolev space $H_\alpha^1$
\Aligns{
\|f\|_{\La^\infty}\leq C( \|f\|_{\La^2}+\|\pa f\|_{\La^2}).
}
\end{lemma}
We now turn to the main estimates of this section. We are interested in bounding operators of the forms

\Align{\label{C1}
C_{1}(A,f)(\alpha):=\pv\int_{0}^{2\pi}\frac{\prod_{i\leq m}\left(A_{i}(\alpha)-A_{i}(\beta)\right)}{(\zed(\alpha)-\zed(\beta))^{m+1-k}(\zedbar(\alpha)-\zedbar(\beta))^{k}}f(\beta)d\beta,\qquad k\leq m+1,
}
and 

\Align{\label{C2}
C_{2}(A,f)(\alpha):=\int_{0}^{2\pi}\frac{\prod_{i\leq m}\left(A_{i}(\alpha)-A_{i}(\beta)\right)}{\left(\zed(\alpha)-\zed(\beta)\right)^{m-k}\left(\zedbar(\alpha)-\zedbar(\beta)\right)^{k}}\partial_{\beta}f(\beta)d\beta,\qquad k\leq m.
}
The two propositions below are due, in their original forms, to Calderon \cite{Calderon53}, Coifman, McIntosh, Meyer \cite{CoiMcI1}, Coifman, David, and Meyer \cite{CDM1}, and here we only provide the straightforward modifications necessary for their application in our periodic setting. See also Wu \cite{Wu3} for the proof of the second part of this proposition using these results and the Tb Theorem.

\begin{proposition}\label{prop: C1}
Suppose $\zed$ satisfies 

\Aligns{
\sup_{\alpha\neq\beta}\left|\frac{e^{i\alpha}-e^{i\beta}}{\zed(\alpha)-\zed(\beta)}\right|\leq c_0
}
for some constant $c_0.$ Then there is a constant $C=C(c_0)$ such that the following statements hold.

\begin{enumerate}
\item For any $f\in \La^{2}, A'_{i}\in \La^{\infty}, 1\leq i\leq m$,

\begin{align*}
\|C_{1}(A,f)\|_{\La^{2}}\leq C\|A'_{1}\|_{\La^{\infty}}...\|A'_{m}\|_{\La^{\infty}}\|f\|_{\La^{2}}.
\end{align*}

\item For any $f\in \La^{\infty}, A'_{i}\in \La^{\infty}, 2\leq i\leq m, A'_{1}\in \La^{2}$,

\begin{align*}
\|C_{1}(A,f)\|_{\La^{2}}\leq C\|A'_{1}\|_{\La^{2}}\|A'_{2}\|_{\La^{\infty}}...\|A'_{m}\|_{\La^{\infty}}\|f\|_{\La^{\infty}}.
\end{align*}
\end{enumerate}
\end{proposition}
\begin{proof}
Propositions \ref{prop: C1} is a  consequence of Propositions 3.2 in \cite{Wu3}. Here we describe the modifications necessary to apply this result to our setting. We restrict attention to the case $m=1,~ k=0,$ and write $A$ instead of $A_1.$ The general case can be handled in a similar way. With $\chi$ denoting the characteristic function of the interval $[0,2\pi]$ we have

\begin{align}\label{L2 C1 explicit}
L:=\int_{0}^{2\pi}\left(\int_{0}^{2\pi}\frac{A(\alpha)-A(\beta)}{(\zed(\alpha)-\zed(\beta))^2}f(\beta)d\beta\right)^{2}d\alpha=\int_{\bbR}\chi(\alpha)\left(\int_{\bbR}\frac{A(\alpha)-A(\beta)}{(\zed(\alpha)-\zed(\beta))^2}\chi(\beta)f(\beta)d\beta\right)^{2}d\alpha.
\end{align} 
Since $A$ appears only as $A(\alpha)-A(\beta)$ in this expression, we may assume without loss of generality that $A(0)=A(2\pi)=0.$ We introduce some more notation. First let $\chi_j,~j=1,2,3$ be the characteristic function of the interval $[\frac{2(j-1)\pi}{3},\frac{2j\pi}{3}].$ Next define $\tA$ by $\tA(\alpha)=A(\alpha)$ if $\alpha\in[-4\pi,4\pi]$ and $\tA(\alpha)=0$ if $\alpha\notin[-4\pi,4\pi].$
Let

\begin{align*}
K:=\{w\in\bbC\,|\,w=\dfrac{\zed(\alpha')-\zed(\beta')}{\alpha'-\beta'}\quad \textrm{for some}\quad |\alphap-\betap|\leq \frac{5\pi}{3}\}\subseteq \C.
\end{align*}
From the assumptions of Proposition \ref{prop: C1} it follows that $K$ does not contain the origin $w=0$ in $\C.$ Let $K^\prime\supseteq K$ be a compact set containing $K$ such that $0\notin K^\prime,$ and let $\phi_K$ be a cut-off function supported in $K^\prime$ and equal to one on $K.$ If follows that the function $$F(w):=\frac{\phi_K(w)}{w^2}$$ is smooth. With these definitions we have

\Aligns{
L\lesssim \sum_{i,j=1}^3\int_\R\chi(\alpha)\chi_i(\alpha)\left(\int_\R \frac{\tA(\alpha)-\tA(\beta)}{(\zed(\alpha)-\zed(\beta))^2}\chi(\beta)\chi_j(\beta)f(\beta)
d\beta\right)^2d\alpha=:\sum_{i,j=1}^3 L_{i,j}
}
and we will estimate these integrals separately for different values of $i,j.$ First we treat the case $i=j,$  and for simplicity of notation we assume $i=j=1:$ 

\Aligns{
L_{1,1}&=\int_\R\chi(\alpha)\chi_1(\alpha)\left(\int_\R F\left(\frac{\zed(\alpha)-\zed(\beta)}{\alpha-\beta}\right)\frac{\tA(\alpha)-\tA(\beta)}{(\alpha-\beta)^2}\chi(\beta)\chi_1(\beta)f(\beta)d\beta\right)^2d\alpha\\
&\leq \int_\R\left(\int_\R F\left(\frac{\zed(\alpha)-\zed(\beta)}{\alpha-\beta}\right)\frac{\tA(\alpha)-\tA(\beta)}{(\alpha-\beta)^2}\chi(\beta)\chi_1(\beta)f(\beta)d\beta\right)^2d\alpha\\
&\lesssim_{\|\zed\|_{C^1_\alpha}}\|\tA^\prime\|_{\La^\infty}^2\|\chi\chi_1f\|_{\La^2([0,2\pi])}^2\lesssim \|A^\prime\|_{\La^\infty}^2 \|f\|_{\La^2([0,2\pi])}^2,
}
where we have used Propositions 3.2 in \cite{Wu3} to pass to the last line. The case where $j={i+1}$ is similar, and we now treat the case $i=1,~j=3.$ Using again Propositions 3.2 in \cite{Wu3} and the periodicity of $A,~f,$ and $\zed$ we have

\Aligns{
L_{1,3}&=\int_\R\chi(\alpha)\chi_1(\alpha)\left(\int_\R\frac{A(\alpha)-A(\beta-2\pi)}{(\zed(\alpha)-\zed(\beta-2\pi))^2}\chi(\beta)\chi_3(\beta)
f(\beta-2\pi)d\beta\right)^2d\alpha\\
&=\int_\R\chi(\alpha)\chi_1(\alpha)\left(\int_\R\frac{A(\alpha)-A(\betap)}{(\zed(\alpha)-\zed(\betap))^2}\chi(\betap+2\pi)\chi_3(\betap+2\pi)
f(\betap)d\betap\right)^2d\alpha\\
&\leq\int_\R\left(\int_\R F\left(\frac{\zed(\alpha)-\zed(\betap)}{\alpha-\betap}\right)\frac{\tA(\alpha)-\tA(\betap)}{(\alpha-\betap)^2}\chi(\betap+2\pi)\chi_3(\betap+2\pi)f(\betap)d\betap\right)^2d\alpha\\
&\lesssim_{\|\zed\|_{C^1_\alpha}} \|A^\prime\|_{\La^\infty}^2 \|f\|_{\La^2([0,2\pi])}^2.
}
The remaining cases can be handled using similar arguments.
\end{proof}

A similar argument as in the proof of Proposition \ref{prop: C1} allows us to deduce the following result from Proposition 3.3 in \cite{Wu3}. We omit the proof.
\begin{proposition}\label{prop: C2}
Suppose $\zed$ satisfies 

\Aligns{
\sup_{  \alpha\neq\beta } \left|\frac{e^{i\alpha}-e^{i\beta}}{\zed(\alpha)-\zed(\beta)}\right|\leq c_0
}
for some constant $c_0.$ Then there is a constant $C=C(c_0)$ such that the following statements hold.
\begin{enumerate}
\item For any $f\in \La^{2}, A'_{i}\in \La^{\infty}, 1\leq i\leq m$,

\begin{align*}
\|C_{2}(A,f)\|_{\La^{2}}\leq C\|A'_{1}\|_{\La^{\infty}}...\|A'_{m}\|_{\La^{\infty}}\|f\|_{L^{2}}.
\end{align*}

\item For any $f\in \La^{\infty}, A'_{i}\in \La^{\infty}, 2\leq i\leq m, A'_{1}\in \La^{2}$,

\begin{align*}
\|C_{2}(A,f)\|_{\La^{2}}\leq C\|A'_{1}\|_{\La^{2}}\|A'_{2}\|_{\La^{\infty}}...\|A'_{m}\|_{\La^{\infty}}\|f\|_{\La^{\infty}}.
\end{align*}
\end{enumerate}
\end{proposition}
The next lemma is a simple computation which is used in estimating derivatives of expressions such as $C_1(A,f)$ and $C_2(A,f).$ 

\begin{lemma}\label{lem: kernel der}
Suppose 

\Aligns{
\mathbf{K} f (\alpha)=\pv\int_0^{2\pi}K(\alpha,\beta)f(\beta)d\beta
}
where $K(\alpha,\beta)$ or $e^{i(\alpha}-e^{i\beta})K(\alpha,\beta)$ are continuous and $K$ is $C^1$ away from the diagonal in $[0,2\pi]\times[0,2\pi].$ Then

\Aligns{
\partial_\alpha \mathbf{K} f(\alpha)=\mathbf{K} f_\alpha(\alpha)+\pv\int_0^{2\pi}(\partial_\alpha+\partial_\beta)K(\alpha,\beta)f(\beta)d\beta.
}
\end{lemma}
\begin{proof}
This follows from integration by parts.
\end{proof}

The following two lemmas are important corollaries of Propositions \ref{prop: C1} and \ref{prop: C2} and Lemma \ref{lem: kernel der}. Recall that for a $C^1$ parametrization $\zeta:[0,2\pi]\to\partial\Omega$ of the boundary of $\Omega$ the Hilbert transform is given by

\Aligns{
\CH f(\alpha)=\frac{\pv}{\pi i}\int_0^{2\pi}\frac{f(\beta)}{\zeta(\beta)-\zeta(\alpha)}\zeta_\beta(\beta)d\beta.
}

\begin{lemma}\label{lem: gHf}
Suppose  $\zeta:[0,2\pi]\to\partial\Omega\subseteq \C$ satisfies $\sum_{i=1}^{\ell+1}\|\pa^j\zeta\|_{\La^2}\leq c$ for some nonzero constant $c,$ where $\ell\geq 4$ is a fixed integer, and

\Aligns{
\sup_{  \alpha\neq\beta } \left|\frac{e^{i\alpha}-e^{i\beta}}{\zed(\alpha)-\zed(\beta)}\right|\leq c_0
}
Then there is a constant $C=C(j,c,c_0)$ such that for $4\leq j \leq \ell$

\Aligns{
\sum_{i\leq j}\left\|\pa^i\int_0^{2\pi}\frac{g(\beta)-g(\alpha)}{\zeta(\beta)-\zeta(\alpha)}f(\beta)d\beta\right\|_{\La^2}\leq C\sum_{i\leq j}\|\pa^ig\|_{\La^2}\sum_{i\leq j-1}\|\pa^if\|_{\La^2}.
}
In particular

\Aligns{
\sum_{i\leq j}\left\|\pa^i[g,\CH]\frac{f}{\zeta_\alpha}\right\|_{\La^2}\leq C\sum_{i\leq j}\|\pa^ig\|_{\La^2}\sum_{i\leq j-1}\|\pa^if\|_{\La^2}.
}
\end{lemma}
\begin{proof}
The second estimate follows from the first by writing

\Aligns{
[g,\CH]\frac{f}{\zeta_\alpha}=\frac{1}{\pi i}\int_0^{2\pi}\frac{g(\alpha)-g(\beta)}{\zeta(\beta)-\zeta(\alpha)}f(\beta)d\beta.
}
To prove the first estimate we use Lemma \ref{lem: kernel der} to distribute the derivative on $f$ and $g.$ In the case where all derivatives fall on $f$ Proposition \ref{prop: C2} gives

\Aligns{
\left\|\int_0^{2\pi}\frac{g(\beta)-g(\alpha)}{\zeta(\beta)-\zeta(\alpha)}\pb^jf(\beta)d\beta\right\|_{\La^2} \lesssim\|\pa g\|_{\La^\infty}\|\pa^{j-1}f\|_{\La^2}.
}
When all derivatives fall on $g$ we use the boundedness of the Hilbert transform and Proposition \ref{prop: C1} to estimate

\Aligns{
\left\|\int_0^{2\pi}\frac{\pb^jg(\beta)-\pa^jg(\alpha)}{\zeta(\beta)-\zeta(\alpha)}f(\beta)d\beta\right\|_{\La^2}\lesssim \|\pa^jg\|_{\La^2}\|Hf\|_{\La^\infty}+\|H(f\pa^jg)\|_{\La^2}\lesssim\|\pa^jg\|_{\La^2}(\|f\|_{\La^2}+\|\pa f\|_{\La^2}).
}
The case when $j-1$ derivatives fall on $g$ and one derivative on $f$ can be estimated directly by Proposition \ref{prop: C2}. When $j-1$ derivatives fall on $g$ and none on $f$ we have

\Aligns{
\left\|\int_0^{2\pi}\frac{(\pb^{j-1}g(\beta)-\pa^{j-1}g(\alpha))(\zeta_\beta(\beta)-\zeta_\alpha(\alpha))}{(\zeta(\beta)-\zeta(\alpha))^2}f(\beta)d\beta\right\|_{\La^2}\lesssim&\left\|\int_0^{2\pi}\frac{\pb^{j-1}g(\beta)-\pb^{j-1}g(\alpha)}{(\zeta(\beta)-\zeta(\alpha))^2}f(\beta)\zeta_\beta(\beta)d\beta\right\|_{\La^2}\\
&+\left\|\zeta_\alpha(\alpha)\int_0^{2\pi}\frac{\pb^{j-1}g(\beta)-\pb^{j-1}g(\alpha)}{(\zeta(\beta)-\zeta(\alpha))^2}f(\beta)d\beta\right\|_{\La^2}
}
which can be estimated using Proposition \ref{prop: C1}. All other cases can simply be estimated by bounding the contributions of both $f$ and $g$ in $\La^\infty$ and using the embedding $\La^\infty\hookrightarrow\La^2.$
\end{proof}
\begin{lemma}\label{lem: I-H L2}
Under the assumptions of Lemma \ref{lem: gHf} for any $\ell\geq4$

\Aligns{
\sum_{j\leq \ell}\|\pa^j(I-\CH)f\|_{\La^2}\leq C \sum_{j\leq \ell}\|\pa^jf\|_{\La^2},
}
where $C$ depends on the $H_\al^\ell$ norm of $\zeta.$
\end{lemma}
\begin{proof}
This follows from Lemma \ref{lem: gHf} and Propositions \ref{prop: C1} and \ref{prop: C2} by writing

\Aligns{
\pa^j(I-\CH)f=(I-\CH)\pa^jf-\sum_{i=1}^j\pa^{j-i}[\zeta_\alpha,\CH]\frac{\pa^if}{\zeta_\alpha}=(I-\CH)\pa^jf-\sum_{i=1}^j\pa^{j-i}[\eta,\CH]\frac{\pa^if}{\zeta_\alpha},
}
where $\eta:=\zeta_\alpha-i\zeta,$ Here to compute the commutator $[\pa,\CH]=[\zeta_\alpha,\CH]\frac{\pa}{\zeta_\alpha}$ we have used Lemma \ref{lem: operator H commutator}.
\end{proof}

As another corollary of Proposition \ref{prop: C1} and Lemma \ref{lem: kernel der} we get the following $\La^\infty$ estimate for $C_1(A,f),$ which is similar to Proposition 3.4 in \cite{Wu3}.

\begin{proposition}\label{Linfty intergral hard}
Suppose $\zed$ satisfies

\Aligns{
\sup_{  \alpha\neq\beta } \left|\frac{e^{i\alpha}-e^{i\beta}}{\zed(\alpha)-\zed(\beta)}\right|\leq c_0
}
for some constant $c_0,$ and $A_i^\prime,$ $A_i^{\prime\prime},$ $f,$ and $f^\prime$ are in $\La^\infty.$ Assume further that $\|\zed\|_{H_\alpha^2}\leq M.$ Then there exists a constant $C=C(c_0)$ such that

\begin{align*}
\|C_{1}(A,f)\|_{\La^{\infty}}\leq C(1+M)\prod_{i\leq m}\left(\|A''_{i}\|_{\La^{\infty}}+\|A'_i\|_{\La^\infty}\right)\left(\|f\|_{\La^{\infty}}+\|f'\|_{\La^\infty}\right).
\end{align*}
\end{proposition}
\begin{proof}
This follows from applying the Sobolev inequality to $C(A,f)$ and using Lemma \ref{lem: kernel der} and Proposition \ref{prop: C1}. Note that in view of the embedding $\La^\infty\hookrightarrow\La^2$ we may replace the $\La^2$ norms appearing on the right hand side of the statement of Proposition \ref{prop: C1} by $\La^\infty$ norms. 
\end{proof}

We close this section by stating the following estimates from \cite{Yos1} (see also \cite{Wu1} Lemma 5.2) which are proved using Fourier analysis. The adaptations from the case of the real line to the circle are straightforward and omitted. Here
$\H$ is the Hilbert transform on the circle.

\begin{lemma}\label{lem: Yosihara} Let $r\geq0,$ $q>1/2,$ and $s\geq 1.$ Then for any smooth functions $a$ and $u$

\Aligns{
\|[a,\H]u\|_{H_\alpha^r}\lesssim \|a\|_{H_\alpha^{r+p}}\|u\|_{H_\alpha^{q-p}},\quad p\geq0.
} 
\end{lemma}
\section{The Normal Form Transformation}\label{sec: normal form}
In this section we begin the study of the Cauchy problem of the system \eqref{z eq temp} with small initial data. We start with the following important representation formula for the boundary contribution of the gravity term.

\begin{lemma}\label{lem: gravity}
$-2\partial_\zbar\phi=-\frac{\pi}{2}(I-\Hbar)z=-\pi z+\frac{\pi}{2}(I+\Hbar)z.$
\end{lemma}
\begin{proof}
With $\bfx=z(\alpha)$ and by the dominated convergence theorem

\Aligns{
\nabla\phi(\bfx)=-\int\int_{\Omega}\nabla_\bfy\log(|\bfx-\bfy|)dy=-\lim_{\epsilon\to0}\int\int_{\Omega\minus B_\epsilon}\nabla_\bfy\log(|\bfx-\bfy|)d\bfy
}
where $B_\epsilon$ is a ball of radius $\epsilon$ centered at $\bfx.$ We now identify $\R^2$ with $\C$ in the usual way and abuse notation to write for instance $\nabla\phi=\partial_{x}\phi+i\partial_{y}\phi.$ Defining the vector fields

\Aligns{
X=(\log(|\bfx-\bfy|),0),\qquad Y=(0,\log(|\bfx-\bfy|)),
}
we have

\Aligns{
\nabla\phi(\bfx)=-\lim_{\epsilon\to0}\int\int_{\Omega\minus B_\epsilon}(\div X+i\,\div Y)d\bfy=-\lim_{\epsilon\to0}\int_{\partial(\Omega\minus B_\epsilon)}(X+iY)\cdot N d\sigma(\bfy)
}
where $d\sigma$ is the line element of the boundary and $N$ the outward pointing normal vector. The boundary has two parts: $C_\epsilon$ corresponding to $\partial B_\epsilon$ and $\Gamma_\epsilon$ corresponding to $\partial \Omega_{\epsilon}.$ We can find $\delta_1(\epsilon)$ and $\delta_2(\epsilon)$ which are $O(\epsilon)$ and such that $\Gamma_\epsilon$ is parametrized by $z(\cdot):[0,2\pi]\minus[\alpha-\delta_1,\alpha+\delta_2]\to\Gamma_\epsilon.$ The outward pointing normal vector is therefore given by $-iz_\alpha/|z_\alpha|$ in complex notation or $\frac{1}{|z_\alpha|}(\Im z_\alpha,-\Re z_\alpha)$ in real notation. Similarly there are numbers $\eta_1(\epsilon)<\eta_2(\epsilon)$ in $(0,2\pi)$ such that in complex notation $C_\epsilon$ is parametrized by $\theta\in(\eta_1,\eta_2)\mapsto \bfx+\epsilon e^{i\theta}.$ It follows form the computation above and the $2\pi$-periodicity of $z(\cdot)$ that

\begin{align*}
\nabla\phi(\bfx)=&i\lim_{\epsilon\to0}\int_{\alpha+\delta_1}^{2\pi+\alpha-\delta_2}\log(|z(\alpha)-z(\beta)|)z_\beta(\beta)d\beta\\
&-\lim_{\epsilon\to0}\epsilon\log|\epsilon|\int_{\eta_1}^{\eta_2}(1,i)\cdot N_{C_\epsilon}d\theta\\
=&i\int_{0}^{2\pi}\log(|z(\alpha)-z(\beta)|)\partial_\beta(z(\beta)-z(\alpha))d\beta\\
=&-i\int_0^{2\pi}\frac{(z(\alpha)-z(\beta))\Re((z(\alpha)-z(\beta))\zbar_\beta(\beta))}{|z(\alpha)-z(\beta)|^2}d\beta\\
=&-\frac{i}{2}\int_0^{2\pi}\frac{z(\alpha)-z(\beta)}{\zbar(\alpha)-\zbar(\beta)}\zbar_\beta(\beta)d\beta-\frac{i}{2}\int_0^{2\pi}z_\beta(\beta)d\beta\\
=&-\frac{i}{2}\int_0^{2\pi}\frac{z(\alpha)-z(\beta)}{\zbar(\alpha)-\zbar(\beta)}\zbar_\beta(\beta)d\beta=\frac{\pi}{2}(I-\Hbar)z.
\end{align*}
\end{proof}

In view of Lemma \ref{lem: gravity} we can replace \eqref{z eq temp} by
\begin{align}\label{z eq temp 1}
\begin{cases}
&z_{tt}+iaz_{\alpha}=-\frac{\pi}{2}(I-\Hbar)z,\\
&H\zbar_{t}=\zbar_{t}.
\end{cases}
\end{align}  
The norms in which the data are assumed to be small will be made precise below. Our main objective here to transform equation \eqref{z eq temp 1} to an equation for which the nonlinearity is small of cubic order. Again the exact meaning of the term ``cubic" will be clarified below, but roughly speaking we consider a quantity to be `small' if the corresponding quantity in the case of the static solution $z(t,\alpha)\equiv e^{i\alphap},~z_t(t,\alpha)\equiv0$ is zero. This implies, for instance, that the quantity $(z_\alpha-1)(|z|^2-1)z_t$ is thought of as cubic. However, before we can investigate the structure of \eqref{z eq temp 1} we need to know the existence of a solution, at least locally in time. Theorem \ref{thm: lwp} on local well-posedness for \eqref{z eq temp 1} is therefore the first stepping stone in our analysis. Since local well-posedness is not the focus of this work, we postpone the proof of Theorem \ref{thm: lwp} to Section \ref{sec: RM} and until then we treat it as a black box.

\begin{theorem}\label{thm: lwp}
Let $s\geq 5$. Assume that $z_{0}\in H^{s+\frac{1}{2}}_{\alpha}$ $z_{1}\in H^{s+\frac{1}{2}}_{\alpha}$ and $|z_{0}(\alpha)-z_{0}(\beta)|\geq c'_{0}|e^{i\alpha}-e^{i\beta}|$ for some constant $c'_{0}>0$. Then there is  $T>0$, depending on the norm of the initial data, so that \eqref{z eq temp 1} with initial data $(z,z_t)\vert_{t=0}=(z_0,z_1)$ has a unique solution $z=z(t,\alpha)$ for $t\in[0,T)$ satisfying for all $j\leq s$,

\begin{align*}
&\pa^{j}z, \pa^{j}z_{t}\in C\left([0,T],H^{\frac{1}{2}}_{\alpha}\right),\\
&\pa^{j}z_{tt}\in C\left([0,T],L^{2}_{\alpha}\right),
\end{align*}
and $|z(t,\alpha)-z(t,\beta)|\geq \frac{c'_{0}}{2}|e^{i\alpha}-e^{i\beta}|$ for all $\alpha\neq\beta.$ Moreover, if $T^{*}$ is the supremum over all such time $T$, then either $T^{*}=\infty$, or 

\begin{align*}
\sup_{t< T^*}\left(\|z_{tt}\|_{H_\alpha^{4}}+\|z_t\|_{H_\alpha^\frac{9}{2}}\right)+\sup_{\substack{t< T^*\\ \alpha\neq\beta}}\left|\frac{e^{i\alpha}-e^{i\beta}}{z(t,\alpha)-z(t,\beta)}\right|=\infty.
\end{align*}
\end{theorem}
\begin{remark}
Note that to prove local well-posedness we differentiate equation \eqref{z eq temp} with respect to time (cf. eq \eqref{zt eq}) to reveal the quasilinear structure, and treat the resulting equation as a second order equation for $z_t.$ The original unknown $z$ is then obtained from $z_t$ by integration, which explains the choice of regularity for the initial data. See Section \ref{sec: RM} for more details.
\end{remark}
In what follows we will use the notation

\Aligns{
g^a:=\frac{\pi}{2}(I+\Hbar)z,
}
for the anti-holomorphic part of the contribution of the gravity and

\Aligns{
g^h:=\overline{g^a}=\frac{\pi}{2}(I+H)\zbar,
}
for the holomorphic part of the conjugate of the gravity term. We will show later that $g^a$ and $g^h$ are small in an appropriate sense. With this notation we rewrite the equations for $z$ and $\zbar$ as

\Align{\label{z eq}
z_{tt}+iaz_\alpha=-\pi z+g^a
}
and

\Align{\label{zbar eq}
\zbar_{tt}-ia\zbar_\alpha=-\pi\zbar+g^h.
}
For future reference we also record the time-differentiated versions of equation \eqref{z eq} and \eqref{zbar eq}. Differentiation of \eqref{z eq} and use of anti-holomorphicity of $z_t$ give

\Align{\label{zt eq}
z_{ttt}+iaz_{t\alpha}=-ia_tz_\alpha+\frac{\pi}{2}[\zbar_t,\Hbar]\frac{z_\alpha}{\zbar_\alpha}.
}
Similarly, differentiating \eqref{zbar eq} we get

\Align{\label{zbart eq}
\zbar_{ttt}-ia\zbar_{t\alpha}=ia_t\zbar_\alpha+\frac{\pi}{2}[z_t,H]\frac{\zbar_\alpha}{z_\alpha}.
}
Since $|z|$ is not expected to be small, we want to linearize these equations about the static solution $z_0(\alpha):=e^{i\alpha}$ in some sense, to exploit the smallness of the initial data. This will be achieved in Subsection \ref{subsec: delta eq}, but before that we will need to establish some basic identities involving $H$ and $\Hbar.$ This will be the content of Subsection \ref{subsec: basic identities}. A final point to keep in mind when thinking about the smallness of the solution is that if we start with the static solution $z_{0}(\alpha) := e^{i\alpha}$ but with arbitrary constant initial velocity, then the domain will move
in the direction of the initial velocity without changing its geometry. Therefore to properly interpret small quantities as those which are small when the static solution is the unit disk centered at zero, we need to appropriately renormalize the solution to account for this motion with constant velocity. It turns out that this issue can be resolved simply by choosing coordinates in which the center of mass is static. We begin the analysis in this section by clarifying this point in Subsection \ref{subsec: center of mass}.
\subsection{Center of Mass}\label{subsec: center of mass}

In this subsection we first show that the center of mass $C=C_\Omega(t)$ moves along a straight line with constant speed, that is, $C_{tt}=0,$ which is consistent with the fact that no external force acts on the system. Then we derive a formula for the center of mass only involving quantities defined on the boundary $\partial\Omega(t)$, which will be useful later. We begin by recalling the definition of the center of mass

\Align{\label{center of mass}
C_\Omega(t):=\frac{1}{\pi}\iint\limits_{\Omega}\bfx \,dxdy.
}

\begin{proposition}\label{prop: center of mass}
The center of mass $C:=C_\Omega$ satisfies

\Aligns{
\frac{d^2C}{dt^2}=0.
}
\end{proposition}
\begin{proof}
We prove that

\Align{\label{com 1}
\pi\frac{d^2C}{dt^2}=-\int_{\partial\Omega}\phi \,\vec{n}\,dS,
}
where $\phi$ is the gravity potential, ${\bfn=-\frac{iz_\alpha}{|z_\alpha|}}$ is the exterior unit normal of $\partial\Omega,$ and $dS=|z_\alpha|d\alpha$ is the line element of the boundary. We assume \eqref{com 1} for the moment and prove that the integral on the right hand side vanishes. Recall that $\phi$ satisfies $\partial_z\partial_\zbar\phi=\frac{\pi}{2}$ inside $\Omega.$ Integration in $z$ gives $\partial_\zbar\phi=\frac{\pi}{2}z+A(\zbar)$ where $A$ is an anti-holomorphic function inside $\Omega,$ and another integration in $\zbar$ gives $\phi=\frac{\pi}{2}z\zbar+A(\zbar)+B(z)$ for some holomorphic function $B.$ Moreover, from Lemma \ref{lem: gravity}  we know that for points on the boundary $\partial_\zbar\phi=\partial_\zbar(\phi-B)=\frac{\pi}{4}(I-\Hbar)z.$ With this notation we rewrite \eqref{com 1} as

\Align{\label{com 3}
\pi\frac{d^2C}{dt^2}=&i\int_0^{2\pi}\left(\frac{\pi}{2}z\zbar+A(\zbar)\right)z_\alpha d\alpha+i\int_{\partial\Omega} B(z)dz\\
=&-\frac{\pi i}{2}\int_0^{2\pi}\zbar\, z\, z_\alpha d\alpha-\frac{\pi i}{4}\int_0^{2\pi}((I-\Hbar)z) \,z\,\zbar_\alpha d\alpha\\
=&\frac{\pi i}{4}\int_0^{2\pi}z(\Hbar z) \,\zbar_\alpha d\alpha=\frac{\pi i}{4}\int_{\partial\Omega}z\Hbar z d\zbar.
}
Now recall that the (conjugate) Hilbert transform is defined as

\Aligns{
\Hbar f(\zbar):=-\frac{\pv}{\pi i}\int_{\partial\Omega}\frac{f(w)}{\wbar-\zbar}d\wbar:=-\frac{1}{\pi i }\lim_{\epsilon\to0}\int_{\partial\Omega\backslash B_\epsilon(\zbar)}\frac{f(w)}{\wbar-\zbar}d\wbar,}
where the last limit converges in the $L^2$ sense. In particular if $f,g\in L^2$ then 

\Aligns{
-\pi i{\int_{\partial\Omega}}g(z)\Hbar f(z)d\zbar=&{\int_{\partial\Omega}}g(z)\lim_{\epsilon\to0}\int_{\partial\Omega\backslash B_\epsilon(z)}\frac{f(w)}{\wbar-\zbar}d\wbar d\zbar=\lim_{\epsilon\to0}\int_{\partial\Omega}\int_{\partial\Omega-B_\epsilon(z)}\frac{g(z)f(w)}{\wbar-\zbar}d\wbar d\zbar\\
=&\lim_{\epsilon\to0}\int_{\partial\Omega}\int_{\partial\Omega-B_\epsilon(w)}\frac{g(z)f(w)}{\wbar-\zbar}d\zbar d\wbar=\lim_{\epsilon\to0}\int_{\partial\Omega}\int_{\partial\Omega-B_\epsilon(z)}\frac{g(w)f(z)}{\zbar-\wbar}d\wbar d\zbar\\
=&\int_{\partial\Omega}f(z)\lim_{\epsilon\rightarrow0}
\int_{\partial\Omega-B_{\epsilon}(z)}\frac{g(w)}{\zbar-\wbar}d\wbar d\zbar=\pi i\int_{\partial\Omega}f(\zbar)\Hbar g(\zbar)d\zbar.
}
Applying this observation to $f(z)=g(z)=z$ we see that $\int_{\partial\Omega}z\Hbar z d\zbar=-\int_{\partial\Omega}z\Hbar zd\zbar$ and therefore in view of \eqref{com 3} we get $\frac{d^2C}{dt^2}=0.$ Finally we establish \eqref{com 1} by direct differentiation. For this we denote the flow map by $X,$ that is, $$X(t,\cdot):\Omega(0)\to\Omega(t)$$ satisfies $\frac{dX(t,\bfx)}{dt}=V(t,X(t,\bfx)),~\frac{d^2X(t,\bfx)}{dt^2}=-\nabla P(t,X(t,\bfx))-\nabla\phi(t,X(t,\bfx)).$ Then since the flow is incompressible we have

\Aligns{
C=\frac{1}{\pi}\iint\limits_{\Omega(0)}X(t,\bfx^\prime)d\bfx^\prime,
}
and hence

\Aligns{
\pi\frac{d^2C}{dt^2}=&-\iint\limits_{\Omega(0)}\nabla \left(P(t,X(t,\bfx^\prime))+\nabla\phi(t,X(t,\bfx^\prime))\right)d\bfx^\prime
=-\iint\limits_{\Omega(t)}(\nabla P(t,\bfx)+\nabla\phi(t,\bfx))d\bfx=-\int_{\partial\Omega}\phi\,\vec{n}\,dS,
}
as desired.
\end{proof}
The formula \eqref{center of mass} is in terms of the domain $\Omega(t)$, but since we work with the boundary equation \eqref{z eq temp 1}, it is more convenient to derive a formula for center of mass only involving quantities defined on the boundary. This is achieved in the following proposition.

\begin{proposition}\label{prop: bdry formula CM}
Let us denote

\begin{align}\label{def: ep delta}
\ep:=|z|^{2}-1,\quad \delta:=(I-H)\ep.
\end{align}
Then the center of mass \eqref{center of mass} as a complex number can be written as

\begin{align}\label{bdry formula CM}
C_{\Omega(t)}=-\frac{i}{2\pi}\int_{0}^{2\pi}\ep(t,\alpha)z_{\alpha}(t,\alpha)d\alpha=-\frac{i}{4\pi}\int_{0}^{2\pi}\delta(t,\alpha)z_{\alpha}(t,\alpha)d\alpha
\end{align}
where $z(t,\alpha)$ is the parametrization of $\partial\Omega(t)$. 
\end{proposition} 
\begin{proof}
We can write the center of mass as $\frac{1}{\pi}\iint_{\Omega(t)}(x+iy)dxdy$. Using the divergence theorem, we have

\begin{align*}
\iint_{\Omega(t)}xdxdy=\iint_{\Omega(t)}\textrm{div}\left(\frac{x^{2}}{2},0\right)dxdy=&\int_{\partial\Omega(t)}\left(\frac{x^{2}}{2},0\right)\cdot \left(\frac{y_{\alpha}}{|z_{\alpha}|},-\frac{x_{\alpha}}{|z_{\alpha}|}\right)ds\\
=&\frac{1}{2}\int_{0}^{2\pi}x^{2}y_{\alpha}d\alpha=-\frac{i}{2}\int_{0}^{2\pi}x^{2}(x_{\alpha}+iy_{\alpha})d\alpha
\end{align*}
and

\begin{align*}
i\iint_{\Omega(t)}ydxdy=i\iint_{\Omega(t)}\textrm{div}\left(0,\frac{y^{2}}{2}\right)dxdy=&i\int_{\partial\Omega(t)}\left(0,\frac{y^{2}}{2}\right)\cdot \left(\frac{y_{\alpha}}{|z_{\alpha}|},-\frac{x_{\alpha}}{|z_{\alpha}|}\right)ds\\
=&-\frac{i}{2}\int_{0}^{2\pi}y^{2}x_{\alpha}d\alpha=-\frac{i}{2}\int_{0}^{2\pi}y^{2}(x_{\alpha}+iy_{\alpha})d\alpha.
\end{align*}
Therefore we have

\begin{align*}
\iint_{\Omega(t)}(x+iy)dxdy=-\frac{i}{2}\int_{0}^{2\pi}|z|^{2}z_{\alpha}d\alpha=&-\frac{i}{2}\int_{0}^{2\pi}\ep z_{\alpha}d\alpha\\
=&-\frac{i}{2}\int_{0}^{2\pi}\left(\frac{I-H}{2}\ep\right) z_{\alpha}d\alpha=-\frac{i}{4}\int_{0}^{2\pi}\delta z_{\alpha}d\alpha.
\end{align*}
This completes the proof.
\end{proof}
We have the following corollary.

\begin{corollary}\label{prop: trans CM}
Let $v^{0}_{c}$ and $c^{0}$ be the initial velocity and position of the center of mass respectively. If $z=z(t,\alpha)$ is a solution to \eqref{z eq temp 1} then $z(t,\alpha)-c^{0}-v^{0}_{c}t$ is also a solution to \eqref{z eq temp 1}. Moreover, $z-c^{0}-v^{0}_{c}t$ parametrizes the boundary of a domain whose center of mass is always at the origin.
\end{corollary}
\begin{proof}
This follows from Proposition \ref{prop: center of mass} and the fact that $z_{tt}, z_{\alpha}$ and $H$ are invariant under the transformation

\begin{align}\label{trans CM}
z(\alpha,t)\mapsto z(\alpha,t)-c^{0}-v^{0}_{c}t.
\end{align}
\end{proof}

In what follows, we will only consider this normalized solution to \eqref{z eq temp 1}, that is we assume that the center of mass is always at the origin, and this assumption is justified by Corollary~\ref{prop: trans CM}. Therefore in view of Proposition \ref{prop: bdry formula CM} and Corollary \ref{prop: trans CM} we always have

\begin{align*}
\int_{0}^{2\pi}\ep z_{\alpha}d\alpha=\int_{0}^{2\pi}\delta z_{\alpha}d\alpha=0.
\end{align*}


\subsection{Basic Identities}\label{subsec: basic identities}

In this subsection we record some basic identities which will be used in the remainder of this work. A few more standard properties of the Hilbert transform are recalled in Appendix \ref{app: Hilbert transform}. In the remainder of this section we assume that the parametrization $z$ of $\partial\Omega(t)$ has regularity $C^2_{t,\alpha}.$

\subsubsection{Commutation Relations}

We compute the commutators of various operators with the Hilbert transform. 

\begin{lemma}\label{lem: operator H commutator}
For any $2\pi-$periodic function $f$  in $C^2_{t,\alpha}$

\begin{enumerate}[(i)]
\item $[\partial_t,H]f=[z_t,H]\frac{f_\alpha}{z_\alpha},$
\item $[\partial_t^2,H]f=2[z_t,H]\frac{f_{t\alpha}}{z_\alpha}+[z_{tt},H]\frac{f_\alpha}{z_\alpha}+\frac{1}{\pi i}\int_0^{2\pi}\left(\frac{z_t(\beta)-z_t(\alpha)}{z(\beta)-z(\alpha)}\right)^2f_\beta(\beta)d\beta,$
\item$\partial_\alpha Hf=z_\alpha H\frac{f_\alpha}{z_\alpha},$
\item$[a\partial_\alpha,H]f=[az_\alpha,H]\frac{f_\alpha}{z_\alpha},$
\item$[\partial_t^2+ia\partial_\alpha,H]f=-\frac{\pi}{2}[(I-\Hbar)z,H]\frac{f_\alpha}{z_\alpha}+2[z_t,H]\frac{f_{t\alpha}}{z_\alpha}+\frac{1}{\pi i}\int_0^{2\pi}\left(\frac{z_t(\beta)-z_t(\alpha)}{z(\beta)-z(\alpha)}\right)^2f_\beta(\beta)d\beta.$
\end{enumerate}
\end{lemma}
\begin{proof}[Proof of Lemma \ref{lem: operator H commutator}]
\begin{enumerate}[(i)]
\item \Aligns{
[\partial_t,H]f&=\frac{\pv}{\pi i}\int_0^{2\pi}\left(\frac{f(\beta)z_{t\beta}(\beta)}{z(\beta)-z(\alpha)}-\frac{(z_t(\beta)-z_t(\alpha))f(\beta)z_\beta(\beta)}{(z(\beta)-z(\alpha))^2}\right)d\beta\\
&=-\frac{\pv}{\pi i}\int_0^{2\pi}\frac{(z_t(\beta)-z_t(\alpha))f_\beta(\beta)}{(z(\beta)-z(\alpha))z_\beta(\beta)}z_\beta(\beta)d\beta=[z_t,H]\frac{f_\alpha}{z_\alpha}.
}
\item \Aligns{
[\partial_t^2,H]f=&\partial_t([z_t,H]\frac{f_\alpha}{z_\alpha})+\partial_t H\partial_t f=H\partial_t^2f\\
=&\partial_t([z_t,H]\frac{f_\alpha}{z_\alpha})+[z_t,H]\frac{f_{t\alpha}}{z_\alpha}\\
=&[z_{tt},H]\frac{f_\alpha}{z_\alpha}+2[z_t,H]\frac{f_{t\alpha}}{z_\alpha}+\frac{1}{\pi i}\int_{0}^{2\pi}\left(\frac{z_t(\beta)-z_t(\alpha)}{z(\beta)-z(\alpha)}\right)^2f_\beta(\beta)d\beta.
}
\item \Aligns{
\partial_\alpha H f&=\frac{\pv}{\pi i}\int_0^{2\pi}\frac{f(\beta)}{(z(\beta)-z(\alpha))^2}z_\alpha(\alpha)z_\beta(\beta)d\beta\\
&=z_\alpha(\alpha)\frac{\pv}{\pi i}\int_0^{2\pi}\frac{f_\beta(\beta)}{z\beta)-z(\alpha)}d\beta=z_\alpha H\frac{f_\alpha}{z_\alpha}.
}
\item \Aligns{
[a\partial_\alpha,H]f=az_\alpha H\frac{f_\alpha}{z_\alpha}-H(af_\alpha)=[az_\alpha,H]\frac{f_\alpha}{z_\alpha}.
}
\item This part is a corollary of the previous parts combined with equation \eqref{z eq temp} and Lemma \ref{lem: gravity}.
\end{enumerate}
\end{proof}
\begin{lemma}\label{lem: partialt f H g}
For any $2\pi-$periodic function $f$ and $g$  in $C^2_{t,\alpha}$

\Aligns{
\partial_t[f,H]g=[f_t,H]g+[f,H]g_t+f[z_t,H]\frac{f_\alpha}{z_\alpha}-[z_t,H]\frac{\partial_\alpha(fg)}{z_\alpha}.
}
\end{lemma}
\begin{proof}
Using part $(i)$ of Lemma \ref{lem: operator H commutator} we get

\Aligns{
\partial_t[f,H]g=\partial_t(fHg)-\partial_tH(fg)&=f_tHg+fHg_t+f[z_t,H]\frac{g_\alpha}{z_\alpha}-H(f_tg)-H(fg_t)-[z_t,H]\frac{\partial_\alpha(fg)}{z_\alpha}\\
&=[f_t,H]g+[f,H]g_t+f[z_t,H]\frac{f_\alpha}{z_\alpha}-[z_t,H]\frac{\partial_\alpha(fg)}{z_\alpha}.
}
\end{proof}
Next we recored the following important computation relating $[\zbar,\Hbar]z$ and the area of $\Omega.$

\begin{lemma}\label{lem: zbar Hbar commutator}
If $z:[0,2\pi]\to\partial\Omega$ is a counterclockwise parametrization then 
\Aligns{
[\zbar,\Hbar]z=-\frac{2|\Omega|}{\pi}=-2.
}
\end{lemma}
\begin{proof}
Since the parametrization is counterclockwise the exterior normal $\bfn$ is given by $$\bfn=-\frac{iz_\alpha}{|z_\alpha|}=\frac{y_\alpha-ix_\alpha}{|z_\alpha|}$$ in complex notation. It follows that with $z=x+iy$

\Aligns{
[\zbar,\Hbar]z&=\frac{1}{\pi i}\int_0^{2\pi }z(\beta)\zbar_\beta(\beta)d\beta=\frac{1}{2\pi i}\int_0^{2\pi}(z(\beta)\zbar_\beta(\beta)-\zbar(\beta)z_\beta(\beta))d\beta\\
&=\frac{1}{\pi}\int_0^{2\pi}\Im(z(\beta)\zbar_\beta(\beta))d\beta=\frac{1}{\pi}\int_0^{2\pi}(x_\beta(\beta)y(\beta)-y_\beta(\beta) x(\beta))d\beta\\
&=-\frac{1}{\pi}\int_{\partial\Omega}\pmat{x\\y}\cdot \bfn\,|z_\beta(\beta)|d\beta=-\frac{1}{\pi}\iint_{\Omega}\div\pmat{x\\y}dxdy=-\frac{2|\Omega|}{\pi}.
}
\end{proof}

\begin{lemma}\label{lem: z H}
For any $2\pi-$periodic function $f$  in $C^2_{t,\alpha}$

\Aligns{
[z,H]\frac{f_\alpha}{z_\alpha}=0
}
\end{lemma}
\begin{proof}
This is an immediate consequence of the definition of the Hilbert transform and the periodicity of $f.$
\end{proof}

\begin{lemma}\label{lem: fgh commutator}
For any $2\pi-$periodic function $f$, $g,$ and $h$ in $C^2_{t,\alpha}$

\Aligns{
[fg,H]h=f[g,H]h+[f,H](gh).
}
\end{lemma}
\begin{proof}
\Aligns{
[fg,H]h=fgHh-fH(gh)+fH(gh)-H(fgh)=f[g,H]h+[f,H](gh).
}
\end{proof}
\begin{lemma}\label{lem: H plus Hbar}
Suppose $f$ and $g$ are $2\pi-$periodic functions in $C^{2}_{t,\alpha}$ which are anti-holomorphic inside $\Omega.$ Then with the notation $\ep=|z|^{2}-1$

\Aligns{
[f,H\frac{1}{z_\alpha}+\Hbar\frac{1}{\zbar_\alpha}]g_\alpha=-\frac{1}{\pi i}\int_0^{2\pi}\frac{(f(\alpha)-f(\beta))g_\beta(\beta)\zbar(\alpha)\zbar(\beta)\left(\frac{\ep(\alpha)}{\zbar(\alpha)}-\frac{\ep(\beta)}{\zbar(\beta)}\right)}{|z(\beta)-z(\alpha)|^2}d\beta.
}
\end{lemma}
\begin{proof}
\Aligns{
[f,H\frac{1}{z_\alpha}+\Hbar\frac{1}{\zbar_\alpha}]g_\alpha=&\frac{\pv}{\pi i}\int_0^{2\pi}\left(\frac{1}{z(\beta)-z(\alpha)}-\frac{1}{\zbar(\beta)-\zbar(\alpha)}\right)(f(\alpha)-f(\beta))g_\beta(\beta)d\beta\\
=&-[f,\Hbar]\frac{g_\alpha}{\zbar_\alpha}+\zbar[f,\Hbar]\frac{\zbar g_\alpha}{\zbar_\alpha}-\frac{1}{\pi i}\int_0^{2\pi}\frac{(f(\alpha)-f(\beta))g_\beta(\beta)\zbar(\alpha)\zbar(\beta)\left(\frac{\ep(\alpha)}{\zbar(\alpha)}-\frac{\ep(\beta)}{\zbar(\beta)}\right)}{|z(\beta)-z(\alpha)|^2}d\beta.
}
Since $\frac{g_\alpha}{\zbar_\alpha}$ is anti-holomorphic in side $\Omega,$ the first two terms on the last line above are zero, and this proves the lemma. 
\end{proof}
\subsubsection{The Relation between $H$ and $\Hbar$}

In the static case where the boundary of the domain $\Omega$ is exactly the unit circle, the corresponding Hilbert transform $\H$ satisfies $\overline{\H}=-\H+2\Av$ where $\Av (f):=\frac{1}{2\pi}\int_0^{2\pi}f(\alpha)d\alpha$.  Here we prove an important lemma which quantifies the failure of this identity when $\Omega$ is a small perturbation of the unit disc.

\begin{lemma}\label{lem: H Hbar}
For any $2\pi-$periodic function $f$  in $C^2_{t,\alpha}$

\Align{\label{H Hbar eq}
\Hbar f&=-zH\frac{f}{z}+z[\varepsilon,H]\frac{f_\alpha}{z_\alpha}+E(f)\\
&=-Hf-[z,H]\frac{f}{z}+z[\varepsilon,H]\frac{f_\alpha}{z_\alpha}+E(f)
}
where $\ep:=|z|^2-1$ and $E(f)=E_1(f)+E_2(f)+E_3(f)$ with

 \Aligns{E_{1}(f):=&-\frac{1}{\pi i}\int_{0}^{2\pi}\frac{f(\beta)\left(\frac{\epsilon(\alpha)}{z(\alpha)}-\frac{\epsilon(\beta)}{z(\beta)}\right)(z(\alpha)z(\beta))^{2}}{(z(\alpha)-z(\beta))^{2}}\partial_{\beta}\left(\frac{\epsilon(\beta)}{z(\beta)}\right)d\beta\\
E_{2}(f):=&-\frac{1}{\pi i}\int_{0}^{2\pi}\frac{f(\beta)\left(\frac{\epsilon(\beta)}{z(\beta)}-\frac{\epsilon(\alpha)}{z(\alpha)}\right)^{2}(z(\alpha)z(\beta))^{2}}{(z(\beta)-z(\alpha))|z(\beta)-z(\alpha)|^{2}}\partial_{\beta}\left(\frac{\epsilon(\beta)}{z(\beta)}\right)d\beta\\
E_{3}(f):=&-\frac{1}{\pi i}\int_{0}^{2\pi}\frac{f(\beta)\left(\frac{\epsilon(\beta)}{z(\beta)}-\frac{\epsilon(\alpha)}{z(\alpha)}\right)^{2}(z(\alpha)z(\beta))^{2}}{(z(\beta)-z(\alpha))|z(\beta)-z(\alpha)|^{2}}\partial_{\beta}\left(\frac{1}{z(\beta)}\right)d\beta.
}
\end{lemma}
\begin{proof}
 Recall the following relations:

\begin{align*}
\zbar=\frac{1+\epsilon}{z},\quad \zbar_{\beta}(\beta)=\frac{\epsilon_{\beta}(\beta)z(\beta)-z_{\beta}(\beta)(1+\epsilon(\beta))}{\left(z(\beta)\right)^{2}}.
\end{align*}
We have

\begin{align}\label{Hbar pre}
\begin{split}
(\overline{H}f)(\alpha)=&-\frac{1}{\pi i}\int_{0}^{2\pi}\frac{f(\beta)\zbar_{\beta}(\beta)}{\zbar(\beta)-\zbar(\alpha)}d\beta\\
=&-\frac{1}{\pi i}\int_{0}^{2\pi}\frac{f(\beta)}{\frac{1+\epsilon(\beta)}{z(\beta)}-\frac{1+\epsilon(\alpha)}{z(\alpha)}}\left(\left(\frac{\epsilon(\beta)}{z(\beta)}\right)_{\beta}-\frac{z_{\beta}(\beta)}{(z(\beta))^{2}}\right)d\beta\\
=&-\frac{1}{\pi i}\int_{0}^{2\pi}\frac{f(\beta)}{\frac{1}{z(\beta)}-\frac{1}{z(\alpha)}}\left(\left(\frac{\epsilon(\beta)}{z(\beta)}\right)_{\beta}-\frac{z_{\beta}(\beta)}{(z(\beta))^{2}}\right)d\beta\\
&-\frac{1}{\pi i}\int_{0}^{2\pi}\frac{f(\beta)\left(\frac{\epsilon(\alpha)}{z(\alpha)}-\frac{\epsilon(\beta)}{z(\beta)}\right)}{\left(\frac{1}{z(\beta)}-\frac{1}{z(\alpha)}\right)\left(\frac{1}{z(\beta)}-\frac{1}{z(\alpha)}+\frac{\epsilon(\beta)}{z(\beta)}-\frac{\epsilon(\alpha)}{z(\alpha)}\right)}\left(\left(\frac{\epsilon(\beta)}{z(\beta)}\right)_{\beta}-\frac{z_{\beta}(\beta)}{(z(\beta))^{2}}\right)d\beta\\
=&-\frac{1}{\pi i}\int_{0}^{2\pi}\frac{f(\beta)}{\frac{1}{z(\beta)}-\frac{1}{z(\alpha)}}\left(\left(\frac{\epsilon(\beta)}{z(\beta)}\right)_{\beta}-\frac{z_{\beta}(\beta)}{(z(\beta))^{2}}\right)d\beta\\
&-\frac{1}{\pi i}\int_{0}^{2\pi}\frac{f(\beta)\left(\frac{\epsilon(\alpha)}{z(\alpha)}-\frac{\epsilon(\beta)}{z(\beta)}\right)}{\left(\frac{1}{z(\beta)}-\frac{1}{z(\alpha)}\right)^{2}}\left(\left(\frac{\epsilon(\beta)}{z(\beta)}\right)_{\beta}-\frac{z_{\beta}(\beta)}{(z(\beta))^{2}}\right)d\beta\\
&-\frac{1}{\pi i}\int_{0}^{2\pi}\frac{f(\beta)\left(\frac{\epsilon(\alpha)}{z(\alpha)}-\frac{\epsilon(\beta)}{z(\beta)}\right)^{2}}{\left(\frac{1}{z(\beta)}-\frac{1}{z(\alpha)}\right)^{2}\left(\frac{1}{z(\beta)}-\frac{1}{z(\alpha)}+\frac{\epsilon(\beta)}{z(\beta)}-\frac{\epsilon(\alpha)}{z(\alpha)}\right)}\left(\left(\frac{\epsilon(\beta)}{z(\beta)}\right)_{\beta}-\frac{z_{\beta}(\beta)}{(z(\beta))^{2}}\right)d\beta.
\end{split}
\end{align}
The `constant term' above (the second term in the first line) is 

\begin{align}\label{Hbar const}
\begin{split}
&\frac{1}{\pi i}\int_{0}^{2\pi}\frac{f(\beta)z(\alpha)z(\beta)}{z(\alpha)-z(\beta)}\frac{z_{\beta}(\beta)}{(z(\beta))^{2}}d\beta\\
=&\frac{1}{\pi i}\int_{0}^{2\pi}\frac{f(\beta)z(\alpha)z_{\beta}(\beta)}{\left(z(\alpha)-z(\beta)\right)z(\beta)}d\beta=-zH\left(\frac{f}{z}\right).
\end{split}
\end{align}
The `linear terms' above (the first term in the first line and the second term in the second line) are given by

\begin{align}\label{Hbar linear 1}
-\frac{1}{\pi i}\int_{0}^{2\pi}\frac{f(\beta)z(\alpha)z(\beta)}{z(\alpha)-z(\beta)}\left(\frac{\epsilon(\beta)}{z(\beta)}\right)_{\beta}d\beta.
\end{align}
and

\begin{align}\label{Hbar linear 2}
\begin{split}
&\frac{1}{\pi i}\int_{0}^{2\pi}f(\beta)\left(\frac{\epsilon(\alpha)}{z(\alpha)}-\frac{\epsilon(\beta)}{z(\beta)}\right)\left(\frac{1}{\frac{1}{z(\beta)}-\frac{1}{z(\alpha)}}\right)_{\beta}d\beta\\
=&-\frac{1}{\pi i}\int_{0}^{2\pi}f_{\beta}(\beta)\left(\frac{\epsilon(\alpha)}{z(\alpha)}-\frac{\epsilon(\beta)}{z(\beta)}\right)\left(\frac{1}{\frac{1}{z(\beta)}-\frac{1}{z(\alpha)}}\right)d\beta\\
&+\frac{1}{\pi i}\int_{0}^{2\pi}f(\beta)\left(\frac{\epsilon(\beta)}{z(\beta)}\right)_{\beta}\left(\frac{1}{\frac{1}{z(\beta)}-\frac{1}{z(\alpha)}}\right)d\beta.
\end{split}
\end{align}
The last term in \eqref{Hbar linear 2} cancels with \eqref{Hbar linear 1}. Therefore the `linear term' in $\overline{H}f$ is given by

\begin{align}\label{Hbar linear}
-\frac{1}{\pi i}\int_{0}^{2\pi}f_{\beta}(\beta)\frac{\epsilon(\alpha)z(\beta)}{z(\alpha)-z(\beta)}d\beta+\frac{1}{\pi i}\int_{0}^{2\pi}f_{\beta}(\beta)\frac{\epsilon(\beta)z(\alpha)}{z(\alpha)-z(\beta)}d\beta
=z[\epsilon,H]\left(\frac{f_{\alpha}}{z_{\alpha}}\right).
\end{align}
where in the last step we used the fact that $f(0)=f(2\pi)$. The remaining terms in $\overline{H}f$ are the first term in the second line and the two terms in the third line of \eqref{Hbar pre}. The first term in the second line can be written as

\begin{align}\label{Hbar R 1}
\begin{split}
E_{1}(f):=-\frac{1}{\pi i}\int_{0}^{2\pi}\frac{f(\beta)\left(\frac{\epsilon(\alpha)}{z(\alpha)}-\frac{\epsilon(\beta)}{z(\beta)}\right)(z(\alpha)z(\beta))^{2}}{(z(\alpha)-z(\beta))^{2}}\partial_{\beta}\left(\frac{\epsilon(\beta)}{z(\beta)}\right)d\beta.
\end{split}
\end{align}
The first term in the third line of \eqref{Hbar pre} can be written as

\begin{align}\label{Hbar R 2}
E_{2}(f):=-\frac{1}{\pi i}\int_{0}^{2\pi}\frac{f(\beta)\left(\frac{\epsilon(\beta)}{z(\beta)}-\frac{\epsilon(\alpha)}{z(\alpha)}\right)^{2}(z(\alpha)z(\beta))^{2}}{(z(\beta)-z(\alpha))|z(\beta)-z(\alpha)|^{2}}\partial_{\beta}\left(\frac{\epsilon(\beta)}{z(\beta)}\right)d\beta.
\end{align}
The second term in the third line of \eqref{Hbar pre} can be written as

\begin{align}\label{Hbar R 3}
E_{3}(f):=-\frac{1}{\pi i}\int_{0}^{2\pi}\frac{f(\beta)\left(\frac{\epsilon(\beta)}{z(\beta)}-\frac{\epsilon(\alpha)}{z(\alpha)}\right)^{2}(z(\alpha)z(\beta))^{2}}{(z(\beta)-z(\alpha))|z(\beta)-z(\alpha)|^{2}}\partial_{\beta}\left(\frac{1}{z(\beta)}\right)d\beta.
\end{align}
\end{proof}

\begin{remark}
Note that if we measure smallness of quantities by comparison with the static case $z\equiv e^{i\alpha},$ then by Lemma \ref{lem: H Hbar}, $E(f)$ is order of $\ep^2$ smaller than $f.$ This observation will be made precise when we carry out the estimates in Sections~\ref{sec: relations} and~\ref{sec: energy estimates}.
\end{remark}
\subsection{The $\delta$ Equation}\label{subsec: delta eq}
In this section we derive an equation for the small quantity

\Align{\label{delta final def}
\delta:=(I-H)\varepsilon,
}
where

\Align{\label{ep final def}
\varepsilon:=|z|^2-1.
}
Note that in view of our small data assumptions we expect the quantities $\varepsilon$ and $\delta$ to be (linearly) small. Our main goal here is to show that $\delta$ satisfies a constant-coefficient PDE with cubic nonlinearity. This will be accomplished in two steps. In the first step we show that the nonlinear part of $(\partial_t^2+ia\partial_\alpha)\delta$ is cubic. If we then replace the operator $\partial_t^2+ia\partial_\alpha$ by $\partial_t^2+i\pi\partial_\alpha,$ corresponding to the value of $a$ in the static case, we will notice that the resulting error is only quadratic. For this reason, in the second step we perform a change of variables $\beta(t,\alpha)=k^{-1}(t,\alpha)$ such that the nonlinearity in the  equation for $(\partial_t^2+i\pi\partial_\beta)\delta$ has no quadratic part. The first step is achieved in the following proposition.

\begin{proposition}\label{prop: delta temp eq}
The quantities $\delta=(I-H)\varepsilon$ and $\delta_t=\partial_t\delta$ satisfy

\Align{\label{delta temp eq}
(\partial_t^2+ia\partial_\alpha-\pi)\delta=\calN_{1}:=&\frac{\pi}{2}[E(z),H]\frac{\ep_\alpha}{z_\alpha}+\frac{\pi}{2}(I-H)E(\ep)\\
&-2[z_t,H\frac{1}{z_\alpha}+\Hbar\frac{1}{\zbar_\alpha}]\partial_\alpha(z_t\zbar)-\frac{1}{\pi i}\int_0^{2\pi}\left(\frac{z_t(\beta)-z_t(\alpha)}{z(\beta)-z(\alpha)}\right)^2\ep_\beta(\beta)d\beta
}
and

\Align{\label{deltat temp eq}
(\partial_t^2+i\partial_\alpha-\pi)\delta_t=&~\calN_2:= -ia_t\partial_\alpha\delta+\frac{\pi}{2}\left((I-H)\dt E(\ep)-[z_t,H]\frac{\da E(\ep)}{z_\alpha}\right)\\
&-\frac{\pi}{2}\left([\dt E(z),H]\frac{\ep_\alpha}{z_\alpha}+[E(z),H]\partial_t\left(\frac{\ep_\alpha}{z_\alpha}\right)+E(z)[z_t,H]\frac{\partial_\alpha\left(\frac{\ep_\alpha}{z_\alpha}\right)}{z_\alpha}-[z_t,H]\frac{\partial_\alpha\left(E(z)\frac{\ep_\alpha}{z_\alpha}\right)}{z_\alpha}\right)\\
&+\frac{2}{\pi i}\partial_t\int_0^{2\pi}\frac{(z_t(\alpha)-z_t(\beta))\partial_\beta(z_t(\beta)\zbar(\beta))\zbar(\alpha)\zbar(\beta)\left(\frac{\ep(\alpha)}{\zbar(\alpha)}-\frac{\ep(\beta)}{\zbar(\beta)}\right)}{|z(\beta)-z(\alpha)|^2}d\beta\\
&-\frac{1}{\pi i}\partial_t\int_0^{2\pi}\left(\frac{z_t(\beta)-z_t(\alpha)}{z(\beta)-z(\alpha)}\right)^2\ep_\beta(\beta)d\beta,
}
where $E(f)$ is as in Lemma \ref{lem: H Hbar}. Moreover, we can write

\Align{\label{delta temp eq expanded}
[z_t,H\frac{1}{z_\alpha}+\Hbar\frac{1}{\zbar_\alpha}]\partial_\alpha(z_t\zbar)=-\frac{1}{\pi i}\int_0^{2\pi}\frac{(z_t(\alpha)-z_t(\beta))\partial_\beta(z_t(\beta)\zbar(\beta))\zbar(\alpha)\zbar(\beta)\left(\frac{\ep(\alpha)}{\zbar(\alpha)}-\frac{\ep(\beta)}{\zbar(\beta)}\right)}{|z(\beta)-z(\alpha)|^2}d\beta.
}
\end{proposition}
\begin{proof}
We want to apply the last part of Lemma \ref{lem: operator H commutator}. To this end we first compute $(\partial_t^2+ia\partial_\alpha)\ep.$

\Align{
(\partial_t^2+ia\partial_\alpha)\ep=&(z_{tt}+iaz_\alpha)\zbar+(\zbar_{tt}+ia\zbar_\alpha)z+2z_t\zbar_t\\
=&-\frac{\pi}{2}(\zbar(I-\Hbar)z-z(I-H)\zbar)+2\partial_t(z\zbar_t),
}
and since $z\zbar_t$ is holomorphic

\Aligns{
(I-H)(\partial_t^2+ia\partial_\alpha)\ep=\frac{\pi}{2}(I-H)\left(z(I-H)\zbar-\zbar(I-\Hbar)z\right)+2[z_t,H]\frac{\partial_\alpha(z\zbar_t)}{z_\alpha}.
}
Applying Lemma \ref{lem: operator H commutator} we get

\Align{\label{delta eq temp 1}
(\partial_t^2+ia\partial_\alpha)\delta=&\frac{\pi}{2}(I-H)\left(z(I-H)\zbar-\zbar(I-\Hbar)z\right)+\frac{\pi}{2}[(I-\Hbar)z,H]\frac{\ep_\alpha}{z_\alpha}\\
&-2[z_t,H]\frac{\partial_\alpha(z_t\zbar)}{z_\alpha}-\frac{1}{\pi i}\int_0^{2\pi}\left(\frac{z_t(\beta)-z_t(\alpha)}{z(\beta)-z(\alpha)}\right)^2\ep_\beta(\beta)d\beta\\
=&\frac{\pi}{2}(I-H)\left(z(I-H)\zbar-\zbar(I-\Hbar)z\right)+\frac{\pi}{2}[(I-\Hbar)z,H]\frac{\ep_\alpha}{z_\alpha}\\
&-2[z_t,H\frac{1}{z_\alpha}+\Hbar\frac{1}{\zbar_\alpha}]\partial_\alpha(z_t\zbar)-\frac{1}{\pi i}\int_0^{2\pi}\left(\frac{z_t(\beta)-z_t(\alpha)}{z(\beta)-z(\alpha)}\right)^2\ep_\beta(\beta)d\beta.
}
The last two terms already have the right form so we concentrate on the first two. Using Lemma \ref{lem: zbar Hbar commutator} we write

\Aligns{
\zbar(I-\Hbar)z=(I-\Hbar)(z\zbar)-[\zbar,\Hbar]z=\deltabar+\frac{2|\Omega|}{\pi}=\deltabar+2,
}
and hence

\Align{\label{delta eq temp 2}
\frac{\pi}{2}(I-H)\left(z(I-H)\zbar-\zbar(I-\Hbar)z\right)=\frac{\pi}{2}(I-H)(\delta-\deltabar)=\pi\delta-\frac{\pi}{2}(I-H)\deltabar,
}
where to pass to the last equality we have used the fact that $\left(\frac{1}{2}(I-H)\right)^2=\frac{1}{2}(I-H).$ To understand the contributions of $\deltabar$ and $(I-\Hbar)z$ we use Lemma \ref{lem: H Hbar} to replace $\Hbar$ by $H.$ For $\deltabar,$ noting that $H\frac{1}{z}=-\frac{1}{z}$ we get

\Aligns{
\deltabar=&\ep+zH(\zbar-\frac{1}{z})-z[\ep,H]\frac{\ep_\alpha}{z_\alpha}-E(\ep)\\
=&z(I+H)\zbar-z[\ep,H]\frac{\ep_\alpha}{z_\alpha}-E(\ep)\\
=&z(I+H)\zbar-z\ep(I+H)\frac{\ep_\alpha}{z_\alpha}+z(I+H)\frac{\ep\ep_\alpha}{z_\alpha}-E(\ep),
}
which implies

\Align{\label{qua vanish temp 1}
-\frac{\pi}{2}(I-H)\deltabar=&\frac{\pi}{2}(I-H)(z\ep(I+H)(\frac{\ep_\alpha}{z_\alpha}))+\frac{\pi}{2}(I-H)E(\ep)\\
=&\frac{\pi}{4}(I-H)(z\delta(I+H)(\frac{\ep_\alpha}{z_\alpha}))+\frac{\pi}{2}(I-H)E(\ep).
}
Similarly for $(I-\Hbar)z$ we have 

\Aligns{
(I-\Hbar)z=2z-z[\ep,H]1+E(z)=2z-z\delta+E(z).
}
It follows from this and Lemma \ref{lem: z H} that

\Align{\label{qua vanish temp 2}
\frac{\pi}{2}[(I-\Hbar)z,H]\frac{\ep_\alpha}{z_\alpha}=&-\frac{\pi}{2}[z\delta,H]\frac{\ep_\alpha}{z_\alpha}+\frac{\pi}{2}[E(z),H]\frac{\ep_\alpha}{z_\alpha}\\
=&-\frac{\pi}{4}[z\delta,H](I+H)\frac{\epsilon_{\alpha}}{z_{\alpha}}-\frac{\pi}{4}[z\delta,H](I-H)\frac{\epsilon_{\alpha}}{z_{\alpha}}+\frac{\pi}{2}[E(z),H]\frac{\ep_\alpha}{z_\alpha}.
}
By Lemma \ref{lem: fgh commutator}, the second term in \eqref{qua vanish temp 2} can be written as

\begin{align}\label{qua vanish temp 3}
-\frac{\pi}{4}z[\delta,H](I-H)\frac{\epsilon_{\alpha}}{z_{\alpha}}-\frac{\pi}{4}[z,H]\delta(I-H)\left(\frac{\epsilon_{\alpha}}{z_{\alpha}}\right).
\end{align}
The first term in \eqref{qua vanish temp 3} can be written as

\begin{align*}
-\frac{\pi}{4}z[\delta,I+H](I-H)\frac{\epsilon_{\alpha}}{z_{\alpha}}=\frac{\pi}{4}z(I+H)\left((I-H)\epsilon(I-H)\left(\frac{\epsilon_{\alpha}}{z_{\alpha}}\right)\right)=0.
\end{align*}
By $(iii)$ in Lemma \ref{lem: operator H commutator}, the second term in \eqref{qua vanish temp 3} can be written as

\begin{align*}
-\frac{\pi}{4}[z,H]\delta\frac{\delta_{\alpha}}{z_{\alpha}}=0.
\end{align*}
Combining these observations with the fact that $$[z\delta,H](I+H)\frac{\ep_\alpha}{z_\alpha}=(I-H)\left(z\delta(I+H)\frac{\ep_\alpha}{z_\alpha}\right)$$  we get

\Aligns{
-\frac{\pi}{2}(I-H)\deltabar+\frac{\pi}{2}[(I-\Hbar)z,H]\frac{\ep_\alpha}{z_\alpha}=\frac{\pi}{2}(I-H)E(\ep)-\frac{\pi}{2}[E(z),H]\frac{\ep_\alpha}{z_\alpha}.
}
Equation \eqref{delta temp eq} now follows from combining this identity with \eqref{delta eq temp 1} and \eqref{delta eq temp 2}.
Finally, equations \eqref{deltat temp eq} and \eqref{delta temp eq expanded} are direct consequences of Lemmas \ref{lem: operator H commutator}, \ref{lem: partialt f H g}, and \ref{lem: H plus Hbar} and equation \eqref{delta temp eq}.
\end{proof}
By comparing the terms on the right hand sides of the equations \eqref{delta temp eq} and \eqref{deltat temp eq} with their corresponding values in the static case, one can see that the nonlinearity is cubic. This is least clear for the first term involving $a_t$ in the equation for $\delta_t$ so in the following lemma we present a formula for $a_t$ which sheds some light the structure of this term.

\begin{lemma}\label{lem: K*at}
Let $K^*$ denote the formal adjoint of $K:=\Re H=\frac{1}{2}(H+\Hbar)$, i.e.,

\begin{align*}
K^*g(\alpha)=-\Re\frac{\pv}{\pi i}\int_0^{2\pi}\frac{z_\alpha(\alpha)}{|z_\alpha(\alpha)|}\frac{|z_\beta(\beta)|}{z(\beta)-z(\alpha)}g(\beta)d\beta=-\Re\left\{\frac{z_\alpha}{|z_\alpha|}H\frac{|z_\beta|g}{z_\beta}\right\}.
\end{align*}
Then
\Aligns{
(I+K^*)(a_t|z_\alpha|)=\Re\Bigg[\frac{-iz_\alpha}{|z_\alpha|}\Big\{&2[z_t,H]\frac{\zbar_{tt\alpha}}{z_\alpha}+2[z_{tt},H]\frac{\zbar_{t\alpha}}{z_\alpha}-[g^a,H]\frac{\zbar_{t\alpha}}{z_\alpha}\\
&+\frac{1}{\pi i}\int_0^{2\pi}\left(\frac{z_t(\beta)-z_t(\alpha)}{z(\beta)-z(\alpha)}\right)^2\zbar_{t\beta}(\beta)d\beta+\frac{\pi}{2}([z_t,H]\frac{\pa g^{h}}{z_\alpha})\Big\}\Bigg].
}
\end{lemma}
\begin{proof}
Using equation \eqref{z eq}, \eqref{zbart eq} and Lemma \ref{lem: operator H commutator} we have

\Aligns{
(I-H)(ia_t\zbar_\alpha)=&(I-H)(\zbar_{ttt}-ia\zbar_{t\alpha}-\frac{\pi}{2}[z_t,H]\frac{\zbar_\alpha}{z_\alpha})\\
=&[\partial_t^2-ia\partial_\alpha,H]\zbar_t-\frac{\pi}{2}(I-H)([z_t,H]\frac{\zbar_\alpha}{z_\alpha})\\
=&2[z_t,H]\frac{\zbar_{tt\alpha}}{z_\alpha}+[z_{tt},H]\frac{\zbar_{t\alpha}}{z_\alpha}-[iaz_\alpha,H]\frac{\zbar_{t\alpha}}{z_\alpha}\\
&+\frac{1}{\pi i}\int_0^{2\pi}\left(\frac{z_t(\beta)-z_t(\alpha)}{z(\beta)-z(\alpha)}\right)^2\zbar_{t\beta}(\beta)d\beta-\frac{\pi}{2}(I-H)([z_t,H]\frac{\zbar_\alpha}{z_\alpha})\\
=&2[z_t,H]\frac{\zbar_{tt\alpha}}{z_\alpha}+2[z_{tt},H]\frac{\zbar_{t\alpha}}{z_\alpha}-[g^a,H]\frac{\zbar_{t\alpha}}{z_\alpha}\\
&+\frac{1}{\pi i}\int_0^{2\pi}\left(\frac{z_t(\beta)-z_t(\alpha)}{z(\beta)-z(\alpha)}\right)^2\zbar_{t\beta}(\beta)d\beta-\frac{\pi}{2}(I-H)([z_t,H]\frac{\zbar_\alpha}{z_\alpha}).
}
The lemma now follows by multiplying the two sides of this equation by $\frac{-iz_\alpha}{|z_\alpha|}$ and taking real parts and also observing that

\begin{align*}
\frac{\pi}{2}(I-H)\left([z_{t},H]\frac{\zbar_{\alpha}}{z_{\alpha}}\right)=-\frac{\pi}{2}(I-H)\left(\pt(I-H)\zbar\right)=\frac{\pi}{2}[z_{t},H]\frac{\pa g^{h}}{z_{\alpha}}.
\end{align*}
\end{proof} 
For future reference we also record the following representation for $K^*$ which is more amenable to estimates.

\begin{lemma}\label{lem: K*}
For any real valued $2\pi-$periodic function $f$

\Aligns{
K^*f=\frac{1}{\pi|z_\alpha|}\int_0^{2\pi}f(\beta)|z_\beta(\beta)|d\beta-\frac{1}{2|z_\alpha|}(H+\Hbar)(|z_\beta|f)-\Re\left\{\frac{1}{|z_\alpha|}[z_\alpha-iz,H]\frac{|z_\beta|f}{z_\beta}\right\}.
}
\end{lemma}
\begin{proof}
Using the definition $K^*f=-\Re\{\frac{z_\alpha}{|z_\alpha|}H\frac{|z_\beta|}{z_\beta}f\}$ of $K^*$ we have

\Aligns{
-2K^*g=&\frac{z_\alpha}{|z_\alpha|}H\frac{|z_\beta|f}{z_\beta}+\frac{\zbar_\alpha}{|z_\alpha|}\Hbar\frac{|z_\alpha|f}{\zbar_\alpha}\\
=&2\frac{\Re}{|z_\alpha|}\left\{[z_\alpha,H]\frac{|z_\beta|f}{z_\beta}\right\}+\frac{1}{|z_\alpha|}(H+\Hbar)(|z_\beta|f)\\
=&2\Re\left\{\frac{1}{|z_\alpha|}[z_\alpha-iz,H]\frac{|z_\beta|f}{z_\beta}\right\}+\frac{1}{|z_\alpha|}(H+\Hbar)(|z_\alpha|f)+2\Re\left\{\frac{i}{|z_\alpha|}[z,H]\frac{|z_\beta|f}{z_\beta}\right\}.
}
The lemma now follows by noting that

\Aligns{
\frac{i}{|z_\alpha|}[z,H]\frac{|z_\beta|f}{z_\beta}=-\frac{1}{\pi|z_\alpha|}\int_0^{2\pi}|z_\beta(\beta)|f(\beta)d\beta.
}
\end{proof}
We now turn to the left hand side of \eqref{delta temp eq}. As mentioned above, the nonlinear contribution of $a$ to \eqref{delta temp eq} can be seen to be only quadratic, and therefore a change of variables is necessary to retain the cubic structure. More precisely suppose

\Aligns{
k(t,\cdot):\R\to\R
}
is an increasing function such that $k-\alpha$ is $2\pi$ periodic and $k$ is differentiable on $(0,2\pi)$. Let us define

\Aligns{
\zeta(t,\alphap):=z\circ k^{-1}(t,\alphap),\qquad \chi(t,\alphap):=\delta\circ k^{-1}(t,\alphap).
}
Then introducing

\Aligns{
b:=k_t\circ k^{-1},\qquad A= (ak_\alpha)\circ k^{-1}
}
we have

\Aligns{
z_t\circ k^{-1}=(\partial_t+b\partial_\alphap)\zeta,\qquad (az_\alpha)\circ k^{-1}=A\zeta_\alphap.
}
In particular,

\Aligns{
(\delta_{tt}+ia\delta_\alpha-\pi\delta)\circ k^{-1}=((\partial_t+b\partial_\alphap)^2+iA\partial_\alphap-\pi)\chi.
}
We wish to choose the change of variables $k$ in such a way that $b$ consists of quadratic and higher order terms, and $A$ has no linear terms. This is achieved in the following three propositions. First in Propositions \ref{prop: b akalpha} and \ref{prop: averages} we derive the desired representations for $b$ and $A$ under various assumptions on $k.$ Then in Remark \ref{prop: k existence} we explain how to construct $k$ satisfying these assumptions.

\begin{proposition}\label{prop: b akalpha}
Suppose that $z(t,\cdot)$ is a simple closed curve containing origin in its interior for each $t$, and that $k$ is increasing and such that $k-\alpha$ is $2\pi$ periodic and $(I-H)(\zbar e^{ik})=(I-H)(\log \zbar+ik)=0.$ Then

\begin{align*}
&(I-H)k_{t}=-i(I-H)\frac{\zbar_{t}\ep}{\zbar }-i[z_{t},H]\frac{\left(\log(\zbar e^{ik})\right)_{\alpha}}{z_{\alpha}}.\\
&(I-H)(ak_{\alpha})=[z_{t},H]\frac{(\zbar_{t}z)_{\alpha}}{z_{\alpha}}-[z_t,H]\zbar_t\\
&\qquad\qquad\qquad\quad-(I-H)\frac{\zbar_{tt}\ep}{\zbar}+(I-H)\frac{g^{h}\ep}{\zbar}+[z_{tt}-g^{a},H]\frac{\left(\log(\zbar e^{ik})\right)_{\alpha}}{z_{\alpha}}.
\end{align*}
\end{proposition}
\begin{remark}
The conditions on $k$ in the proposition can be understood in the following way. First note that if we fix a value of $\arg(z(t,0))$ (uniquely determined up to an integer multiple of $2\pi$) then $\log \zbar(t,\cdot)$ is an unambiguously defined \emph{continuous} function of the real variable $\alpha$ for each fixed $t.$ Moreover, if $z(t,\cdot)$ is a simple closed curve surrounding the origin, then by the periodicity assumption on $k,$ the curve $\zbar e^{ik}$ does not contain the origin in its interior. Therefore $\log (\zbar e^{ik})$ is defined unambiguously as a complex logarithm, its value agrees with $\log \zbar +ik,$ and for any other choice of $\arg(z(t,0))$ it differs from this by an additive constant, so in particular the condition $(I-H)(\log \zbar +ik)=0$ is independent of this choice. The conditions on $k$ can now be understood as requiring that $\zbar e^{ik}$ be the boundary value of a holomorphic function $F,$ such that $0\notin\{F(z)\vert z\in\Omega\}$ and therefore $\log F$ is also well-defined and holomorphic.
\end{remark}
\begin{proof}[Proof of Proposition \ref{prop: b akalpha}]
Differentiating $(I-H)\left(\log(\zbar e^{ik})\right)=0$ on both sides with respect to $t$, we have:

\begin{align*}
0=(I-H)\left(\frac{\zbar_{t}}{\zbar}+ik_{t}\right)-[z_{t},H]\frac{\left(\log(\zbar e^{ik})\right)_{\alpha}}{z_{\alpha}},
\end{align*}
which implies

\begin{align*}
(I-H)k_{t}=i(I-H)\left(\frac{\zbar_{t}}{\zbar}\right)-i[z_{t},H]\frac{\left(\log(\zbar e^{ik})\right)_{\alpha}}{z_{\alpha}}.
\end{align*}
In view of the fact $(I-H)(\zbar_{t}z)=0$, the first term on the right hand side above can be written as

\begin{align*}
-i(I-H)\frac{\zbar_{t}\ep}{\zbar },
\end{align*}
which gives the first formula in the proposition. For the second formula, we apply the operator $ia\partial_{\alpha}$ to the equation $(I-H)\left(\log(\zbar e^{ik})\right)=0$, to arrive at

\begin{align*}
0=(I-H)\left(\frac{ia\zbar_{\alpha}}{\zbar}-ak_{\alpha}\right)-i[az_{\alpha},H]\frac{\left(\log(\zbar e^{ik})\right)_{\alpha}}{z_{\alpha}}.
\end{align*}
Using \eqref{zbar eq} we get

\begin{align*}
(I-H)(ak_{\alpha})=&(I-H)\left(\frac{\zbar_{tt}}{\zbar}+\pi-\frac{g^{h}}{\zbar}\right)+[z_{tt}+\pi z-g^{a},H]\frac{\left(\log(\zbar e^{ik})\right)_{\alpha}}{z_{\alpha}}\\
=&(I-H)\left(\frac{\zbar_{tt}z}{|z|^2}-\frac{zg^{h}}{|z|^2}\right)+[z_{tt}-g^{a},H]\frac{\left(\log(\zbar e^{ik})\right)_{\alpha}}{z_{\alpha}}\\
=&(I-H)(\zbar_{tt}z)-(I-H)\frac{\zbar_{tt}\ep}{\zbar}+(I-H)\frac{g^{h}\ep}{\zbar}+[z_{tt}-g^{a},H]\frac{\left(\log(\zbar e^{ik})\right)_{\alpha}}{z_{\alpha}}.
\end{align*}
Here we used the fact that $(I-H)(zg^{h})=0$. The first term above can be written as

\begin{align*}
(I-H)\left((\zbar_{t}z)_{t}-z_{t}\zbar_{t}\right)=[z_{t},H]\frac{(\zbar_{t}z)_{\alpha}}{z_{\alpha}}-[z_t,H]\zbar_t,
\end{align*}
and this completes the proof.
\end{proof}
Now suppose we define $k$ in a way that $(I-H)\zbar e^{ik}=0.$ Then in view of Proposition \ref{prop: b akalpha}, to prove that $b$ is quadratic and $ak_\alpha$ contains no linear terms we need to understand the invertibility properties of $\Re (I-H)$ (note that $ak_\alpha$ and $k_t$ are real). In fact, a proper understanding of this is necessary also for controlling various other quantities, such as $\ep$ from our control of $\delta.$ A rigorous quantitative treatment of this in the context of small data problem will be given when we carry out the estimates in Sections~\ref{sec: relations} and~\ref{sec: energy estimates}, but for now we note that if $f$ is real valued, then regarding the last two terms on the second line of \eqref{H Hbar eq} in Lemma \ref{lem: H Hbar} as $O(\ep),$\footnote{Note that in the static case $z(\alpha)=e^{i\alpha}$ the `average' $[z,H]\frac{f}{z}$ is real if $f$ is real-valued. We may therefore treat the imaginary part of this average as perturbative.}

\Aligns{
\Re(I-H)f=(I-\frac{1}{2}(H+\Hbar))f=\left(I+O(\ep)\right)f+\frac{1}{2}\Re\left([z,H]\frac{f}{z}\right).
}
Therefore, roughly speaking, if $\ep$ is small, then we expect $\Re(I-H)$ to be invertible on the space of functions in $\La^{2}$ which satisfy $\Re\AV(f)=0,$ where \footnote{Note that the sign convention is such that $\AV(1)=-1$.}

\Align{\label{average def}
\AV(f):=\frac{1}{2}[z,H]\frac{f}{z}=-\frac{1}{2\pi i}\int_0^{2\pi}\frac{f(\beta)z_\beta(\beta)}{z(\beta)}d\beta.
}
With this observation in mind we compute $\AV$ for $\ep,~b,$ and $ak_\alpha$ in the following proposition.

\begin{proposition}\label{prop: averages}
Suppose that $z(t,\cdot),~t\in I,$ for some interval $I\subseteq \R,$ is a simple closed curve containing the origin in its interior for each $t,$ $|\Omega|=\pi,$ and that $\zbar e^{ik}$ is the boundary value of a holomorphic function $F(t,z)$ such that $\log F(t,z)$ is also holomorphic and $F(t,0)\in\R_{+},~\forall t\in I.$ Then

\Aligns{
&\AV (\ep)=0,\\
&\AV(a k_\alpha)=-\pi+\frac{1}{2\pi i}\int_0^{2\pi}\zbar_tz_{t\beta}d\beta+\frac{1}{2\pi i}\int_0^{2\pi}\frac{(\zbar_{tt}-g^h)\ep z_\beta}{|z|^2}d\beta-\frac{1}{2\pi i}\int_0^{2\pi}\left(\frac{z_{tt}-g^h}{z}\right)\partial_\beta\log Fd\beta,\\
&\Re \AV (k_t)=\frac{\Re}{2\pi }\int_0^{2\pi}\frac{\zbar_t\ep}{|z|^2}z_\beta d\beta-\frac{\Re}{2\pi}\int_0^{2\pi}\log F \left(\frac{zz_{t\beta}-z_tz_\beta}{z^2}\right)d\beta.
} 
\end{proposition}
\begin{proof}
For $\ep$ we have

\Aligns{
\AV(\ep)=\frac{1}{2\pi i}\int_0^{2\pi}\frac{z_\beta(\beta)}{z(\beta)}d\beta+\frac{1}{2}[z,H]\zbar=0
}
by Lemma \ref{lem: zbar Hbar commutator}. To compute $\AV(ak_\alpha)$ we write $\zbar e^{ik}=F(z)$ where $F$ is as in the statement of the lemma. Differentiating with respect to $\alpha$ and multiplying by $ia$ we get

\Aligns{
ia\zbar_\alpha e^{ik}-ak_\alpha F=iaz_\alpha F_z
}
or

\Aligns{
ak_\alpha=\frac{ia\zbar_\alpha}{\zbar}-iaz_\alpha\partial_z\log F.
}
Using equations \eqref{z eq} and \eqref{zbar eq} and the relation $\frac{1}{|z|^2}=1-\frac{\ep}{|z|^2}$ we get

\Aligns{
\frac{ak_\alpha}{z}=&\frac{\pi}{z}+\frac{\zbar_{tt}-g^h}{|z|^2}+\pi \partial_z\log F+\left(\frac{z_{tt}-g^a}{z}\right)\partial_z\log F\\
=&\frac{\pi}{z}+\zbar_{tt}-g^h-\frac{(\zbar_{tt}-g^h)\ep}{|z|^2}+\pi \partial_z\log F+\left(\frac{z_{tt}-g^a}{z}\right)\partial_z\log F.
}
It follows that

\Aligns{
\AV(ak_\alpha)=&-\frac{1}{2\pi i}\int_0^{2\pi}\frac{ak_\beta z_\beta}{z}d\beta\\
=&-\pi+\frac{1}{2\pi i}\int_0^{2\pi}\zbar_tz_{t\beta}d\beta+\frac{1}{2\pi i}\int_0^{2\pi}\frac{(\zbar_{tt}-g^h)\ep z_\beta}{|z|^2}d\beta-\frac{1}{2\pi i}\int_0^{2\pi}\left(\frac{z_{tt}-g^a}{z}\right)\partial_\beta\log Fd\beta.
}
The computation for $\AV(k_t)$ is similar. We differentiate the equation $\zbar e^{ik}=F(t,z)$ with respect to time to get

\Aligns{
\zbar_t e^{ik}+ik_t F=\partial_t (F)
}
or

\Aligns{
k_t=\frac{i\zbar_t}{\zbar}-i\partial_t (\log F).
}
It follows that

\Aligns{
\frac{k_t}{z}z_\alpha=i\zbar_tz_\alpha-\frac{i\zbar_t\ep z_\alpha}{|z|^2}-i\partial_t\left( \frac{z_\alpha\log F}{z}\right)+i\log F \left(\frac{zz_{t\alpha}-z_tz_\alpha}{z^2}\right).
}
Therefore, since $\zbar_t$ and $\log F$ are holomorphic,

\Aligns{
\AV (k_t)=\frac{1}{2\pi }\int_0^{2\pi}\frac{\zbar_t\ep}{|z|^2}z_\beta d\beta-\frac{1}{2\pi}\int_0^{2\pi}\log F \left(\frac{zz_{t\beta}-z_tz_\beta}{z^2}\right)d\beta+\partial_t(i\log F(0,t)).
}
The last equality in the lemma now follows by taking real parts of this expression and noting that $\log F(0,t)\in\R,~\forall t.$
\end{proof}
It follows from the previous two propositions that if $k$ satisfies the conditions in these propositions then $A$ and $b$ have the desired smallness properties. In the following remark we explain how to construct $k$ satisfying the hypotheses of these propositions. Note that the fact that $k$ is increasing will follow from the definition of $k$ below and the smallness assumptions in our problem. See Proposition \ref{prop: initial data} and the proof of Theorem \ref{long time existence theorem} for more details.
\begin{remark}\label{prop: k existence}
Suppose $z(t,\cdot)$ is a simple closed curve containing the origin in its simply connected interior for each $t\in I,$ where $I$ is some time interval. We explain how to construct a function $k:I\times\R\to\R$ such that $k-\alpha$ is periodic and $\zbar e^{ik}$ is the boundary value of a function $F$ such that $F$ and $\log F$ are holomorphic inside $\Omega,$ so in particular $(I-H)(\log \zbar+ik)=0.$ Moreover, we normalize $k$ such that $\log F(t,0)\in \R$ for all $t\in I.$

We fix a choice of the logarithm so that $\log z -i\alpha$ is continuous and $2\pi$ periodic. We let $u$ be the solution of the Dirichlet problem in $\Omega$ with boundary value $\log|z|$ and let $v$ be the harmonic conjugate of $u$ which exists because the domain is simply connected.  It is then easy to see that if $k:=v\vert_{\partial\Omega}+\arg z$ then  $\zbar e^{ik}$ is the boundary value of a holomorphic function $F$ such that $\log F$ is also holomorphic and $k-\alpha$ is $2\pi$ periodic. It remains to show that $k$ may be chosen such that $\log F(t,0)\in\R.$ For this note that $0\notin \{F(t,z)~|~z\in\Omega\}$ and the function $$G(t,z):=F(t,z)e^{-i\arg F(t,0)}$$ is also holomorphic and $0\notin \{G(t,z)~|~z\in\Omega\},$ so $\log G$ is also holomorphic. Moreover, $G$ now satisfies $G(t,0)=|F(t,0)|\in\R$ and the boundary value of $G$ is $\zbar e^{ip}$ where $p(t,\alpha)=k(t,\alpha)-\arg F(t,0).$ In other words, we have found $p$ such that  $(I-H)(\zbar e^{ip})=(I-H)(\log \zbar+ip)=0$ and $\zbar e^{ip}$ is the boundary value of a holomorphic $G$ such that $G(t,0)\in \R$ and $\log G$ is holomorphic.
\end{remark}
\begin{corollary}\label{cor: chi eq}
Suppose $k$ is defined as in Remark~\ref{prop: k existence}. Assume also that $k$ is increasing. Then $\chi:=(I-H)\delta\circ k^{-1}$ and $v:=\delta_t\circ k^{-1}$ satisfy the equations

\begin{equation}\label{delta eq}
(\pt+b\pap)^2\chi+iA\pap\chi-\pi\chi=N_1
\end{equation}
and

\begin{equation}\label{deltat eq}
(\pt+b\pap)^2v+iA\pap v-\pi v=N_2
\end{equation}
where $N_j:=\calN_j\circ k^{-1}$ and $\calN_1$ and $\calN_2$ are as defined in equations \eqref{delta temp eq} and \eqref{deltat temp eq} respectively.
\end{corollary}

\section{Relations between original and transformed quantities}\label{sec: relations}
In the previous section we derived an equation for the transformed quantities $\delta=(I-H)\ep$ and $\chi=\delta\circ k^{-1},$ defined as in \eqref{delta final def}--\eqref{ep final def}, where $k$ was chosen according to Remark~\ref{prop: k existence}. In order to prove energy estimates for this equation it will be important to be able to transfer estimates on $\delta$ to estimates on $\ep$ and conversely. More precisely we define the following quantities

\begin{align}\label{quantities}
\begin{split}
&\zeta=z\circ k^{-1},\quad u=z_{t}\circ k^{-1}, \quad w=z_{tt}\circ k^{-1}\\
&\chi=\delta\circ k^{-1},\quad v=(\partial_{t}\delta)\circ k^{-1}=(\partial_{t}+b\partial_{\alphap})\chi,\\
&\mu=\ep\circ k^{-1},\quad \eta=\zeta_\alpha-i\zeta,\quad \ep=|z|^2-1.
\end{split}
\end{align}

We will use $\CH$ for the Hilbert transform in the variable $\zeta$ and $H$ for the Hilbert transform in $z.$ Our goal in this section will be to obtain algebraic and analytic relations between the `transformed' quantities

\Align{\label{tqs}
\chi,~v,~(\pt+b\pa)v 
}
and the `original' quantities 

\Align{\label{oqs}
\zeta,~u,~w,~\eta,~\mu. 
}
Note that by comparison with the static case (where $z\equiv e^{i\alpha}$ and $k\equiv \alpha$) we expect the `small quantities' to be 
\begin{align}
&\mathrm{original:}\quad \eta,~\mu,~u,~w,\label{osqs}\\
&\mathrm{transformed:}\quad \chi,~v,~(\pt+b\pa)v.\label{tsqs}
\end{align}
The analytic relations in this section will be derived under the following bootstrap assumption, where $\ell\geq 5$ is a fixed integer and $M \leq M_0 <\infty$ are small numbers to be fixed:

\Align{\label{bootstrap 1}
\begin{cases}\sum_{k\leq \ell}\left(\|\partial_\alphap^kw\|_{L_\alphap^2}+\|\partial_\alphap^ku\|_{L_\alphap^2}+\|\partial_\alphap^k \eta\|_{L_\alphap^2}\right)\leq M<M_0,\\
|\zeta(t,\alphap)|^2\geq\frac{1}{4}
\end{cases}
}
for $t\in I$ where $I$ is some interval containing $0$.

We start with the following estimates for $\zeta$ which will be used in many other computations.

\begin{proposition}\label{prop: zeta bounds}
\begin{enumerate}
\item 
There exists $\alphap_0=\alphap_0(t)$ such that

\Aligns{
&\|\zeta(\cdot)-e^{i(\alphap_0+\cdot)}\|_{L^{\infty}_{\alphap}\cap L^{2}_{\alphap}}\leq C\|\eta\|_{\Lap^{2}}\leq C M,\\
&\|\mu\|_{L^{\infty}_{\alphap}\cap L^{2}_{\alphap}}\leq C\|\eta\|_{\Lap^{2}}\leq C M.
}

\item
If $M_0$ in \eqref{bootstrap 1} is sufficiently small then there exist non-zero constants $c$ and $C$ such that for all $j\leq \ell$ and $k\leq \ell-1$

\Aligns{
c\leq\|\partial_\alphap^j \zeta_\alphap\|_{L_\alphap^2},\|\partial_\alphap^k \zeta_\alphap\|_{L_\alphap^\infty}\leq C.
}
\item If $M_0$ in \eqref{bootstrap 1} is sufficiently small then for all $k\leq \ell$

\begin{align*}
\sum_{j\leq k}\|\partial_\alphap^j\mu_\alphap\|_{L_\alphap^2}\leq C\sum_{j\leq k}\|\partial_\alphap^j\eta\|_{L_\alphap^2},\qquad \sum_{j\leq k-1}\|\partial_\alphap^j\mu_\alphap\|_{L_\alphap^{\infty}}\leq C\sum_{j\leq k-1}\|\partial_\alphap^j\eta\|_{L_\alphap^{\infty}}.
\end{align*}
\end{enumerate}
\end{proposition}
\begin{proof}
\begin{enumerate}
\item Note that since $0\in\Omega(0)$ and $|\zeta(t,\alphap)|\geq\frac{1}{2}$ for all $\alphap,$ $0\in\Omega(t)$ as long as the bootstrap assumptions hold. Direct differentiation implies that $f(\alpha):=e^{-i\alpha}\zeta(\alpha)$ satisfies $\|f_\alphap\|_{L_\alphap^2}\leq \|\eta\|_{\Lap^2}.$ Moreover since the area of $\Omega$ is a preserved by the flow, there exists $\gamma\in [0,2\pi]$ such that $|f(\gamma)|=1,$ or equivalently $f(\gamma)=e^{i\alphap_0}$ for some $\alphap_0.$ Now for any other $\alphap\in[0,2\pi]$ we have

\Aligns{
|f(\alphap)-e^{i\alphap_0}|\leq \int_{\min\{\gamma,\alphap\}}^{\max\{\gamma,\alphap\}}|f_\betap(\betap)|d\betap\leq\sqrt{2\pi}\|\eta\|_{\Lap^{2}},
}
proving the first inequality. The second inequality is a direct consequence of the first and the definition of $\mu.$
\item From the definition of $\eta$ we have

\Aligns{
\partial_\alphap^j\zeta_\alphap=i\partial_\alphap^j\zeta
+\partial_\alphap^j\eta.
}
The desired estimates now follow from the previous part by induction on $j$ and use of the Sobolev inequality $\|\partial_\alphap^k\eta\|_{L_\alphap^\infty}\leq C(\|\partial_\alphap^k\eta\|_{L_\alphap^2}+\|\partial_\alphap^{k+1}\eta\|_{L_\alphap^2}).$
\item This estimate is a direct consequence of the previous part and the relation $\mu_\alphap=\zetabar\eta+\zeta\etabar.$
\end{enumerate}
\end{proof}

A corollary of Proposition \ref{prop: zeta bounds} is the following result which allows us to use the tools from Section \ref{sec: analysis} in the remainder of this section and in the next section. 

\begin{corollary}\label{cor: inverse Lip}
Under the bootstrap assumption \eqref{bootstrap 1}, and if $M_0$ is sufficiently small,

\begin{align*}
|\zeta(\alphap)-\zeta(\betap)|\geq \frac{1}{10}|e^{i\alphap}-e^{i\betap}|. 
\end{align*}
\end{corollary}
\begin{proof}
Since $\zeta(\alpha\pm2\pi)=\zeta(\alpha)$ and $e^{i(\alpha\pm2\pi)}=e^{i\alpha}$ it suffices to prove the corollary for $\alpha$ and $\beta$ such that $|\alpha-\beta|\leq\frac{3\pi}{2}.$ Since for this range of $\alpha$ and $\beta$ we have $|e^{i\alpha}-e^{i\beta}|\gtrsim |\alpha-\beta|,$

\begin{align*}
\abs{\zeta(\alphap)-\zeta(\betap})&=\abs{\int_{\betap}^{\alphap}\zeta_{\alphap'}(\alpha'')d\alpha''}=\abs{i\int_{\betap}^{\alphap}\zeta(\alpha'')d\alpha''
+\int_{\beta}^{\alpha}O(M)d\alpha'}\\
&=\abs{e^{i\alphap_0}\int_{\betap}^{\alphap}ie^{i\alpha''}d\alpha''
+\int_{\betap}^{\alphap}O(M)d\alphap'}\geq\abs{e^{i\alphap_0}
\left(e^{i\alphap}-e^{i\betap}\right)}-\abs{\int_{\betap}^{\alphap}O(M)d\alpha''}\\
&\geq \frac{1}{10}|e^{i\alpha}-e^{i\beta}|.
\end{align*}
if $M$ is sufficiently small.
\end{proof}

As immediate important consequences of Proposition \ref{prop: zeta bounds} and Corollary \ref{cor: inverse Lip} we record the following two corollaries.

\begin{corollary}\label{cor: E}
If $M_0$ in \eqref{bootstrap 1} is sufficiently small, then for any $2\leq k\leq \ell,$ any $2\pi-$periodic function $f,$ and with $E$ defined as in Lemma \ref{lem: H Hbar}

\Aligns{
\sum_{j\leq k}\|\pap^j E(f)\|_{\Lap^2}\leq C\left(\sum_{j\leq k}\|\pap^{j}\mu\|_{\Lap^{2}}^{2}\right)\left(\sum_{j\leq k}\|\pap^jf\|_{\Lap^2}\right)\leq C M^2 \sum_{j\leq k}\|\pap^jf\|_{\Lap^2}.
}
\end{corollary}
\begin{proof}
From the definition of $E$ this is a direct corollary of Proposition \ref{prop: C1} and Lemma \ref{lem: kernel der}. Notice that by Proposition \ref{prop: zeta bounds} $\|\pap^{\ell+1}\mu\|_{\Lap^2}\leq CM$ and $1\lesssim\|\pap^{\ell+1}\zeta\|_{\Lap^2}\lesssim1$ under the bootstrap assumptions \eqref{bootstrap 1}.
\end{proof}

\begin{corollary}\label{cor: mu chi}
If $M_0$ in \eqref{bootstrap 1} is sufficiently small, then for any $2\leq k\leq \ell,$

\Aligns{
\sum_{j\leq k+1}\|\pap^j\chi\|_{\Lap^2}\leq C\sum_{j\leq k}\|\pap^j\eta\|_{\Lap^2}.}

\end{corollary}
\begin{proof}
Since $\chi=(I-\CH)\mu,$ this follows from Lemma \ref{lem: I-H L2} and Proposition \ref{prop: zeta bounds}.
\end{proof}

Next we record the following algebraic relations.

\begin{proposition}\label{algebraic relation}
With the same notation as \eqref{quantities}

\begin{align}
&\partial_{\alphap}\chi=\left(I-\zeta_{\alphap}\calH\frac{1}{\zeta_{\alphap}}\right)\mu_\alphap\label{chi mu}\\
&v=(\partial_{t}+b\partial_{\alphap})\chi=2u\zetab-\left(\calH
+\overline{\calH}\right)u\zetab-[u,\calH]\frac{\mu_\alphap}{\zeta_{\alphap}},\label{u v}\\
&(\partial_{t}+b\partial_{\alphap})v=2w\zetab+2u\ubar
-(\calH+\overline{\calH})(w\zetab+u\ubar)-[u,\calH]\frac{(u\zetab)_{\alphap}}{\zeta_{\alphap}}-[\ubar,\CHbar]\frac{(u\zetab)_{\alphap}}{\zetab_{\alphap}}\nonumber\\
&\quad\quad\quad\quad\quad\quad-[w,\calH]\frac{\mu_{\alphap}}{\zeta_{\alphap}}-[u,\calH]\frac{(u\zetabar+\ubar\zeta)_{\alphap}}{\zeta_{\alphap}}\nonumber\\
&\quad\quad\quad\quad\quad\quad-\frac{1}{\pi i}\int_0^{2\pi}\left(\frac{u(t,\alphap)-u(t,\betap)}{\zeta(t,\betap)-\zeta(t,\alphap)}\right)^2\mu_\betap(t,\betap)d\betap.\label{vt w}
\end{align}
\end{proposition}
\begin{proof}
First
\begin{align*}
\partial_{\alphap}\chi=&\left(\zeta\zetab-1\right)_{\alphap}
-\partial_{\alphap}\left(\calH(\zeta\zetab-1)\right)\\
=&(I-\calH)\left(\zeta\zetab-1\right)_{\alphap}-\left[\zeta_{\alphap},\calH\right]\frac{\left(\zeta\zetab-1\right)_{\alphap}}{\zeta_{\alphap}}\\
=&\left(I-\zeta_{\alphap}\calH\frac{1}{\zeta_{\alphap}}\right)\left((\zeta\zetab-1)_{\alphap}\right).
\end{align*}
Composing with $k^{-1}$ we get the first identity. Similarly
\begin{align*}
\partial_{t}\delta=&\partial_{t}(I-H)(z\zbar-1)=(z\zbar-1)_{t}-\partial_{t}\left(H(z\zbar-1)\right)\\
=&(z\zbar-1)_{t}-H\left((z\zbar-1)_{t}\right)-[z_{t},H]\frac{(z\zbar-1)_{\alpha}}{z_{\alpha}}\\
=&z_{t}\zbar-H(z_{t}\zbar)-[z_{t},H]\frac{(z\zbar-1)_{\alpha}}{z_{\alpha}}\\
=&2z_{t}\zbar-\left(H+\overline{H}\right)(z_{t}\zbar)-[z_{t},H]\frac{(z\zbar-1)_{\alpha}}{z_{\alpha}}.
\end{align*}
To derive the third formula, we need to compute a time derivative as follows:

\begin{align*}
\partial_{t}\left([z_{t},H]\frac{f_{\alpha}}{z_{\alpha}}\right)=[z_{tt},H]\left(\frac{f_{\alpha}}{z_{\alpha}}\right)+[z_{t},H]\frac{f_{t\alpha}}{z_{\alpha}}+\frac{1}{\pi i}\int_{0}^{2\pi}\left(\frac{z_{t}(\beta)-z_{t}(\alpha)}{z(\beta)-z(\alpha)}\right)^{2}f_{\beta}(\beta)d\beta.
\end{align*}
Therefore

\begin{align*}
\partial^{2}_{t}\delta=&2z_{tt}\zbar+2z_{t}\zbar_{t}-(H+\Hbar)(z_{tt}\zbar+z_{t}\zbar_{t})\\
&-[z_{t},H]\frac{(z_{t}\zbar)_{\alpha}}{z_{\alpha}}-[\zbar_{t},\Hbar]\frac{(z_{t}\zbar)_{\alpha}}{\zbar_{\alpha}}\\
&-[z_{t},H]\frac{(z_{t}\zbar+z\zbar_{t})_{\alpha}}{z_{\alpha}}-[z_{tt},H]\frac{(z\zbar)_{\alpha}}{z_{\alpha}}\\
&-\frac{1}{\pi i}\int_{0}^{2\pi}\left(\frac{z_{t}(\beta)-z_{t}(\alpha)}{z(\beta)-z(\alpha)}\right)^{2}(z\zbar)_{\beta}d\beta.
\end{align*}
The third formula follows by precomposing with $k^{-1}$.
\end{proof}

The estimates for $u$ and $w$ are given in the following proposition.

\begin{proposition}\label{prop: u v}
If $M_0$ in \eqref{bootstrap 1} is sufficiently small, then there are non-zero constants $c$ and $C$ such that for any $2\leq k \leq \ell$

\begin{align}
&c \sum_{j\leq k}\|\partial_\alphap^j v\|_{L^{2}_\alphap}^2 \leq \sum_{j\leq k}\|\partial_\alphap^j u\|_{L^{2}_\alphap}^2\leq C \sum_{j\leq k}\|\partial_\alphap^j v\|_{L^{2}_\alphap}^2, \label{u v 1}\\
&\left|\sum_{j_\leq k}\|\pap^j(\partial_t+b\partial_\alphap)v\|_{\Lap^2}-\sum_{j\leq k}\|\pa^jw\|_{\Lap^2}\right|\leq C\sum_{j\leq k}\|\pap^ju\|_{\Lap^2}^2.\label{u v 2}
\end{align}
In particular
\Aligns{
c\sum_{j\leq k}(\|\pap^ju\|_{\La^2}+\|\pa^jw\|_{\Lap^2})\leq C\sum_{j\leq k}(\|\pap^jv\|_{\La^2}+\|\pap^j(\partial_t+b\partial_\alphap)v\|_{\Lap^2})\leq C\sum_{j\leq k}(\|\pap^ju\|_{\Lap^2}+\|\pap^jw\|_{\Lap^2}).
}
\end{proposition}

\begin{proof}
First we prove \eqref{u v 1}. We begin by rewriting \eqref{u v} as

\Align{\label{u v temp 1}
\begin{cases}
2u= \frac{v}{\zetabar} +\frac{\zeta}{\zetabar}[\mu,\CH]\frac{\pap(u\zetabar)}{\zeta_\alphap}+\frac{1}{\zetabar}[u,\CH]\frac{\mu_\alphap}{\zeta_\alphap}+\frac{1}{\zetabar}E(u\zetabar)-\frac{2}{\zetabar}\AV(u\zetabar)\\
 v=2u\zetabar-\zeta[\mu,\CH]\frac{\pap(u\zetabar)}{\zeta_\alphap}-[u,\CH]\frac{\mu_\alphap}{\zeta_\alphap}-E(u\zetabar)+2\AV(u\zetabar)
\end{cases}.
}
Now we estimate the terms above in $H_\alphap^k.$ First note that by Proposition \ref{prop: zeta bounds} and Sobolev

\Aligns{
\sum_{j\leq k}\|\pap^j\frac{v}{\zetabar}\|_{\Lap^2}\lesssim \sum_{j\leq k}\|\pap^jv\|_{\Lap^2},\qquad \sum_{j\leq k}\|\pap^j(u\zetabar)\|_{\Lap^2}\lesssim \sum_{j\leq k}\|\pap^ju\|_{\Lap^2},
}
so it suffices to bound the contribution of all other terms on the right hand sides by $M\sum_{j\leq k}\|\pap^ju\|_{\Lap^2}.$ The contribution of $E(u\zetabar)$ is already handled in Corollary \ref{cor: E}. For $\AV(u\zetabar)$ note that since $u$ is anti-holomorphic inside $\Omega$
\Aligns{
\int_0^{2\pi}\frac{u\zetabar\zeta_\alphap}{\zeta}d\alphap=\int_0^{2\pi}\frac{u\mu_\alphap}{\zeta}d\alphap-\int_0^{2\pi}u\zetabar_\alphap d\alphap=\int_0^{2\pi}\frac{u\mu_\alphap}{\zeta}d\alphap,
}
which is bounded by $M\|u\|_{L^\infty}$ (note that $\AV(u\zetabar)$ is a constant as a function of $\alpha$). The contribution of the other terms is handled by Lemma \ref{lem: gHf}.

For \eqref{u v 2} we use \eqref{vt w} and a similar argument as for the proof of \eqref{u v 1} to bound the contributions of the last integral in \eqref{vt w}, $|u|^2,$ $[u,\CH]\frac{(u\zetabar)_\alpha}{\zeta_\alpha},$ and $[\ubar,\CHbar]\frac{(u\zetabar)_\alpha}{\zetabar_\alpha}$ by $\sum_{j\leq k}\|\pa^ju\|_{\La^2}^2.$ The contribution of $[w,\CH]\frac{\mu_\alpha}{\zeta_\alpha}$ is bounded by $M\sum_{j\leq k}\|\pa^jw\|_{\La^2},$ by Lemma \ref{lem: gHf}. Finally, applying the identity

\Aligns{
(\cH+\CHbar)f=-2\AV(f)+E(f)+\zeta[\mu,\CH]\frac{f_\alpha}{\zeta_\alpha}
}
to $f=w\zetabar$ and $f=|u|^2$ and using similar arguments as above we can estimate the contribution of $(\CH+\CHbar)(w\zetabar+|u|^2)$ by

\Aligns{
M\sum_{j\leq k} \|\pap^jw\|_{\Lap^2}+\sum_{j\leq k}\|\pa^j u\|_{\Lap^2}^2.
}
Here to estimate $\AV(w\zetabar)$ we have noted that 
\Aligns{
\int_0^{2\pi}\frac{w\zetabar\zeta_\alphap}{\zeta}d\alphap=\int_0^{2\pi}\frac{w\mu_\alphap}{\zeta}d\alphap-\int_0^{2\pi}w\zetabar_\alphap d\alphap
=\int_0^{2\pi}\frac{w\mu_\alphap}{\zeta}d\alphap-\int_0^{2\pi}u_\alphap\ubar \,d\alphap,
}
where for the last equality we have written $z_t=F(t,\zbar)$ for some anti-holomorphic function $F$ to get $z_{tt}=F_t+F_\zbar\zbar_t=F_t+\frac{z_{t\alpha}\zbar_t}{\zbar_\alpha}.$ The desired estimates now follow from the bootstrap assumptions \eqref{bootstrap 1} if $M_0$ is sufficiently small. Note that the term $\int_{0}^{2\pi}u_{\alphap}\ubar d\alphap$ is bounded by $\sum_{j\leq k}\|\pap^{j}u\|_{\Lap^{2}}^{2}$ because it does not depend on $\alphap$. Therefore it vanishes when the spatial derivatives hit it.
\end{proof}

Our next goal is to estimate $\eta$ and its higher derivatives. To this end we rearrange equation \eqref{z eq} to get

\Align{\label{eta eq}
\pi \eta=iw-(A-\pi)\zeta_\alphap-i g^a\circ k^{-1}.
}
To use this equation we first need to estimate $A-\pi.$ This is accomplished in the next proposition.

\begin{proposition}\label{prop: A}
If $M_0$ in \eqref{bootstrap 1} is sufficiently small then for any $2\leq k\leq \ell$

\Aligns{
\sum_{j\leq k}\|\pap^j(A-\pi)\|_{\Lap^2}\leq C\left( \sum_{j\leq k}\|\pap^j u\|_{\Lap^2}^2+\sum_{j\leq k}\|\pap^j\eta\|_{\Lap^2}^2+\sum_{j\leq k}\|\pap^jw\|_{\La^2}\sum_{j\leq k}\|\pap^j\eta\|_{\Lap^2}\right).
}
\end{proposition}
\begin{proof}
Using Propositions \ref{prop: b akalpha} and \ref{prop: averages} and the fact that $\partial_\alphap \log(\zetabar e^{i\alphap})=\frac{\etabar}{\zetabar}$ we write

\Align{\label{A temp 1}
\begin{cases}
(I-\CH)(A-\pi)=[u,\CH]\frac{(\ubar\zeta)_\alphap}{\zeta_\alphap}-[u,\CH]\ubar-(I-\CH)\frac{\wbar\mu}{\zetabar}+(I-\CH)\frac{\mu g^h\circ k^{-1}}{\zetabar}+[w-g^a\circ k^{-1},\CH]\frac{\etabar}{\zeta_{\alphap}\zetabar}\\\\
\AV(A-\pi)=\frac{1}{2\pi i}\int_0^{2\pi}\ubar u_\betap d\betap+\frac{1}{2\pi i}\int_0^{2\pi}\frac{(\wbar-g^h\circ k^{-1})\mu\zeta_\betap}{|\zeta|^2}d\betap-\frac{1}{2\pi i}\int_0^{2\pi}\frac{\eta(w-g^h\circ k^{-1})}{|\zeta|^2}d\betap
\end{cases}.
}
Now since $A$ is real

\Align{\label{A temp 2}
A-\pi&=\Re(I-\CH)(A-\pi)+\frac{1}{2}(\CH+\CHbar)(A-\pi)\\
&=\Re(I-\CH)(A-\pi)+\frac{1}{2}\left(\zeta[\mu,\CH]\frac{A_\alphap}{\zeta_\alphap}+E(A-\pi)-2\AV(A-\pi)\right).
}
Moreover, by Lemma \ref{lem: H Hbar} we can write

\Aligns{
g^a\circ k^{-1}=\overline{g^h\circ k^{-1}}=\frac{\pi}{2}(\CH+\CHbar)\zeta=\frac{\pi}{2}\zeta\chi+\frac{\pi}{2}E(\zeta).
}
Notice that by Corollary \ref{cor: E} and Proposition \ref{prop: zeta bounds}

\Aligns{
\sum_{j\leq k}\|\pap^jE(\zeta)\|_{\Lap^2}\leq C\sum_{j\leq k}\|\pap^j\eta\|_{\Lap^2}^2,
}
and by Proposition \ref{prop: zeta bounds} if $M_0$ in \eqref{bootstrap 1} is sufficiently small
\Aligns{
\sum_{j\leq k}\|\pap^j(g^a\circ k^{-1})\|_{\Lap^2}\leq C \sum_{j\leq k}\|\pap^j\eta\|_{\Lap^2}.
}
It follows from this, Proposition \ref{prop: zeta bounds}, Lemma \ref{lem: gHf}, and \eqref{A temp 1} that

\Aligns{
\sum_{j\leq k}\|\pap^j\Re(I-\CH)(A-\pi)\|_{\Lap^2}&+\|\AV(A-\pi)\|_{\Lap^2}\\
&\leq C\left(\sum_{j\leq k}\|\pap^j u\|_{\Lap^2}^2+\sum_{j\leq k}\|\pap^j\eta\|_{\Lap^2}^2+\sum_{j\leq k}\|\pap^jw\|_{\Lap^2}\sum_{j\leq k}\|\pap^j\eta\|_{\Lap^2}\right).
}
Similarly the bootstrap assumptions give

\Aligns{
\sum_{j\leq k}\big(\|\pap^j(\zeta[\mu,\CH]\frac{A_\alphap}{\zeta_\alphap})\|_{\Lap^2}+\|\pap^j E(A-\pi)\|_{\Lap^2}\big)\leq CM \sum_{j\leq k}\|\pap^j(A-\pi)\|_{\Lap^2}.
}
Combining these estimates with \eqref{A temp 2} we arrive at the desired conclusion if $M$ is sufficiently small.
\end{proof}
We now go back to the analysis of equation \eqref{eta eq}. As observed in the proof of Proposition \ref{prop: A} we can write $g^a\circ k^{-1}=\frac{\pi}{2}(\zeta\chi+E(\zeta)).$ This shows that equation \eqref{eta eq} by itself is not enough to obtain estimates on $\eta$ and its higher derivatives in terms of $(\pt+b\pap)\chi$ and $(\pt+b\pap)v$ and their higher derivatives. To get such estimates we will also need to use the original equation \eqref{delta eq}, which in turn requires estimates on the right hand side of \eqref{delta eq}. These estimates are also of independent interest in proving energy estimates, so before stating the final estimates for $\eta$ we state the following estimates on the right hand sides of the equations \eqref{delta eq} and \eqref{deltat eq}.
\begin{proposition}\label{prop: NL}
Let $N_1$ and $N_2$ be as in Corollary \ref{cor: chi eq}. Then if $M_0$ in \eqref{bootstrap 1} is sufficiently small, for any $3\leq k \leq \ell$

\Align{\label{N1 N2 bounds}
\sum_{j\leq k}\left( \|\pap^j N_1\|_{\Lap^2}+\|\pap^j N_2\|_{\Lap^2} \right) \lesssim \sum_{j\leq k} \left( \|\pap^j\eta\|_{\Lap^2} + \|\pap^j u\|_{\Lap^2} + \| \pap^j w\|_{\Lap^2}\right)^3
}
\end{proposition}

\begin{proof}
We begin with $N_1.$ Using Lemmas \ref{lem: gHf} and \ref{lem: I-H L2} and Corollary \ref{cor: E} we can bound the contributions of the first two terms on the right hand side of \eqref{delta temp eq} by the right hand side of \eqref{N1 N2 bounds}. Similarly, in view of equation \eqref{delta temp eq expanded}, the contributions of the last two terms on the right hand side of \eqref{delta temp eq} can be bounded by the right hand side of \eqref{N1 N2 bounds} by using Lemma \ref{lem: kernel der} and Propositions \ref{prop: C1} and \ref{prop: C2}. This completes the estimates for $N_1.$ The contribution of $N_2$ can be treated in a similar way. Indeed except for the first term on the right hand side of \eqref{deltat temp eq} all other terms can be estimated by similar arguments as above using Propositions \ref{prop: C1} and \ref{prop: C2} and Lemmas \ref{lem: kernel der}, \ref{lem: gHf}, and \ref{lem: I-H L2}. Here we will also use the observations that

\Aligns{
\pt \left(\frac{A(t,\alpha)-A(t,\beta)}{z(t,\beta)-z(t,\alpha)}\right)=\frac{A_t(t,\alpha)-A_t(t,\beta)}{z(t,\beta)-z(t,\alpha)}-\frac{(A(t,\alpha)-A(t,\beta))(z_t(t,\beta)-z_t(t,\alpha))}{(z(t,\beta)-z(t,\alpha))^2}
}
and $\pt\ep=z_t\zbar+z\zbar_t.$ We omit the details. Finally the term $a_t$ is treated independently in the proof of Lemma \ref{lem: at} below.\footnote{We note that the treatment in Lemma \ref{lem: at} does not rely on the validity of Proposition \ref{prop: NL}. In fact we only use the estimates for $N_1$ in the proof of Proposition \ref{prop: energy bounds temp} below and the proof of the estimates for $a_t$ in Lemma \ref{lem: at} are even independent of this proposition.} 
\end{proof}

Using equation \eqref{eta eq}, we can now combine Propositions \ref{prop: u v}, \ref{prop: A}, and \ref{prop: NL} to prove the following proposition.

\begin{proposition}\label{prop: energy bounds temp}
If $M_0$ in \eqref{bootstrap 1} is sufficiently small then for any $3\leq k\leq\ell,$

\Aligns{
\sum_{j_\leq k}&\left(\|\pap^jw\|_{\Lap^2}+\|\pap^ju\|_{\Lap^2}+\|\pap^j\eta\|_{\Lap^2}\right)\leq C\sum_{j\leq k}\left(\|\pap^j(\partial_t+b\pap)v\|_{\Lap^2}+\|\pap^j(\partial_t+b\pap)\chi\|_{\Lap^2}\right).
}
\end{proposition}
\begin{proof}
In view of Proposition \ref{prop: u v} we only need to prove this estimate for $\eta.$ From Proposition \ref{prop: A} we know that $(A-\pi)$ is quadratic. In equation \eqref{eta eq}, using Lemma \ref{lem: H Hbar} we can write the term $g^{a}\circ k^{-1}$ as

\begin{align*}
g^{a}\circ k^{-1}=\frac{\pi}{2}(I+\overline{\calH})\zeta=\frac{\pi}{2}(\calH+\overline{\calH})\zeta=\frac{\pi}{2}\zeta\chi+E(\zeta)
\end{align*}
and equation \eqref{eta eq} can be written as

\begin{align}\label{eta eq rough}
\pi\eta=iw-\frac{\pi i}{2}\zeta\chi-\frac{\pi i}{2}E(\zeta)-(A-\pi)\zeta_{\alphap}
\end{align}
The arguments for estimating $\eta$ itself and its derivatives are different. For $\eta$ we use equation \eqref{delta eq} and the definition of $v:=(\pt+b\pap) \chi$ to get

\Align{\label{eqs chi mu}
\begin{cases}
\pi\chi_\alphap+i\pi\chi=i\left(\partial_{t}+b\partial_{\alphap}\right)v-(A-\pi)\pap\chi-iN_1\\
N_1:=(\pt+b\pap)^2\chi+iA\pap\chi-\pi \chi
\end{cases}.
}
For higher derivatives of $\eta$ we instead use the following system which is obtained by differentiating \eqref{eta eq rough} and the second equation in \eqref{eqs chi mu}

\Align{\label{eta deriv}
\begin{cases}
\pi\partial_{\alphap}^{\ell}\eta=i\partial_{\alphap}^{\ell}w-\frac{\pi i}{2}\partial_{\alphap}^{\ell-1}\pap(\zeta\chi)-\frac{\pi i}{2}\partial_{\alphap}^{\ell}E(\zeta)-\partial_{\alphap}^{\ell}
\left((A-\pi)\zeta_{\alphap}\right),\quad \ell\geq 1\\
\pap(\zeta\chi)=\zeta(\chi_\alphap+i\chi)+\eta\chi=\zeta(\frac{i}{\pi}(\pt+b\pap)v-\frac{1}{\pi}(A-\pi)\chi_\alphap-\frac{i}{\pi}N_1)+\eta\chi
\end{cases}.
}
We start with the estimates for $\eta$ itself. In view of equation \eqref{eta eq rough}, we need to obtain an estimate for $\zeta\chi$. On the other hand, by Propositions \ref{prop: zeta bounds}, \ref{prop: A}, \ref{prop: NL} and Corollary \ref{cor: mu chi}, the second equation in \eqref{eta deriv} gives us an estimate for $\pap(\zeta\chi)$:

\begin{align}\label{esti: deri zetachi final}
\|\pap(\zeta\chi)\|_{\Lap^{2}}\lesssim\|(\pt+b\pap)v\|_{\Lap^{2}}+M\|\eta\|_{\Lap^{2}}+M^{2}\sum_{j\leq 3}\left(\|\pap^ju\|_{\Lap^2}+\|\pap^jw\|_{\Lap^2}+\|\pap^j\eta\|_{\Lap^2}\right).
\end{align}
In order to obtain the $L^{2}$-estimate for $\zeta\chi$, we still need to 
know the value of $\zeta\chi$ at least at one point. Note that by Proposition \ref{prop: bdry formula CM}

\begin{align*}
\int_{0}^{2\pi}\zeta\chi\cdot \frac{\zeta_{\alphap}}{\zeta}d\alphap=0.
\end{align*}
Therefore 

\begin{align*}
\int_{0}^{2\pi}\zeta\chi d\alphap=i\int_{0}^{2\pi}\chi\eta d\alphap,
\end{align*}
from which we have

\begin{align*}
\left|\int_{0}^{2\pi}\Re(\zeta\chi)d\alphap\right|,\quad \left|\int_{0}^{2\pi}\Im(\zeta\chi)d\alphap\right|\lesssim M\|\eta\|_{\Lap^{2}}.
\end{align*}
%
These together with \eqref{esti: deri zetachi final} imply that

\begin{align}\label{esti: zetachi}
\|\zeta\chi\|_{\Lap^{2}}\lesssim\|(\pt+b\pap)v\|_{\Lap^{2}}+M\|\eta\|_{\Lap^{2}}+M^{2}\sum_{j\leq 3}\left(\|\pa^ju\|_{\Lap^2}+\|\pap^jw\|_{\La^2}+\|\pap^j\eta\|_{\Lap^2}\right).
\end{align}
Substituting this into \eqref{eta eq rough}, using Corollary~\ref{cor: E} and Proposition~\ref{prop: A}, and taking $M>0$ sufficiently small we get

\Align{\label{esti: eta}
\|\eta\|_{\Lap^{2}}\lesssim&\|(\pt+b\pap)v\|_{\Lap^{2}}+\|(\pt+b\pap)\chi\|_{\Lap^{2}}+M^{2}\sum_{j\leq 3}\|\pap^{j}\eta\|_{\Lap^{2}}\\
&+M^2\sum_{j\leq 3}\left(\|\pap^j(\pt+b\pap)\chi\|_{\Lap^2}+\|\pap^j(\pt+b\pap)v\|_{\Lap^2}\right).
}
Finally applying Propositions \ref{prop: A} and \ref{prop: NL} and Corollary \ref{cor: E} to equation \eqref{eta deriv}, for and $3\leq k\leq \ell,$ we get the bound

\Aligns{
\sum_{1\leq j\leq k}\|\pap^{j}\eta\|_{\Lap^2}^2\lesssim \sum_{j\leq k}\left(\|\pap^j(\pt+b\pap)\chi\|_{\Lap^2}^2+\|\pap^j(\pt+b\pap)\chi\|_{\Lap^2}^2\right)+M\sum_{j\leq k}\|\pap^j\eta\|_{\Lap^2}^2.
}
Combining this with \eqref{esti: eta} and choosing $M$ sufficiently small gives

\Aligns{
\sum_{j\leq k}\|\pap^{j}\eta\|_{\Lap^2}^2\lesssim \sum_{j\leq k}\left(\|\pap^j(\pt+b\pap)\chi\|_{\Lap^2}^2+\|\pap^j(\pt+b\pap)\chi\|_{\Lap^2}^2\right),
}
which completes the proof of the proposition.
\end{proof}

The next step in our analysis is to obtain estimates for quantities of the form $\|\pap^j(\partial_t+b\pap)f\|_{\Lap^2}$ in terms of $\|(\partial_t+b\pap)\pa^jf\|_{\Lap^2},$ which in turn will be bounded by the higher order energies to be defined in the next section. For this we first obtain estimates on $b$ and its derivatives.

\begin{proposition}\label{prop: b}
If $M_0$ in \eqref{bootstrap 1} is sufficiently small then for $2\leq k \leq \ell$

\Aligns{
\sum_{j\leq k}\|\pap^jb\|_{\Lap^2}\leq CM\sum_{j\leq k}\|\pap^ju\|_{\Lap^2}.
}
\end{proposition}
\begin{proof}
The proof is similar to that of Proposition \ref{prop: A}. Recall from Propositions \ref{prop: b akalpha} and \ref{prop: averages} that

\Aligns{
\begin{cases}
b=\Re\{-i(I-\CH)\frac{\ubar\chi}{\zetabar}-i[u,\CH]\frac{\etabar}{\zetabar\zeta_\alphap}\}+\frac{1}{2}\zeta[\mu,\CH]\frac{b_\alphap}{\zeta_\alphap}+\frac{1}{2}E(b)-\AV(b)\\
\AV(b)=\frac{\Re}{2\pi}\int_0^{2\pi}\frac{\ubar\chi}{|\zeta|^2}\zeta_\beta d\beta+\frac{\Re}{2\pi}\int_0^{2\pi}\frac{u\etabar}{|\zeta|^2}d\beta.
\end{cases}.
}
The Proposition now follows from similar arguments as those in the proof of Proposition \ref{prop: A}.
\end{proof}
An important corollary of Propositions \ref{prop: energy bounds temp} and \ref{prop: b} is the following result.

\begin{corollary}\label{cor: energy bounds}
If $M_0$ in \eqref{bootstrap 1} is sufficiently small then for $2\leq k \leq \ell,$

\Aligns{
\sum_{j_\leq k}&\left(\|\pa^jw\|_{\Lap^2}+\|\pap^ju\|_{\Lap^2}+\|\pap^j\eta\|_{\Lap^2}\right)\leq C\sum_{j\leq k}\left(\|(\partial_t+b\pap)\pap^jv\|_{\Lap^2}+\|(\partial_t+b\pap)\pap^j\chi\|_{\Lap^2}\right).
}
\end{corollary}
\begin{proof}
We first note that for any function $f$

\Aligns{
\pap^j(\partial_t+b\pap)f=(\partial_t+b\pap)\pap^jf+\sum_{1\leq i\leq j} {j\choose i} \pap^ib\,\pap^{j+1-i}f,
}
and therefore by Sobolev 

\Aligns{
\|\pap^j(\partial_t+b\pap)f\|_{\Lap^2}\leq\|(\partial_t+b\pap)\pap^jf\|_{\Lap^2}+c\sum_{i\leq \max\{2,j\}} \|\pap^ib\|_{\Lap^2}\sum_{i\leq j}\|\pap^{i}f\|_{\Lap^2}.
}
Summing this estimate over $j\leq k$ for $f=\chi$ and $f=v$ and using Propositions \ref{prop: u v}, \ref{prop: energy bounds temp} and \ref{prop: b}, Corollary \ref{cor: mu chi}, and the bootstrap assumption \eqref{bootstrap 1} we get the desired result.
\end{proof}

\section{Energy Estimates}\label{sec: energy estimates}
In this section we define the energy and prove energy estimates for equations \eqref{delta eq} and \eqref{deltat eq}. The main energy estimates are stated in Proposition \ref{prop: final energy estimates} below. We consider an equation of the form

\Align{\label{general eq}
((\partial_t+b\pap)^2+iA\pap-\pi)\Theta=G.
}
In most applications $\Theta$ will be the boundary value of a holomorphic function outside $\Omega$ decaying to zero as $|z|\to\infty,$ that is, $\Theta=(I-\CH)f$ for some $f.$ More precisely, the relevant choices of $\Theta$ are $\chi=(I-\CH)\mu$ and $v=(\pt+b\pap)\chi$. Since $v$ cannot be written as $(I-\CH)f$ we define the new unknown

\begin{align}\label{v tilde}
\tilde{v}:=(I-\CH)v.
\end{align}
  Associated to \eqref{general eq} we define the following basic energy

\Aligns{
E_0^\Theta:=\int_0^{2\pi}\frac{|(\pt+b\pap)\Theta|^2}{A}d\alphap+\int_0^{2\pi}\left(i\Theta_\alphap\Thetabar-\frac{\pi}{A}|\Theta|^2\right)d\alphap=:\E_0^\Theta+\F_0^\Theta.
}
and for the choices of $\Theta$ above we let

\Aligns{
&E_0(\chi):=\E_0(\chi)+\F_0(\chi):=\E_0^{\chi}+\F_0^{\chi}=E_0^{\chi},\\
&E_0(v):=\E_0(v)+\F_0(v):=\E_0^{\tv}+\F_0^{\tv}=E_0^{\tv}.
}
We will show below that if $\Theta=(I-\CH)f$ for some $f$ then $i\int_0^{2\pi}\Theta_\alpha\Thetabar d\alphap$ is non-negative. It is not, however, in general true that $\F_0^\Theta$ is non-negative even if $\Theta=(I-\CH)f,$ but this is the case if $\partial\Omega$ is a an exact circle. This can be seen by noting that in this case the Fourier expansion of $\Theta$ contains only negative frequencies if $\Theta=(I-\CH)f$, and then carrying out the integration on the frequency side after an application of Plancherel. Therefore, we expect that for small data where $\partial\Omega$ is nearly a circle, $\F_0^\Theta$ can be written as a positive term plus `higher order terms.' This can be achieved for instance by writing $\CH$ as the Hilbert transform on the circle plus an error. While this intuition is helpful, we will not use this argument in our applications, but instead explicitly decompose $\F_0^\Theta$ as the sum of a positive term and a `higher order' difference in terms of known quantities for choices of $\Theta$ that interest us. We will postpone this computation to after defining the higher order energies and now only prove the following general estimate.

\begin{lemma}\label{lem: H1/2}
The integral $i\int_0^{2\pi}\Theta_\alpha\Thetabar d\alphap$ is real and if $\Theta=(I-\CH)f$ for some $2\pi-$periodic function $f,$ then

\Aligns{
i\int_0^{2\pi}\Theta_\alphap\Thetabar d\alphap\geq 0.
}
\end{lemma}
\begin{remark}
Note that this lemma does not apply to the choice $\Theta=v,$ which is why we have replaced $v$ by $\tv$ in the definition of $E_0(v).$
\end{remark}
\begin{proof}
Integration by parts shows that the integral $i\int_0^{2\pi}\Theta_\alphap\Thetabar d\alphap$ is equal to its conjugate and is therefore real, and hence

\Aligns{
i\int_0^{2\pi}\Theta_\alphap\Thetabar d\alphap = \Re\left\{i\int_0^{2\pi}\frac{\Theta_\alphap}{|\zeta_\alphap|}\Thetabar |\zeta_\alphap|d\alphap\right\}.
}
 Now note that if $\Theta=(I-\CH)f$ then by Proposition~\ref{prop: hilbert} we can write $\Theta$ as the boundary value of a function $F$ which is holomorphic in $\Omega^c$ and decays as $|\zeta|^{-1}$ when $|\zeta|\to\infty.$ A simple computation using the holomorphicity of $F$ in $\Omega^c$ and the Cauchy-Riemann equations gives
 
 \Aligns{
 \Re\left\{\frac{i\Theta_\alphap\Thetabar}{|\zeta_\alphap|}\right\}=-\langle F, \frac{\partial F}{\partial n} \rangle
 }
 where $\bfn:=-\frac{iz_\alpha}{|z_\alpha|}$ is the exterior normal of $\Omega$ and $\langle F , G \rangle:= f_1 g_1 +f_2 g_2$ for complex numbers $F=f_1+if_2$ and $G=g_1+ig_2.$  From Green's formula and with $ds$ denoting the arc-length measure we get
 
\Aligns{
 i\int_0^{2\pi}\Theta_\alphap\Thetabar d\alphap=-\int_{\partial\Omega^c}\langle F, \frac{\partial F}{\partial n} \rangle ds=\iint_{\Omega^c}|\nabla F|^2 dxdy \geq0,
 }
 where we have used the decay properties of $F$ stated above to justify the use of Green's formula.
\end{proof}
With this basic positivity estimate in place we turn to the following energy identity for $E_0^\Theta.$

\begin{lemma}\label{lem: E0 identity}
Suppose $\Theta$ satisfies equation \eqref{general eq}. Then
\Aligns{
\frac{d}{dt}E_0^\Theta=2\int_0^{2\pi}\frac{1}{A}\Re\{G(\pt+b\pap)\Thetabar\}d\alphap
-\int_0^{2\pi}\left(\frac{1}{A}\frac{a_t}{a}\circ k^{-1}\right)\left(|(\pt+b\pap)\Theta|^2-\pi|\Theta|^2\right)d\alphap.
}
\end{lemma}
\begin{remark}
Note that if $\Theta=\chi$ or $v,$ then by the results of Section \ref{sec: normal form} the first integral on the right hand side above is of `order four'. However, in the definition of $E_0(v)$ we have used $\Theta=\tv,$ so to have this smallness we still need to show that $\tv$ satisfies a `cubic' equation. This will be accomplished in Proposition \ref{prop: tv} below.
\end{remark}

\begin{proof}
Precomposing with $k$ we can rewrite \eqref{general eq} as

\Align{\label{k general eq}
(\pt^2+ia\pa-\pi)\theta=g,\qquad \theta:=\Theta\circ k,~g=G\circ k.
}
Then

\Align{\label{k basic energy}
E_0^\Theta=\int_0^{2\pi}\frac{|\pt\theta|^2}{a}d\alpha+\int_0^{2\pi}\left(i\thetabar\pa\theta-\frac{\pi}{a}|\theta|^2\right)d\alpha.
}
It follows that

\Aligns{
\frac{d}{dt} E_0^\Theta=&\int_0^{2\pi}\frac{2}{a}\Re\{\pt^2\theta\pt\thetabar\} d\alpha-\int_0^{2\pi}\frac{a_t}{a^2}|\pt\theta|^2d\alpha\\
&+\int_0^{2\pi}i\pa\theta\pt\thetabar d\alpha+\int_0^{2\pi}i\thetabar\partial^2_{t\alpha}\theta d\alpha\\
&-\int_0^{2\pi}\frac{2\pi}{a}\Re\{\theta\pt\thetabar\}d\alpha+\int_0^{2\pi}\frac{\pi a_t}{a^2}|\theta|^2d\alpha\\
=&2\int_0^{2\pi}\frac{1}{a}\Re\{g\pt\thetabar\}d\alpha-\int_0^{2\pi}\frac{a_t}{a^2}(|\pt\theta|^2-\pi|\theta|^2)d\alpha.
}
Composing back with $k^{-1}$ we get the desired identity.
\end{proof}

We now turn to the higher energy estimates for \eqref{general eq}. For simplicity of notation we define

\Aligns{
\P=(\pt+b\pap)^2+iA\pap-\pi,
}
and note that

\Aligns{
\P\pap^jf-\pap^j\P f=\sum_{i=1}^j\pap^{j-i}[\P,\pap]\pap^{i-1}f.
}
Applying this identity to \eqref{general eq} we get

\Align{\label{j general eq}
\begin{cases}
\P\pap^j\Theta=G_j\\
G_j:=\pap^jG+\sum\limits_{i=1}^j\pap^{j-i}[\P,\pap]\pap^{i-1}\Theta\\
\quad\quad=\pap^jG+\sum\limits_{i=1}^j\pap^{j-i}\left(-b_\alphap(\pt+b\pap)\pap^{i}\Theta
-(\pt+b\pap)\left(b_\alphap\pap^{i}\Theta\right)-b_{\alphap}^{2}\pap^{i}\Theta-iA_\alphap\pap^{i}\Theta\right)
\end{cases}.
}
The $j$th order energy is now defined as

\Aligns{
E_j^\Theta:=\int_0^{2\pi}\frac{|(\pt+b\pap)\pap^j\Theta|^2}{A}d\alphap+\int_0^{2\pi}\left(i\pap^{j+1}\Theta\pap^j\Thetabar
-\frac{\pi}{A}|\pap^j\Theta|^2\right)=:\E_j^\Theta+\F_j^\Theta.
}
and in analogy with the undifferentiated case we let

\Align{\label{energy def 2}
&E_j(\chi):=\E_j(\chi)+\F_j(\chi):=\E_j^{\chi}+\F_j^{\chi}=E_j^{\chi},\\
&E_j(v):=\E_j(v)+\F_j(v):=\E_j^{\tv}+\F_j^{\tv}=E_j^{\tv}.
}
The following lemma follows from a similar argument to the proof of Lemma \ref{lem: E0 identity}.
\begin{lemma}\label{lem: Ej identity} $\pt E_j^\theta=R_j(t)$ where
\Aligns{
R_j(t):=2\int_0^{2\pi}\frac{1}{A}\Re\left(G_j(\pt+b\pap)\pap^j\Thetabar\right)d\alphap
-\int_0^{2\pi}\left(\frac{1}{A}\frac{a_t}{a}\circ k^{-1}\right)\left(|(\pt+b\pap)\pap^j\Theta|^2-\pi|\pap^j\Theta|^2\right)d\alphap.
}
\end{lemma}
Supposing for the moment that we know how to deal with the non-positive part $\F_j^\Theta$ of the energy, we can use Corollary \ref{cor: energy bounds} and Proposition \ref{prop: A} to estimate the quantities appearing in the bootstrap assumption \eqref{bootstrap 1} in terms of the positive parts of the energy $\E_j(\chi)$ and $\E_j(v).$ The only difficulty with this is that in the definition of $\E_j(v)$ we have replaced $v$ by $\tv,$ so in the next proposition we show that the conclusions of Corollary \ref{cor: energy bounds} hold with $v$ replaced by $\tv.$

\begin{proposition}\label{prop: energy dominance} If $M_0$ in \eqref{bootstrap 1} is sufficiently small then for $2\leq k \leq \ell,$

\Aligns{
\sum_{j_\leq k}&\left(\|\pap^jw\|_{\Lap^2}^2+\|\pap^ju\|_{\Lap^2}^2+\|\pap^j\eta\|_{\Lap^2}^2\right)\leq C\sum_{j\leq k}(\E_j^{\chi}+\E_j^{\tv}).
}
\end{proposition}
\begin{proof}
In view of Corollary \ref{cor: energy bounds} and Proposition \ref{prop: A} we only need to show that under the assumptions of the proposition

\Align{\label{ed temp 1}
\int_0^{2\pi}|(\pt+b\pap)\pap^jv|^2d\alphap \lesssim \int_0^{2\pi}|(\pt+b\pap)\pap^j\tv|^2d\alphap+M\sum_{i\leq j}\int_0^{2\pi}\left(|(\pt+b\pap)\pap^i\chi|^2+|(\pt+b\pap)\pap^iv|^2\right)d\alphap.
}
To see this we first write

\Aligns{
\tv\circ k=(I-H)\pt \delta=\pt(I-H)\delta+[z_t,H]\frac{\delta_\alpha}{z_\alpha}=2\pt\delta +[z_t,H]\frac{\delta_\alpha}{z_\alpha}
}
so

\Align{\label{ed temp 2}
\tv=2v+[u,\CH]\frac{\chi_\alphap}{\zeta_\alphap}.
}
Now 

\Align{\label{ed temp 4}
(\pt+b\pap)\pap^j[u,\CH]\frac{\chi_\alphap}{\zeta_\alphap}=&\pa^j(\pt+b\pap)[u,\CH]\frac{\chi_\alphap}{\zeta_\alphap}-\sum_{i=1}^j {{j}\choose{i}} \pap^{i}b \,\pap^{j-i}[u,\CH]\frac{\chi_\alphap}{\zeta_\alphap}.
}
By Corollary \ref{cor: energy bounds}, Proposition \ref{prop: b}, and Lemma \ref{lem: gHf} the contribution of the last term above can be bounded as

\Align{\label{ed temp 5}
\left\|\pap^{i}b \,\pap^{j-i}[u,\CH]\frac{\chi_\alphap}{\zeta_\alphap}\right\|_{\Lap^2}^2\lesssim M \sum_{i\leq j}\left(\|(\pt+b\pap)\pap^iv\|_{\Lap^2}+\|(\pt+b\pap)\pap^i\chi\|_{\Lap^2}\right).
}
To estimate the first term on the right hand side of \eqref{ed temp 4} we first note that

\Aligns{
\pt[z_t,H]\frac{\delta_\alpha}{z_\alpha}=[z_{tt},H]\frac{\delta_\alpha}{z_\alpha}+[z_t,H]\frac{\pa\delta_t}{z_\alpha}+\frac{1}{\pi i}\int_0^{2\pi}\left(\frac{z_t(\alpha)-z_t(\beta)}{z(\beta)-z(\alpha)}\right)^2\delta_\beta(\beta) d\beta,
}
so

\Aligns{
(\pt+b\pap)[u,\CH]\frac{\chi_\alphap}{\zeta_\alphap}=[w,\CH]\frac{\chi_\alphap}{\zeta_\alphap}+[u,\CH]\frac{\pap v}{\zeta_\alphap}+\frac{1}{\pi i}\int_0^{2\pi}\left(\frac{u(\alphap)-u(\betap)}{\zeta(\betap)-\zeta(\alphap)}\right)^2\chi_\betap(\betap) d\betap.
}
It follows from this, Corollary \ref{cor: energy bounds}, Proposition \ref{prop: u v}, Proposition \ref{prop: C2}, and Lemma \ref{lem: kernel der} that

\Align{\label{ed temp 6}
\left\|\pap^j(\pt+b\pap)[u,\CH]\frac{\chi_\alphap}{\zeta_\alphap}\right\|_{\Lap^2}^2\lesssim M \sum_{i\leq j}\left(\|(\pt+b\pap)\pap^iv\|_{\Lap^2}+\|(\pt+b\pap)\pap^i\chi\|_{\Lap^2}\right).
}
Combining \eqref{ed temp 2}--\eqref{ed temp 6} we get \eqref{ed temp 1}.
\end{proof}
We now turn to the issue of non-positivity of $\F_j^\Theta.$ Note that even if $\Theta$ can be written as $(I-\CH)f$ this  will not in general imply that $\pap^j\Theta$ is the boundary value of a function holomorphic outside of $\Omega,$ so even the first integral in the definition of $\F_j^\Theta$ above may not be non-negative for $j\geq1.$ Nevertheless, as for $\F_0^\Theta,$ we are able to show that the negative part of $\F_j^\Theta$ is of higher order for the choices of $\Theta$ we need in the energy estimates. The following simple observation is the main step in this direction.

\begin{lemma}\label{lem: energy positivity}
Suppose $\Theta:=(I-\CH)f$ for some $2\pi-$periodic function $f.$ Then with $g=\zeta\Theta$

\Aligns{
i\int_0^{2\pi}\Theta_\alphap\Thetabar d\alphap-\int_0^{2\pi}|\Theta|^2d\alphap&=\int_0^{2\pi}g_\alphap \gbar d\alphap-\int_0^{2\pi}(i\zeta_\alpha\zetabar+1)|\Theta|^2d\alphap-\int_0^{2\pi}i\mu\Theta_\alphap\Thetabar d\alphap\\
&\geq-\left(\|i\zeta_\alpha\zetabar+1\|_{\Lap^\infty}\|\Theta\|_{\Lap^2}^2+\|\mu\|_{\Lap^\infty}\|\Theta_\alphap\|_{\Lap^2}\|\Theta\|_{\Lap^2}\right).
}
\end{lemma}
\begin{proof}
The first equality follows from

\Aligns{
i\Theta_\alphap\Thetabar-|\Theta|^2=i\pap(\zeta\Theta)(\zetabar\Thetabar)-(i\zeta_\alphap\zetabar+1)|\Theta|^2-i\mu\Theta_\alphap\Thetabar.\\
}
To get the inequality it suffices to show that $i\int_0^{2\pi}g_\alphap\gbar d\alphap\geq0.$ For this note that

\Aligns{
g=(I-\CH)(\zeta f)-[\zeta,\CH]f
}
and that $[\zeta,\CH]f$ is independent of $\alphap.$ It follows that $$i\int_0^{2\pi}g_\alphap\gbar d\alphap=i\int_0^{2\pi}\pap[(I-\CH)(\zeta f)]\overline{[(I-\CH)(\zeta f)]}d\alphap$$ which is non-negative by Lemma \ref{lem: H1/2}.
\end{proof}

Lemma \ref{lem: energy positivity} shows that the difference between the energy and a positive term is of higher order. Note however, that the lower order term involves an extra derivative of $\Theta.$ This causes a problem only when we consider $\pap^\ell v,$ where $\ell$  is the maximum number of derivatives we commute. But in this case we can write

\Aligns{
\pap^\ell v=(\pt+b\pap)\pap^\ell\chi+[\pap^\ell,\pt+b\pap]\chi,
}
and the main term here is already bounded by the energy of $\chi,$ that is,

\Aligns{
\|(\pt+b\pap)\pap^\ell\chi\|_{\Lap^2}^2\lesssim \E_\ell^\chi.
}
Since $\|(\pt+b\pap)\pap^\ell\chi\|_{\Lap^2}^2$ is precisely the negative term in the energy of $\pap^\ell v$ this idea can be used to resolve the issue in the case where we commute the maximum number of derivatives. We will now make this argument more precise, starting with a few important identities stated only for the choices of $\Theta$ which will be used in the energy estimates, namely $\Theta=\chi$ and $\tv.$

\begin{lemma}\label{lem: chi v hol} If $\chi$ and $v$ are as in \eqref{quantities} then
\Aligns{
&\pap^j\chi=(I-\CH)\pap^j\mu-\sum_{i=1}^j\pap^{j-i}[\eta,\CH]\frac{\pap^i\mu}{\zeta_\alphap},\\
&\pap^j\tv=(I-\CH)\pap^jv-\sum_{i=1}^j\pap^{j-i}[\eta,\CH]\frac{\pap^iv}{\zeta_\alphap}.
}
\end{lemma}
\begin{proof}
The first identity follows from commuting $\pap^j$ with $\CH$ in the definition $\chi=(I-\CH)\mu$ of $\chi$ and noting that

\Aligns{
[\pap^j,\CH]f=\sum_{i=1}^j\pap^{j-i}[\pa,\CH]\pap^{i-1}f,
}
and

\Aligns{
[\pap,\CH]f=[\zeta_\alphap,\CH]\frac{f_\alphap}{\zeta_\alphap}=[\eta,\CH]\frac{f_\alphap}{\zeta_\alphap}.
}
The proof of the second identity is similar where we use the definition $\tv=(I-H)v.$
\end{proof}

We can now prove the following positivity estimate.

\begin{lemma}\label{lem: Fj}
\begin{enumerate}
\item If $M_0$ in \eqref{bootstrap 1} is sufficiently small then for $\ell \geq 2$

\Aligns{
&\sum_{i=0}^\ell\F_i^\chi\geq -C\sum_{i=0}^\ell\left(\E_i^\chi+\E_i^{\tv}\right)^{\frac{3}{2}},\\
&\sum_{i=0}^{\ell-1}\F_i^{\tv}\geq -C\sum_{i=0}^\ell\left(\E_i^\chi+\E_i^{\tv}\right)^{\frac{3}{2}}.
}

\item If $M_0$ in \eqref{bootstrap 1} is sufficiently small then for $\ell \geq 2$

\Aligns{
\F_\ell^{\tv}\geq -C\sum_{i=0}^\ell\left(\E_i^\chi+\E_i^{\tv}\right)^2 - C\sum_{i=0}^{\ell}(\E_i^\chi+\E_i^{\tv})^{\frac{3}{2}}-C\E_\ell^\chi.
}
\end{enumerate}
\end{lemma}
\begin{proof}
\begin{enumerate}
\item We assume $M_0$ is small enough that Corollary \ref{cor: energy bounds} holds. We start with the estimate for $\chi.$ By Lemma \ref{lem: chi v hol}

\Aligns{
\pap^i\chi=(I-\CH)\pa^i\mu-\sum_{m=1}^i\pap^{i-m}[\eta,\CH]\frac{\pap^m\mu}{\zeta_\alpha}=:f_i+g_i.
}
It follows that

\Align{\label{pos temp 1}
\F_i^\chi=&i\int_0^{2\pi}\pap f_i\fibar d\alphap-\int_0^{2\pi}|f_i|^2d\alphap\\
&-2\Re i\int_0^{2\pi}f_i\pap\gibar d\alphap-2\Re\int_0^{2\pi}f_i\gibar d\alphap+i\int_0^{2\pi}\pap g_i\gibar d\alphap-\int_0^{2\pi}|g_i|^2d\alphap.
}
To estimate the first line above we apply Lemma \ref{lem: energy positivity} with $\Theta=(I-\CH)\pap^j\mu=f_i$ to get (for $i\leq \ell$)

\Align{\label{pos temp 2}
i\int_0^{2\pi}\pap f_i\fibar d\alphap-\int_0^{2\pi}|f_i|^2d\alphap\geq&-(\|\zeta\|_{\Lap^\infty}\|\eta\|_{\Lap^\infty}+\|\mu\|_{\Lap^\infty})\|f_i\|_{\Lap^2}^2-\|\mu\|_{\Lap^\infty}\|\pap f_i\|_{\Lap^2}\|f_i\|_{\Lap^2}\\
\geq&-C\sum_{j=0}^i\left(\E_j^\chi+\E_j^{\tv}\right)^{\frac{3}{2}},
}
by Corollary \ref{cor: energy bounds} and Proposition \ref{prop: A}. Here to estimate $\|\pap f_\ell\|_{\Lap^2}$ we have noted that

\Align{\label{pos temp 3}
\pap f_\ell=(I-\CH)\pap^{\ell+1}\mu-[\eta,\CH]\frac{\pap^{\ell+1}\mu}{\zeta_\alphap}=(I-\CH)\pap^\ell(\zeta\etabar+\zetabar\eta)-[\eta,\CH]\frac{\pap^\ell(\zeta\etabar+\zetabar\eta)}{\zeta_\alphap}.
}
To estimate the second line in \eqref{pos temp 1} it suffices to show that for $i\leq \ell$

\Align{\label{pos temp 4}
\left(\|f_i\|_{\Lap^2}+\|g_i\|_{\Lap^2}\right)\left(\|g_i\|_{\Lap^2}+\|\pa g_i\|_{\Lap^2}\right)\leq C\sum_{j=0}^i\left(\E_j^\chi+\E_j^{\tv}\right)^{\frac{3}{2}}.
}
But \eqref{pos temp 4} is a direct consequence of Corollary \ref{cor: energy bounds} and Lemma \ref{lem: gHf}. Combining \eqref{pos temp 1}, \eqref{pos temp 2}, and \eqref{pos temp 4} we get the estimate for $\chi.$

The estimate for $\tv$ is similar. Using Lemma \ref{lem: chi v hol} we write

\Aligns{
\pap^i \tv=\phi_i+\psi_i
}
where

\Aligns{
\phi_i:=(I-\CH)\pap^iv,\qquad \psi_i=-\sum_{m=1}^i\pa^{i-m}[\eta,\CH]\frac{\pap^m v}{\zeta_\alpha}.
}
The argument is now the same as for $\chi$ where we replace $g_i$ by $\psi_i$ and $f_i$ by $\phi_i$ everywhere. The only difference is that \eqref{pos temp 3} is now replaced by

\Aligns{
\pap \phi_{\ell-1}=(I-\CH)\pap^\ell v-[\eta,\CH]\frac{\pap^\ell v}{\zeta_\alphap},
}
which is responsible for the loss of one derivative.

\item Note that with $c_i:={{\ell}\choose{i}}$

\Aligns{
\pap^\ell v =(\pt+b\pap)\pap^\ell \chi+\sum_{i=1}^\ell c_i \pap^ib\,\pap^{\ell+1-i}\chi.
}
It follows from this, Corollary \ref{cor: energy bounds}, Proposition \ref{prop: A}, and Proposition \ref{prop: b} that

\Aligns{
\|\pap^\ell v\|_{\Lap^2}^2\leq C\sum_{i=0}^{\ell}(\E_i^\chi+\E_i^{\tv})^2+C \E_\ell^\chi,
}
and from the second identity in Lemma \ref{lem: chi v hol} that

\begin{align*}
\|\pap^\ell \tv\|_{\Lap^2}^2\leq C\sum_{i=0}^{\ell}(\E_i^\chi+\E_i^{\tv})^2+C \E_\ell^\chi.
\end{align*}
From the definition of $\F_\ell^{\tv}$ and in view of Proposition \ref{prop: A} if $M_{0}$ is sufficiently small it follows that

\Aligns{
\F_\ell^{\tv}\geq i\int_0^{2\pi}\pap^{j+1}\tv\pap^j\tvbar d\alphap-C\sum_{i=0}^{\ell}(\E_i^\chi+\E_i^{\tv})^2- C \E_\ell^\chi,
}
so it suffice to show

\Align{\label{pos temp 5}
i\int_0^{2\pi}\pap^{\ell+1}\tv\pa^\ell\tvbar d\alphap\geq -C\sum_{i=0}^\ell\left(\E_i^\chi+\E_i^{\tv}\right)^{\frac{3}{2}}.
}
For this we use  Lemma \ref{lem: chi v hol} to write

\Aligns{
\pap^\ell \tv=f+g
}
where

\Aligns{
f=(I-\CH)\pap^\ell v,\quad g=-\sum_{i=1}^\ell\pap^{\ell-i}[\eta,\CH]\frac{\pap^i v}{\zeta_\alphap}.
}
Since

\Aligns{
i\int_0^{2\pi}f_\alphap\fbar d\alphap\geq0,
}
arguing as in \eqref{pos temp 1} and \eqref{pos temp 4} we just need to show that

\Aligns{
\|g\|_{\Lap^2}\|g_\alphap\|_{\Lap^2}+\|f\|_{\Lap^2}\|g_\alphap\|_{\Lap^2}+\|f\|_{\Lap^{2}}\|g\|_{\Lap^{2}}\leq C\sum_{i=0}^\ell\left(\E_i^\chi+\E_i^{\tv}\right)^{\frac{3}{2}}.
}
But this is again a consequence of Corollary~\ref{cor: energy bounds}, Proposition~\ref{prop: u v}, and Lemma~\ref{lem: gHf}. This now proves \eqref{pos temp 5} which concludes the proof of the Lemma.
\end{enumerate}
\end{proof}

Combining Lemmas \ref{lem: Ej identity} and \ref{lem: Fj} we see that if $M_0$ in \eqref{bootstrap 1} is sufficiently small we can find constants $c_{1}$, $C_1$ and $C_2$ such that with $R_k$ as in Lemma \ref{lem: Ej identity}

\Aligns{
\sum_{k\leq \ell}\E_k^\chi(t)+\sum_{k\leq\ell-1}\E_k^{\tv}(t)\leq \sum_{k\leq \ell}\left(E_k^\chi(0)+E_k^{\tv}(0)\right)+C_1\sum_{k\leq\ell}\left(\E_k^\chi(t)+\E_k^{\tv}(t)\right)^\frac{3}{2}+\sum_{k\leq\ell}\int_0^t |R_k(t)|dt
}
and

\Aligns{
\E_\ell^{\tv}(t)-c_{1}\E_\ell^\chi(t)\leq &\sum_{k\leq \ell}\left(E_k^\chi(0)+E_k^{\tv}(0)\right)+C_2\sum_{k\leq\ell}\left(\E_k^\chi(t)+\E_k^{\tv}(t)\right)^\frac{3}{2}+C_2\sum_{k\leq\ell}\left(\E_k^\chi(t)+\E_k^{\tv}(t)\right)^2\\
&+\sum_{k\leq\ell}\int_0^t |R_k(t)|dt.
}
Adding an appropriate multiple of the second estimate to the first we get the following energy estimate.

\begin{corollary}\label{cor: energy estimate}
If $M_0$ in \eqref{bootstrap 1} is sufficiently small then with $N_k$ as in Lemma \ref{lem: Ej identity}

\Aligns{
\sum_{k\leq\ell}(\E_k^\chi(t)+\E_k^{\tv}(t))\leq &C\sum_{k\leq \ell}(E_k^\chi(0)+E_k^{\tv}(0))+ C\sum_{k\leq \ell}(\E_k^\chi(t)+\E_k^{\tv}(t))^{\frac{3}{2}}+  C\sum_{k\leq \ell}(\E_k^\chi(t)+\E_k^{\tv}(t))^{2}\\
&+C\sum_{k\leq\ell}\int_0^t|R_k(t)|dt.
}
\end{corollary}

We now turn to the estimates for $R_k.$ For notational convenience we define

\Align{\label{total energy}
\E:=\sum_{i=0}^\ell\left(\E_i^\chi+\E_i^{\tv}\right).
}
Our first step will be to compute the equation for $\tv.$

\begin{proposition}\label{prop: tv} $\tv=(I-H)v$ satisfies
\Aligns{
(\pt^2+ia\pa-\pi)(\tv\circ k)=&(I-H)(\pt^2+ia\pa-\pi)\delta_t+2[z_t,H]\frac{\pa(\pt^2+ia\pa-\pi)\delta}{z_\alpha}\\
&+\frac{1}{\pi i}\int_0^{2\pi}\left(\frac{z_t(t,\beta)-z_t(t,\alpha)}{z(t,\beta)-z(t,\alpha)}\right)^2\delta_{t\beta}(t,\beta)d\beta\\
&+2\pi\left[z[\ep,H]\frac{z_{t\alpha}}{z_\alpha},H\right]\frac{\delta_\alpha}{z_\alpha}+2\pi[E(z_t),H]\frac{\delta_\alpha}{z_\alpha}\\
&+\pi\left[z[\ep,H]\frac{z_{t\alpha}}{z_\alpha},H\right]\left(\frac{\pa}{z_\alpha}\right)^2\delta+\pi[E(z_t),H]\left(\frac{\pa}{z_\alpha}\right)^2\delta\\
&+-2[z_t,H]\frac{\pa}{z_\alpha}\left(\frac{g^a\delta_\alpha}{z_\alpha}\right)+2[z_t,H]\frac{\pa}{z_\alpha}\left(\frac{z_{tt}\delta_\alpha}{z_\alpha}\right)\\
&+\frac{\pi}{2}[z\delta,H]\frac{\pa}{z_\alpha}[z_t,H]\frac{\ep_\alpha}{z_\alpha}-\frac{\pi}{2}[E(z),H]\frac{\delta_{t\alpha}}{z_\alpha}.
}
\end{proposition}

\begin{proof}
From Lemma \ref{lem: operator H commutator} we have

\Aligns{
(\pt^2+ia\pa-\pi)(\tv\circ k)=&(\pt^2+ia\pa-\pi)(I-H)\delta_t=(I-H)(\pt^2+ia\pa-\pi)\delta_t-[\pt^2+ia\pa,H]\delta_t\\
=&(I-H)(\pt^2+ia\pa-\pi)\delta_t-\frac{\pi}{2}[(I-\Hbar)z,H]\frac{\delta_{t\alpha}}{z_\alpha}+2[z_t,H]\frac{\pa\delta_{tt}}{z_\alpha}\\
&+\frac{1}{\pi i}\int_0^{2\pi}\left(\frac{z_t(\beta)-z_t(\alpha)}{z(\beta)-z(\alpha)}\right)^2\delta_{t\beta}(\beta)d\beta\\
=&-\frac{\pi}{2}[(I-\Hbar)z,H]\frac{\delta_{t\alpha}}{z_\alpha}-2[z_t,H]\frac{\pa(ia\pa\delta-\pi\delta)}{z_\alpha}\\
&+(I-H)(\pt^2+ia\pa-\pi)\delta_t+2[z_t,H]\frac{\pa(\pt^2+ia\pa-\pi)\delta}{z_\alpha}\\
&+\frac{1}{\pi i}\int_0^{2\pi}\left(\frac{z_t(\beta)-z_t(\alpha)}{z(\beta)-z(\alpha)}\right)^2\delta_{t\beta}(\beta)d\beta.
}
The last three terms above already have the right form, so we only need to consider

\Aligns{
I+II+III:=-\frac{\pi}{2}[(I-\Hbar)z,H]\frac{\delta_{t\alpha}}{z_\alpha}-2[z_t,H]\frac{\pa(ia\pa\delta)}{z_\alpha}+2\pi[z_t,H]\frac{\pa\delta}{z_\alpha}.
} 
Note that if $g$ is the boundary value of a decaying holomorphic function $F$ outside of $\Omega,$ i.e., $g=(I-H)f_1$ then $\frac{g_\alpha}{z_\alpha}$ is the boundary value of $F_z$ so $\frac{g_\alpha}{z_\alpha}=(I-H)f_2$ for some $f_2.$ We will use this observation repeatedly in the rest of this proof. Applying this observation to $III$ we see that since $\delta=(I-H)\ep$

\Aligns{
III=\pi[(I+H)z_t,H]\frac{\delta_\alpha}{z_\alpha}=\pi[(H+\Hbar)z_t,H]\frac{\delta_\alpha}{z_\alpha}=\pi\left[z[\ep,H]\frac{z_{t\alpha}}{z_\alpha},H\right]\frac{\delta_\alpha}{z_\alpha}+\pi[E(z_t),H]\frac{\delta_\alpha}{z_\alpha}.
}
For $II$ we use \eqref{z eq} to write

\Aligns{
II=-2[z_t,H]\frac{\pa}{z_\alpha}\left(\frac{g^a\delta_\alpha}{z_\alpha}\right)+2[z_t,H]\frac{\pa}{z_\alpha}\left(\frac{z_{tt}\delta_\alpha}{z_\alpha}\right)+2\pi[z_t,H]\frac{\pa}{z_\alpha}\left(\frac{z\delta_\alpha}{z_\alpha}\right).
}
The first two terms have the right form and we can rewrite the last term as

\Aligns{
2\pi[z_t,H]\frac{\pa}{z_\alpha}\left(\frac{z\delta_\alpha}{z_\alpha}\right)=&\pi[(I+H)z_t,H]\frac{\delta_\alpha}{z_\alpha}+\pi\left[z_t,H\right]z\left(\frac{\pa}{z_\alpha}\right)^2\delta\\
=&\pi\left[z[\ep,H]\frac{z_{t\alpha}}{z_\alpha},H\right]\frac{\delta_\alpha}{z_\alpha}+\pi[E(z_t),H]\frac{\delta_\alpha}{z_\alpha}\\
&+2\pi z[z_{t},H]\left(\frac{\pa}{z_{\alpha}}\right)^{2}\delta+2\pi [[z_{t},H],z]\left(\frac{\pa}{z_{\alpha}}\right)^{2}\delta
}
The first term in the last line above can be written as

\begin{align*}
&\pi z[(I+H)z_{t},H]\left(\frac{\pa}{z_{\alpha}}\right)^{2}\delta=\pi z[(H+\Hbar)z_{t},H]\left(\frac{\pa}{z_{\alpha}}\right)^{2}\delta\\
=&\pi z\left[z[\ep,H]\frac{z_{t\alpha}}{z_\alpha},H\right]\left(\frac{\partial_{\alpha}}{z_{\alpha}}\right)^{2}\delta+\pi[E(z_t),H]\left(\frac{\pa}{z_{\alpha}}\right)^{2}\delta.
\end{align*}
For the second term $2\pi [[z_{t},H],z]\left(\frac{\pa}{z_{\alpha}}\right)^{2}\delta$, we use Jacobi identity to write this as

\begin{align*}
&-2\pi[[H,z],z_{t}]\left(\frac{\pa}{z_{\alpha}}\right)^{2}\delta=-2\pi[H,z]z_{t}\left(\frac{\pa}{z_{\alpha}}\right)^{2}\delta\\
=&\frac{2}{i}\int_{0}^{2\pi}z_{t}(\beta)\partial_{\beta}\left(\frac{\partial_{\beta}}{z_{\beta}}\left((I-H)\ep\right)(\beta)\right)d\beta.
\end{align*}
By Lemma \ref{lem: H plus Hbar}, we have

\begin{align*}
(I-H)\ep=(I+\Hbar)\ep=-z[\ep,H]\frac{\ep_{\alpha}}{z_{\alpha}}-E(\ep).
\end{align*}
Therefore the contribution we need to consider is

\begin{align*}
\int_{0}^{2\pi}z_{t}(\beta)\partial_{\beta}\left((I+\Hbar)\frac{\ep_{\beta}}{z_{\beta}}\right)(\beta)d\beta=\int_{0}^{2\pi}z_{t}(\beta)G_{\zbar}(\zbar(\beta))\zbar_{\beta}(\beta)d\beta=0.
\end{align*}
Here $G_{\zbar}(\zbar(\beta))=\frac{\partial_{\beta}}{\zbar_{\beta}}\left((I+\Hbar)\frac{\ep_{\beta}}{z_{\beta}}\right)$ is the boundary value of an anti-holomorphic function $G_{\zbar}(\zbar)$ in $\Omega$. We also used the fact that $z_{t}(\beta)$ is the boundary value of an anti-holomorphic function in $\Omega$. Finally for $I$ we compute

\Aligns{
I=&-\frac{\pi}{2}[(I-\Hbar)z,H]\frac{\delta_{t\alpha}}{z_\alpha}\\=&\frac{\pi}{2}[(I+\Hbar)z,H]\frac{\delta_{t\alpha}}{z_{\alpha}}\\
=&\frac{\pi}{2}[z\delta,H]\frac{\delta_{t\alpha}}{z_\alpha}+\frac{\pi}{2}[E(z),H]\frac{\delta_{t\alpha}}{z_\alpha}\\
=&\frac{\pi}{2}[z\delta,H]\frac{\pa}{z_\alpha}(I-H)(z_t\zbar)-\frac{\pi}{2}[z\delta,H]\frac{\pa}{z_\alpha}[z_t,H]\frac{\ep_\alpha}{z_\alpha}+\frac{\pi}{2}[E(z),H]\frac{\delta_{t\alpha}}{z_\alpha}.
}
Again the last two terms have the right form and for the first we use Lemma \ref{lem: fgh commutator} with $f=z,$ $g=\delta$ and $h=\frac{\pa}{z_\alpha}(I-H)(z_t\zbar)$ and the fact that for and $f_1$ and $f_2$

\Aligns{
[(I-H)f_1,H](I-H)f_2=0
}
to write

\Aligns{
\frac{\pi}{2}[z\delta,H]\frac{\pa}{z_\alpha}(I-H)(z_t\zbar)=\frac{\pi}{2}[z,H]\delta\frac{\pa}{z_\alpha}(I-H)(z_t\zbar)=-\frac{1}{\pi i}\frac{\pi}{2}\int_0^{2\pi}\delta\pa(I-H)(z_t\zbar)d\alpha=0.
}
Here for the last step we have used the fact that since $\delta$ and $(I-H)(z_t\zbar)$ are boundary values of holomorphic functions $F_1$ and $F_2,$ respectively, in $\Omega^c$ going to zero as $|z|\to\infty,$

\Aligns{
\int_0^{2\pi}\delta\pa(I-H)(z_t\zbar)d\alpha=\int_{\partial\Omega^c}F_1(z)\partial_zF_2(z)dz=0.
}
\end{proof}

The following estimate is used for estimating the second integral in the definition of $N_k.$

\begin{lemma}\label{lem: at}
If $M_0$ in \eqref{bootstrap 1} is sufficiently small

\Aligns{
\left\|\frac{1}{A}\frac{a_t}{a}\circ k^{-1}\right\|_{\Lap^\infty}\leq C\E.
}
\end{lemma}
\begin{proof}
We use Lemmas \ref{lem: K*at} and \ref{lem: K*} and the Sobolev embedding $H_\alphap^1\hookrightarrow \Lap^\infty.$ Recalling that $A=(ak_\alpha)\circ k^{-1},$ from Lemma \ref{lem: K*at} and precomposition with $k^{-1}$ we get

\Aligns{
(I+\CK^*)(\frac{a_t}{a}\circ k^{-1} A|\zeta_\alphap|)=\Re\Bigg\{&-i\frac{\zeta_\alphap}{|\zeta_\alphap|}\Big\{2[u,\CH]\frac{\wbar_\alphap}{\zeta_\alphap}+[2w-g^a\circ k^{-1},\CH]\frac{\ubar_\alphap}{\zeta_\alphap}\\
&-\frac{\pi}{2}(I-\CH)\left([u,\CH]\frac{\zetabar_\alphap}{\zeta_\alphap}\right)+\frac{1}{\pi i}\int_0^{2\pi}\left(\frac{u(t,\betap)-u(t,\alphap)}{\zeta(t,\betap)-\zeta(t,\alphap)}\right)^2\ubar_\betap(t,\betap)d\betap\Big\}\Bigg\}.
}
Recalling that $g^1\circ k^{-1}=\frac{\pi}{2}\zeta\chi+E(\zeta),$ it follows from this, Lemmas \ref{lem: gHf} and \ref{lem: I-H L2}, and Corollary \ref{cor: energy bounds} that

\Align{\label{at temp 1}
\sum_{i=0}^1\|\pap^i\left((I+\CK^*)(\frac{a_t}{a}\circ k^{-1} A|\zeta_\alphap|)\right)\|_{\Lap^2}\leq C\E.
}
On the other hand,

\Aligns{
\frac{1}{A}\frac{a_t}{a}\circ k^{-1}=\frac{1}{A^2|\zeta_\alphap|}(I+\CK^*)(\frac{a_t}{a}\circ k^{-1} A|\zeta_\alphap|)-\frac{1}{A^2|\zeta_\alphap|}\CK^*(\frac{a_t}{a}\circ k^{-1} A|\zeta_\alphap|)=:I-\frac{1}{A^2|\zeta_\alphap|}II.
}
By \eqref{at temp 1} and Propositions \ref{prop: zeta bounds} and \ref{prop: A}

\Aligns{
\|I\|_{\Lap^2}+\|\pap I\|_{\Lap^2}\leq C\E,
}
and therefore in view of Propositions \ref{prop: zeta bounds} and \ref{prop: A} to complete the proof of the lemma it suffice to show that

\Align{\label{at temp 2}
\|II\|_{\Lap^2}+\|\pa II\|_{\Lap^2}\leq CM\left(\left\|\frac{1}{A}\frac{a_t}{a}\circ k^{-1}\right\|_{\Lap^2}+\left\|\pap\left(\frac{1}{A}\frac{a_t}{a}\circ k^{-1}\right)\right\|_{\Lap^2}\right)+C\E.
}
For this we use Lemma \ref{lem: K*}. Note that since $K^*f=-\Re\{\frac{z_\alpha}{|z_\alpha|}H\frac{|z_\alpha|f}{z_\alpha}\}$ we may replace $z$ by $\zeta$ and $z_\alpha$ by $\zeta_\alphap$ everywhere in formula derived in Lemma \ref{lem: K*} to get a representation for $\CK^*.$ Using this observation and Lemma \ref{lem: H Hbar}, \ref{lem: K*} we get with $f=\frac{a_t}{a}\circ k^{-1}A|\zeta_\alphap|$

\Align{\label{at temp 3}
II=&\frac{1}{\pi |\zeta_\alphap|}\int_0^{2\pi}f(\alphap)|\zeta_\alphap(\alphap)|d\alphap+\frac{\AV(f|\zeta_\alphap|)}{|\zeta_\alphap|}\\
&-\frac{\zeta}{2|\zeta_\alphap|}[\mu,\CH]\frac{\pa(f|\zeta_\alphap|)}{\zeta_\alphap}-\frac{E(f|\zeta_\alphap|)}{2|\zeta_\alphap|}-\Re\{\frac{1}{|\zeta_\alphap|}[\eta,\CH]\frac{f|\zeta_\alphap|}{\zeta_\alphap}\}.
}
The contribution of the second line above can be bounded by the right hand side of \eqref{at temp 2} using Lemma \ref{lem: gHf}, Corollary \ref{cor: E}, and Proposition \ref{prop: zeta bounds}. To estimate the contribution of the first line of \eqref{at temp 3} we go back to equation \eqref{zt eq} which we rewrite as

\Align{\label{at temp 4}
f|\zeta_\alphap|=\frac{a_t}{a}\circ k^{-1} A|\zeta_\alphap|^2=i\zetabar_\alphap(\pt+b\pap)w-\pi u_\alphap\zetabar_\alphap-(A-\pi)u_\alphap\zetabar_\alphap-\frac{\pi i}{2}[\ubar,\CH\frac{1}{\zeta_\alphap}+\CHbar\frac{1}{\zetabar_\alpha}]\zeta_\alphap.
}
Moreover, we can write

\Aligns{
\AV(g)=\int_0^{2\pi}\frac{\eta g}{\zeta}\,d\alphap+i\int_0^{2\pi}g \,d\alphap
}
so to prove \eqref{at temp 2} for the first line of \eqref{at temp 3}, it suffices to bound $\int_0^{2\pi} gd\alphap$ by the right hand side of \eqref{at temp 2} with $g$ replaced by each of the terms on the right hand side of \eqref{at temp 4}. For the last two terms of \eqref{at temp 4} the contributions are of the right form in view of Lemma \ref{lem: H plus Hbar} and Proposition \ref{prop: A}. For the second term of \eqref{at temp 4} it suffices to note that since $u$ is anti-holomorphic inside $\Omega$

\Aligns{
\int_0^{2\pi}u_\alphap\zetabar_\alphap d\alphap=\int_0^{2\pi}u_\alphap \etabar\,d\alphap+i\int_0^{2\pi}u\zetabar_\alphap d\alphap=\int_0^{2\pi}u_\alphap \etabar\,d\alphap
}
which can be bounded by the right hand side of \eqref{at temp 2}. Finally for the first term of \eqref{at temp 4} we write $z_t=F(t,\zbar)$ for an anti-holomorphic function to get

\Aligns{
z_{tt}=F_t+F_\zbar\zbar_t=F_t+\frac{z_{t\alpha}\zbar_t}{\zbar_\alpha}, \qquad z_{ttt}=F_{tt}+\frac{z_{tt\alpha}}{\zbar_\alpha}\zbar_t+\frac{(\zbar_{tt\alpha}\zbar_t+z_{t\alpha}\zbar_{tt})\zbar_\alpha-\zbar_{t\alpha}z_{t\alpha}\zbar_t}{\zbar_\alpha^2}.
}
Since $F_{tt}$ is anti-holomorphic, it follows that

\Aligns{
\int_0^{2\pi}(\pt+b\pap)w\zetabar_\alphap d\alphap=\int_0^{2\pi} w_\alphap\ubar\, d\alphap+\int_0^{2\pi}\frac{(\wbar_\alphap\ubar+u_\alphap\wbar)
\zetabar_\alphap-\ubar_\alphap u_\alphap u}{\zetabar_\alphap} d\alphap
}
which can be bounded by the right hand side of \eqref{at temp 2}. This completes the proof of \eqref{at temp 2} and hence of the lemma.
\end{proof}

\begin{corollary}\label{cor: at}
If $M_0$ in \eqref{bootstrap 1} is sufficiently small then for all $j\leq \ell$ and with $\Theta=\chi$ or $v$

\Aligns{
\int_0^{2\pi}\left|\frac{1}{A}\frac{a_t}{a}\circ k^{-1}\right|\left(|(\pt+b\pap)\pap^j\Theta|^2+\pi|\pap^j\Theta|^2\right)d\alphap\leq C\E^2.
}
\end{corollary}
\begin{proof}
This is a direct corollary of the definition of $\E,$ Lemma \ref{lem: at}, Proposition \ref{prop: u v}, and Corollaries \ref{cor: mu chi} and \ref{cor: energy bounds}.
\end{proof}
The last step before stating the main result of this section is to obtain an expression for the time derivative of $b$ and then estimates for it.

\begin{proposition}\label{prop: ktt}
Suppose that $k$ is given as in Remark~\ref{prop: k existence} and that it is increasing. Then $k_{tt}=(\pt+b\pa)b\circ k$ satisfies

\begin{align*}
(I-H)k_{tt}=&-i(I-H)\frac{\zbar_{tt}\ep+\zbar_{t}\ep_{t}}{\zbar}+i(I-H)\frac{\zbar_{t}^{2}\ep}{\zbar^{2}}-i[z_{t},H]\frac{(\log(\zbar e^{ik}))_{t\alpha}+ik_{t\alpha}}{z_{\alpha}}+i[z_{t},H]\frac{1}{z_{\alpha}}\pa\left(\frac{\zbar_{t}\ep}{\zbar}\right)\\
&-i[z_{tt},H]\frac{(\log(\zbar e^{ik}))_{\alpha}}{z_{\alpha}}-\frac{1}{\pi}\int_{0}^{2\pi}\left(\frac{z_{t}(\beta)-z_{t}(\alpha)}{z(\beta)-z(\alpha)}\right)^{2}(\log(\zbar e^{ik}))_{\beta}d\beta
\end{align*}
and

\begin{align*}
\Re\AV(\pt k_t)=\frac{\Im}{2\pi}\int_0^{2\pi}\left(\frac{z_{t\beta}z-z_tz_\beta}{z^2}\right)k_td\beta+\frac{\Re}{2\pi}\pt\int_0^{2\pi}\frac{\zbar_t\ep}{|z|^2}z_\beta d\beta+\frac{\Re}{2\pi}\pt \int_0^{2\pi}(\log (\zbar e^{ik}))_\beta \frac{z_t}{z}d\beta.
\end{align*}
\end{proposition}
\begin{proof}
Differentiating the first formula in Proposition \ref{prop: b akalpha} with respect to time, we obtain

\begin{align*}
(I-H)k_{tt}&=\pt(I-H)k_{t}+[z_{t},H]\frac{k_{t\alpha}}{z_{\alpha}}\\
&=-i\pt(I-H)\frac{\zbar_{t}\ep}{\zbar}-i\pt[z_{t},H]\frac{\left(\log(\zbar e^{ik})\right)_{\alpha}}{z_{\alpha}}+[z_{t},H]\frac{k_{t\alpha}}{z_{\alpha}}\\
&=:I+II+III.
\end{align*}
Direct computations imply that

\begin{align*}
I=&-i(I-H)\frac{\zbar_{tt}\ep+\zbar_{t}\ep_{t}}{\zbar}+i(I-H)\frac{\zbar_{t}^{2}\ep}{\zbar^{2}}+i[z_{t},H]\frac{1}{z_{\alpha}}\pa\left(\frac{\zbar_{t}\ep}{\zbar}\right)\\
II=&-i[z_{t},H]\frac{\left(\log(\zbar e^{ik})\right)_{t\alpha}}{z_{\alpha}}-i[z_{tt},H]\frac{\left(\log(\zbar e^{ik})\right)_{\alpha}}{z_{\alpha}}-\frac{1}{\pi}\int_{0}^{2\pi}\left(\frac{z_{t}(\beta)-z_{t}(\alpha)}{z(\beta)-z(\alpha)}\right)^{2}\left(\log(\zbar e^{ik})\right)_{\beta}d\beta.
\end{align*}
Putting all these together, the first formula in the proposition follows. The second formula follows from differentiating the last formula in Proposition \ref{prop: averages} with respect to time.
\end{proof}

We are finally ready to prove the main result of this section.

\begin{proposition}\label{prop: final energy estimates}
If $M_0$ in \eqref{bootstrap 1} is sufficiently small then with $R_k$ as in Lemma \ref{lem: Ej identity}

\Aligns{
\sum_{k\leq\ell}(\E_k^\chi(t)+\E_k^{\tv}(t))\lesssim &\sum_{k\leq \ell}(\E_k^\chi(0)+\E_k^{\tv}(0))+ \sum_{k\leq \ell}(\E_k^\chi(t)+\E_k^{\tv}(t))^{\frac{3}{2}}+  \sum_{k\leq \ell}(\E_k^\chi(t)+\E_k^{\tv}(t))^{2}\\
&+\sum_{k\leq\ell}\int_0^t (\E_k^\chi(s)+\E_k^{\tv}(s))^{2}\ ds.
}
\end{proposition}
\begin{proof}
By Corollary \ref{cor: energy estimate} we only need to estimate the nonlinear term $R_k.$ Here $R_k$ is defined in Lemma \ref{lem: Ej identity} and $G_j$ is given in \eqref{j general eq} as

\Align{\label{fee temp 0}
G_j=\pap^jG-\sum\limits_{i=1}^j\pap^{j-i}\left(b_\alphap(\pt+b\pap)\pap^{i}\Theta+b_\alphap^2\pa^i\Theta+(\pt+b\pap)
\left(b_\alphap\pap^{i}\Theta\right)+iA_\alphap\pap^{i}\Theta\right).
}
It follows from Corollary \ref{cor: at} that we only need to consider the first integral in the expression for $R_k$ in Lemma \ref{lem: Ej identity}. In particular we need to show that

\Align{\label{fee temp 1}
\|G_j\|_{\Lap^2}^2\lesssim \E^3.
}
We begin with the contribution of $\pap^jG.$ When $\Theta=\chi$ this is already dealt with in Propositions \ref{prop: NL} and \ref{prop: energy dominance} and Corollaries \ref{cor: energy bounds} and \ref{cor: at}.  When $\Theta=\tv$ we use the equation derived for $\tv$ in Proposition \ref{prop: tv}. But then in view of Proposition \ref{prop: NL}, the contribution of $\pap^jG$ when $\Theta=\tv$ is also handled by simlar arguments as before using Lemmas \ref{lem: I-H L2}, \ref{lem: gHf}, \ref{lem: kernel der}, Propositions\ref{prop: C1}, \ref{prop: C2}, \ref{prop: zeta bounds}, \ref{prop: energy dominance},  and Corollaries \ref{cor: E}, \ref{cor: energy bounds}. We omit the details. To estimate the contribution of the second term on the right hand side of \eqref{fee temp 0} we note that

\Aligns{
(\pt+b\pap)b_\alphap=\pap(\pt+b\pap)b-b_\alphap^2
}
and use Proposition \ref{prop: ktt} to express $(\pt+b\pap)b$ in terms of quantities we can already control. Here we also use the observation that

\Aligns{
(\log(\zbar e^{ik}))_\alphap=\frac{\etabar}{\zetabar}
}
and that $\pt \eta=u_{\alphap}-iu.$ The proof of the proposition can now be completed by appealing to Propositions \ref{prop: A}, \ref{prop: b}, \ref{prop: energy dominance} and Corollary \ref{cor: energy bounds}.
\end{proof}
\section{Long Time Well-posedness}\label{sec: long wp}

In this final section we prove long-time existence for solutions of the system 

\begin{equation}\label{z final eq}
\begin{cases}
&z_{tt}+iaz_\alpha=-\frac{\pi}{2}(I-\Hbar)z,\quad \zbar_t=H\zbar_t\\
&z(0,\alpha)=z_0(\alpha), \quad z_t(0,\alpha)=z_1(\alpha)
\end{cases},
\end{equation}
with small initial data. More precisely we will complete the proof of Theorem \ref{long time existence theorem}. This section is divided into two parts. To use the energy estimates from the previous section we need to transfer the smallness of the data for equation \eqref{z final eq} to the initial smallness of the quantities appearing in the bootstrap assumption \eqref{bootstrap 1} and the initial energy defined in the previous section. This will be accomplished in Subsection \ref{subsec: initial data}. Then in Subsection \ref{subsec: longtime exist} we will establish Theorem \ref{long time existence theorem}, by showing long-time existence of solutions to \eqref{z final eq} assuming that initially the bootstrap assumptions \eqref{bootstrap 1} hold and that the energy defined in the previous section is sufficiently small.
\subsection{A discussion for initial data}\label{subsec: initial data}
We consider initial data $z_0(\alpha)=e^{i\alpha}+\epsilon f(\alpha)$ and $z_1(\alpha)=\epsilon g(\alpha)$ for the system \eqref{z final eq} such that $z_0$ is a simple closed curve containing the origin in the interior, parametrized counterclockwisely, and such that $(f,g)\in H_\alpha^s\times H^{s}_\alpha,~s\geq15.$ Furthermore, we assume

\Aligns{
\sup_{\alpha\neq \beta}|z_0(\alpha)-z_0(\beta)|\geq \lambda |e^{i\alpha}-e^{i\beta}|
}
for some $\lambda >0.$ We let $H_0$ be the Hilbert transform associated to the initial domain $\Omega(0)$ bounded by $z_0$ and $k_0(\alpha)=k(0,\alpha)$ be defined according to Remark~\ref{prop: k existence}. Using equation \eqref{z final eq} we can now uniquely determine initial values $z_2$ and $a_0$ for $z_{tt}$ and $a$ respectively. Here to get the initial value for $a$ one can for instance use the Riemann mapping formulation of the problem as discussed in Section \ref{sec: RM}. Alternatively one could use the double-layered potential as in \cite{Wu3}, see also \cite{Kenig 1}, \cite{Kenig 2} and \cite{Ver 1}. More precisely, let us write \eqref{z eq temp 1} as

\begin{align}\label{a initial temp 0}
i(a-\pi)\zbar_{\alpha}=\zbar_{tt}-i\pi(\zbar_{\alpha}+i\zbar)-\frac{\pi}{2}(I+H)\zbar.
\end{align}
Applying $(I-H)$ on both sides we obtain

\begin{align}\label{a initial temp 1}
i(I-H)\left((a-\pi)\zbar_{\alpha}\right)=(I-H)\zbar_{tt}-i\pi (I-H)(\zbar_{\alpha}+i\zbar).
\end{align}
Using the holomorphicity of $\zbar_{t}$ and multiplying both sides of \eqref{a initial temp 1} by $\frac{-iz_{\alpha}}{|z_{\alpha}|}$ then taking the real part, we get

\begin{align}\label{a initial temp 2}
(I+K^{*})\left((a-\pi)|z_{\alpha}|\right)=-\Re\left\{\frac{iz_{\alpha}}{|z_{\alpha}|}\left([z_{t},H]\frac{\zbar_{t\alpha}}{z_{\alpha}}-\pi i(I-H)(\zbar_{\alpha}+i\zbar)\right)\right\}.
\end{align}
Note that $z_{\alpha}=ie^{i\alpha}+\epsilon f_{\alpha}(\alpha)$. An argument similar to the proof of Lemma \ref{lem: at} using \eqref{a initial temp 0} and \eqref{a initial temp 2} implies that
%

\begin{align}\label{a L2}
\|a-\pi\|_{\La^{2}}\lesssim\|z_{\alpha}-iz\|_{\La^{2}}+\epsilon\|z_{t}\|_{\La^{2}}
\end{align}
if $\epsilon$ is small enough. The $H^{s}_{\alpha}$ estimate for $(a-\pi)$ can be derived similarly:

\begin{align}\label{a Hs}
\|(a-\pi)|z_{\alpha}|\|_{H^{s}_{\alpha}}\lesssim\|z_{\alpha}-iz\|_{H^{s}_{\alpha}}+\epsilon\|z_{t}\|_{H^{s}_{\alpha}}.
\end{align}
As for $a$ the initial value for $z_{tt}$ can be determined and estimated using the equation \eqref{z eq temp 1}

\begin{align}\label{ztt initial}
\zbar_{tt}=i(a-\pi)\zbar_{\alpha}+\pi i(\zbar_{\alpha}+i\zbar)+\frac{\pi}{2}(I+H)\zbar.
\end{align}

 Finally we let $k_1(\alpha)=\pt k (\alpha,0),$ where $k$ is extended using Theorem \ref{thm: lwp} and Remark~\ref{prop: k existence}. 

Our goal in this subsection is to prove the following proposition.
\begin{proposition}\label{prop: initial data}
Let $z_0,~z_1,~f,~g,~z_2,~k_0,~k_1,~a_0,$ and $H_0$ be defined as above and let $M_0>0$ and $\ell\in \N,~\ell\leq s-2$ be fixed constants. Then there exists $\epsilon_0>0,$ depending only on $\|f\|_{H_\alpha^s}$ and $\|g\|_{H_\alpha^{s}}$, such that if $\epsilon <\epsilon_0$ then $k_0$ is a diffeomorphism and 

\Align{\label{esti: kalpha}
\|k_{0,\alpha}-1\|_{\La^\infty}\leq \frac{1}{2},\quad \|k_{\alpha}-1\|_{H^{s-1}_{\alpha}}\lesssim\|z_{\alpha}-iz\|_{H^{s-1}_{\alpha}}.
}
Moreover, if $\epsilon<\epsilon_0$ and we define

\Align{\label{initial tq}
\zeta_0:=z_0\circ k_0^{-1},\quad \eta_0:=\pa\zeta_0-i\zeta_0,\quad u_0:=z_1\circ k_0^{-1}, \quad w_0:=z_2\circ k_0^{-1},
}
then

\Align{\label{initial bs}
\sum_{j\leq \ell}\left(\|\pa^j\eta_0\|_{\La^2}+\|\pa^ju_0\|_{\La^2}+\|\pa^j w_0\|_{\La^2}\right)\leq \frac{M_0}{2},\qquad |\zeta_0|^2\geq \frac{1}{2}.
}
Finally if we extend $z_0,~z_1$ to a local-in-time solution $(z,z_t)$ of \eqref{z final eq}, with the corresponding Hilbert transform $H,$ and we define $b_0:=k_1\circ k^{-1},~A_0=:(a_0\pa k_0)\circ k_0^{-1},$ and

\Aligns{
&\ep:=|z|^2-1,\quad \delta:=(I-H)\ep_0, \quad \chi:=\delta\circ k^{-1},\quad v=\delta_t\circ k^{-1}, \quad \tv=(I-H)v,
} 
then if $\epsilon <\epsilon_0$

\Align{\label{initial en}
\E(0):=\sum_{j\leq \ell}\left(\int_0^{2\pi}\frac{((\pt+b_0\pa)\pa^j\chi)\vert_{t=0}}{A_0}d\alpha+\int_0^{2\pi}\frac{((\pt+b_0\pa)\pa^j\tv)\vert_{t=0}}{A_0}d\alpha\right)\leq R_0\epsilon^2,
}
for a fixed $R_0>0$ independent of $\epsilon.$
\end{proposition}
\begin{proof}
Let the $F(\cdot)$ be the holomorphic function with the boundary value $\zbar_{0}e^{ik_{0}}$. Differentiating the equation $(I-H_{0})(\zbar_{0}e^{ik_{0}})=0$ with respect to $\alpha$ we get

\begin{align}\label{kalpha initial temp 1}
(I-H_{0})k_{0,\alpha}=i(I-H_{0})\frac{\zbar_{0,\alpha}+i\zbar_{0}}{\zbar_{0}}-i[z_{0,\alpha}-iz_{0},H_{0}]\frac{\pa(\log F)}{z_{0,\alpha}}.
\end{align}
On the other hand, for the initial data we have

\begin{align}\label{small initial energies}
\|z_{0,\alpha}-iz_{0}\|_{H^{s}_{\alpha}},\quad \|z_{1}\|_{H^{s}_{\alpha}},\quad \|z_{2}\|_{H^{s}_{\alpha}}\leq C_{0}\epsilon.
\end{align}
In fact, the first two estimates are straightforward from the construction of $z_{0}$ and $z_{1}$ and the last one follows from \eqref{a Hs} and \eqref{ztt initial}. Equation \eqref{small initial energies} together with
the relation

\begin{align}
k_{\alpha}-1=\frac{i\zbar_{\alpha}-\zbar}{\zbar}-i\pa\left(\log F\right)
\end{align}
implies that

\begin{align}\label{ave kalpha}
\AV(k_{\alpha}-1)\lesssim\|z_{\alpha}-iz\|_{H^{1}_{\alpha}}.
\end{align}
Here we used the fact that $\|\log F\|_{L^{\infty}_{\alpha}}$ is bounded by an absolute constant, which follows from the definition of $F$. Therefore writing $k_{0,\alpha}$ in terms of $\Re(I-H_{0})k_{0,\alpha}$ gives the desired estimate \eqref{esti: kalpha} for $k_{\alpha}$. The other statements of the proposition follow from \eqref{esti: kalpha}, the relation

\begin{align}\label{chain rule}
\pa\left(f\circ k^{-1}\right)=\frac{f_{\alpha}\circ k^{-1}}{k_{\alpha}\circ k^{-1}},
\end{align}
and arguments similar to those in Section \ref{sec: energy estimates}
\end{proof}

\subsection{Completion of the proof}\label{subsec: longtime exist}
In view of Proposition \ref{prop: initial data} the proof of long-time well-posedness will be complete once we prove the following theorem.

\begin{theorem}\label{long time existence theorem}
Let $z_0,~z_1$ be as in Proposition \ref{prop: initial data} and denote by $z(t,\alpha)$ the local-in-time solution of \eqref{z final eq}. Then there exist constant $M_0,~c,$ and $\epsilon_1$ such that if \eqref{initial bs} and \eqref{initial en} hold with $\epsilon <\epsilon_1$ then \eqref{z final eq} has a unique classical solution in $[0,\frac{c}{\epsilon^2}].$
\end{theorem}
\begin{proof}
Let $T^*>0$ be the maximal time of existence guaranteed by Theorem \ref{thm: lwp}. We want to show that $T^*\geq \frac{c}{\epsilon^2}$ for some $c$ independent of $\epsilon.$ Let $T\leq T^*$ be defined as 

\Aligns{
T:=\sup \left\{t\in[0,T^*)~ \vert~ k_\alpha(t,\alpha)>\frac{1}{100},~ \forall \alpha \in[0,2\pi]\right\}.
}
In particular $k$ is a diffeomorphism and continuous in time 
for all $t\leq T.$ Moreover, the energy $\E(t)$ defined in \eqref{total energy} is continuous in $[0,T].$ Next, define $T_{M_{0}}\leq T$ as

\Aligns{
T_{M_0}:=\sup\left\{ t\leq T~\vert~\sum_{j\leq \ell}\left(\|\pap^j\eta\|_{\Lap^2}+\|\pap^ju\|_{\Lap^2}+\|\pap^jw\|_{\Lap^2}\right)\leq M_0\right\},
}
and $T_\epsilon \leq T$ as

\Aligns{
T_\epsilon:=\sup\left\{ t\leq T~\vert~ \E^{\frac{1}{2}}(t)\leq 2 CR_0\epsilon \right\},
}
where $C$ is the constant in Proposition \ref{prop: final energy estimates}.
\subsubsection*{Step 1} We show that $T_\epsilon\leq T_{M_0},$ provided $\epsilon_1$ is sufficiently small. Indeed, if this is not the case then by Corollary \ref{cor: energy bounds} for all $t\in[0,T_{M_0}]$

\Aligns{
\sum_{j\leq \ell}\left(\|\pap^j\eta\|_{\Lap^2}+\|\pap^ju\|_{\Lap^2}+\|\pap^jw\|_{\Lap^2}\right)\leq C_1(M_0)R_0\epsilon,
}
and choosing $\epsilon_1\leq \frac{M_0}{2C_1(M_0)R_0}$ we get a contradiction with the maximality of $T_{M_0}.$

\subsubsection*{Step 2} We show that there exists a constant $c_1=c_1(M_0,R_0)$ such that if $\epsilon_1$ is sufficiently small and $T\leq T_0:=\frac{c_1}{\epsilon^2}$ then $T_\epsilon=T,$ and hence by the previous step $T_{M_0}=T_\epsilon=T.$ To see this, assume the contrary and first let $\epsilon_1$ be so small that the conclusion of the previous step holds. Then we can apply Proposition \ref{prop: final energy estimates} with $t=\frac{c_1}{\epsilon^2}\leq T_\epsilon,$ and conclude that if $\epsilon_1$ and $c_1$ are sufficiently small then $\E^{\frac{1}{2}}(t)\leq 2CR_0\epsilon$ proving the claim by contradiction. Here note that since $t\leq  T_\epsilon$ the last integral in the statement of Proposition \ref{prop: final energy estimates} can be bounded by $16c_1R_0^4C^4\epsilon^2<4C^2R_0^2\epsilon^2$ if $c_1$ is sufficiently small. 

\subsubsection*{Step 3} We show that there exists $c_2=c_2(M_0,R_0)$ such that if $\epsilon_1$ is sufficiently small and $T_1:=\frac{c_2}{\epsilon^2}\leq T_0$ then $k_\alpha\geq\frac{1}{100}$ for all $t\in[0,\min\{T^*,T_1\}).$ Suppose $\epsilon_1$ is small enough that the conclusions of the previous two steps hold. From the definition of $b$

\Aligns{
\pt k_\alpha=\left(b_\alphap\circ k\right) k_\alpha,
}
and hence 

\Align{\label{k rep}
k_\alpha(t,\alpha)=k_\alpha(0,\alpha)e^{\int_0^t\left(b_\alphap\circ k\right)(s,\alpha)ds}\geq k_\alpha(0,\alpha)e^{-\int_0^t\|b_\alphap\|_{\Lap^\infty}ds}.
}
But then by Proposition \ref{prop: b} and Corollary \ref{cor: energy bounds} if $t\leq \min\{T_1,T^*\}$ and $c_2$ is sufficiently small  it follows that

\Aligns{
k_{\alpha}(t,\alpha)\geq k_\alpha(0,\alpha)e^{-c_2C_2(M_0)}\geq \frac{1}{100}.
}

\subsubsection*{Step 4} Finally we show that $T^*\geq \frac{c}{\epsilon^2}$ for a sufficiently small constant $c.$ By Theorem \ref{thm: lwp} it suffices to show that if $T^*<\frac{c}{\epsilon^2}$ the $H_\alpha^{10}$ norms of $z_t$ and $z_{tt}$ remain bounded for $t<T^*$  and

\Align{\label{inv lip final}
\sup_{0\leq t<T^*}\sup_{\alpha\neq\beta}\abs{\frac{e^{i\alpha}-e^{i\beta}}{z(t,\alpha)-z(t,\beta)}}<\infty.
}
 Let $\epsilon_1$ be small enough that the conclusions of the previous steps hold, and let $c=c_2$ be as in Step 3. Then if $T^*<\frac{c_2}{\epsilon}, $ it follows from the previous three steps that $T_\epsilon=T_{M_0}=T=T^*.$  By Corollary \ref{cor: inverse Lip},  $|\zeta_\alphap(t,\alpha)|\geq \frac{1}{2}$ for all $t\leq T^*$ and all $\alpha\in[0,2\pi],$ and therefore combining with the fact that $k_\alpha\geq \frac{1}{100}$ we get \eqref{inv lip final}. Moreover, from the definition of $T_{M_0}$ the $H_\alphap^{10}$ norms of $u$ and $w$ are bounded up to $T^*,$ so by the chain rule, we only need to prove that the derivatives of $k$ up to order $10$ remain bounded for $t\in[0,T^*).$ But this follows from Proposition \ref{prop: b} and successive differentiation of the first indentity in \eqref{k rep}.

\end{proof}

\section{Riemann Mapping Coordinates and Local Well-Posedness}\label{sec: RM}
%

In this section we outline the proof Theorem \ref{thm: lwp} by investigating the quasilinear structure of the equation

\Align{\label{z eq RM}
\zbar_{tt}-ia\zbar_{\alpha}=-\frac{\pi}{2}(I-H)\zbar=-\pi \zbar+\frac{\pi}{2}(I+H)\zbar.
}
More precisely, we find a  quasilinear equation whose well-posedness implies that of equation \eqref{z eq RM}. This is achieved by differentiating \eqref{z eq RM} with respect to time and exploiting the holomorphicity of various quantities. Once the equivalent quasilinear system is found, the proof of well-posedness is standard and follows for instance from the vanishing-viscosity method in \cite{Wu1}. To avoid repetition we only prove the equivalence of \eqref{z eq RM} with a quasilinear equation and refer the reader to \cite{Wu1} for the details of the vanishing-viscosity method.

To get a quasilinear equation we differentiate \eqref{z eq RM} with respect to time, noting that $(I+H)\zbar$ is the boundary value of a holomorphic function in $\Omega(t)$, to get

\Align{\label{zt eq RM}
\zbar_{ttt}-ia\zbar_{t\alpha}=ia_t\zbar_\alpha+\frac{\pi}{2}[z_{t},H]\frac{\zbar_\alpha}{z_\alpha}, \quad H\zbar_t=\zbar_t.
}
Even though the proof of local existence for \eqref{zt eq RM} can be carried out in these coordinates, the structure of the equation will be more clear in Riemann mapping coordinates which we now introduce. Since we are interested in local existence, we fix a point $\bfx_0\in\Omega(0)$ and define the Riemann mapping for $t$ such that $\bfx_0\in\Omega(t).$ For such $t$ we define the Riemann mapping $\Phi(t.\cdot):\Omega(t)\to \D$ using the normalization $\Phi(t,\bfx_0)=0$ and $\Phi_z(t,\bfx_0)>0$ (in particular $\Phi_z(t,\bfx_0)$ is real). To $\Phi$ we associate the coordinate change $h:\R\to\R$ defined by $e^{ih(t,\alpha)}=\Phi(t,z(t,\alpha)).$ Alternatively let $\chi_1(\cdot)=z(0,\cdot):[0,2\pi]\to\partial\Omega(0)$ be the parametrization of the initial boundary and extend the definition of $\chi_{1}(\cdot)$ to $\bbR$ periodically. Similarly let $X(t,\cdot):\Omega(0)\to\Omega(t)$ denote the flow of the velocity vector field, that is, $\dot X(t,\cdot)=\bfv(t,X(t,\cdot))$. Finally let $\chi_2(\cdot):=-i\log(\cdot):\partial\D\to \bbR$ be the inverse parametrization of the boundary of the unit disc. In this notation $h$ is the composition change of variables $h:=\chi_2\circ\Phi\circ X\circ\chi_1,$

\begin{center}
\begin{tikzcd}\label{fig: h}
\bbR
\arrow[bend right]{rrrrrr}[swap]{h(t,\cdot)}
\arrow{rr}{\chi_1(\cdot):=z(0,\cdot)}
&& \partial\Omega(0)\subseteq \Omega(0)  \arrow{r}{X(t,\cdot)} & \Omega(t) \arrow{r}{\Phi(t,\cdot)}
& \D \supseteq \partial\D\arrow{rr}{\chi_2(\cdot)=-i\log(\cdot)} && \bbR
\end{tikzcd}
\end{center}
and the new unknowns in Riemann mapping coordinates are

\Aligns{
&Z(t,\alphap):=z(t,h^{-1}(t,\alphap)),\quad Z_t(t,\alphap):=z_t(t,h^{-1}(t,\alphap)),\\
&Z_{tt}(t,\alphap)=z_{tt}(t,h^{-1}(\alphap)),\quad Z_{ttt}(t,\alphap):=z_{ttt}(t,h^{-1}(t,\alphap)).
}
To avoid confusion we separate the subscripts corresponding to partial differentiation by a comma, so for instance $Z_{,\alphap}(t,\alphap)=\pap Z(t,\alphap).$ We denote by $\H$ the Hilbert transform on the circle which can be written as

\Align{\label{H circle}
\H f(\alphap):=\frac{\pv}{\pi i}\int_0^{2\pi}\frac{f(\betap)}{e^{i\betap}-e^{i\alphap}}ie^{i\betap}d\betap=&\frac{\pv}{2\pi i}\int_0^{2\pi}f(\betap)\cot\left(\frac{\betap-\alphap}{2}\right)d\betap+\frac{1}{2\pi}\int_0^{2\pi}f(\betap)d\betap\\
=&\,\tH f(\alphap)+\Av(f),
}
where
\Aligns{
&\tH f(\alphap):=\frac{\pv}{2\pi i}\int_0^{2\pi}f(\betap)\cot\left(\frac{\betap-\alphap}{2}\right)d\betap,\quad \Av(f):=\frac{1}{2\pi}\int_0^{2\pi}f(\betap)d\betap.
}
For notational convenience we also introduce the following new variables and operators in Riemann mapping coordinates:

\Aligns{
\A:=(a h_\alpha)\circ h^{-1},\quad G:=\frac{\pi}{2}((I+H)\zbar)\circ h^{-1}
}
and

\Aligns{
\CH f(\alphap):=\frac{\pv}{\pi i}\int_0^{2\pi}\frac{f(\betap)}{Z(t,\betap)-Z(t,\alphap)}Z_{,\betap}(t,\betap)d\betap.
}
With this notation, precomposing with $h^{-1}(t,\cdot)$ we can rewrite equations \eqref{z eq RM} and \eqref{zt eq RM} as

\Align{\label{Z eq RM}
\Zbar_{tt}-i\A{\Zbar_{,\alphap}}=-\pi \Zbar+G
} 
and 

\begin{align}\label{Zt eq RM}
\Zbar_{ttt}-i\A{\Zbar_{t,\alphap}}=i\frac{a_t}{a}\circ h^{-1}\A{\Zbar_{,\alphap}}+\frac{\pi}{2}[Z_t,\CH]{\frac{\Zbar_{,\alphap}}{Z_{,\alphap}}}, \quad \H\Zbar_t=\Zbar_t.
\end{align}
Note that if we let 

\Aligns{
B:=h_t\circ h^{-1}
}
we can rewrite \eqref{Zt eq RM} as

\Align{\label{Quasi temp}
\left((\pt+B\pap)^2-i\A\pap\right)\Zbar_t=i\frac{a_t}{a}\circ h^{-1}\A{{\Zbar_{,\alphap}}}+\frac{\pi}{2}[Z_t,\CH]{\frac{\Zbar_{,\alphap}}{Z_{,\alphap}}}.
}
To understand the quasilinear structure of this equation we need to compute $\A,$ $B,$ and $\frac{a_t}{a}\circ h^{-1}$ in terms of the unknowns. We begin with $\A,$ where in addition we verify that $\A$ is in fact a positive quantity so that the Taylor sign condition holds. 

\begin{proposition}\label{prop: A1} $\A_1:=\A|{Z_{,\alphap}}|^2$ is positive and is given by
\Aligns{
\A_1=&\Im[Z_t,\tH]{\Zbar_{t,\alphap}}+\pi\Im[\Zbar,\tH]{Z_{,\alphap}}\\
=&\frac{1}{8\pi}\int_{0}^{2\pi}|Z(t,\beta')-Z(t,\alpha')|^{2}\csc^{2}\left(\frac{\betap-\alphap}{2}\right)d\beta'\\
&+\frac{1}{8}\int_{0}^{2\pi}|Z_{t}(t,\beta')-Z_{t}(t,\alpha')|^{2}\csc^{2}\left(\frac{\betap-\alphap}{2}\right)d\beta'>0.
}
\end{proposition}
\begin{proof}
Multiplying \eqref{Z eq RM} by ${Z_{,\alphap}}$ we get

\Align{\label{A temp 1 RM}
i\A_1=i\A|{Z_{,\alphap}}|^2=\Zbar_{tt}{Z_{,\alphap}}+\pi\Zbar {Z_{,\alphap}}- G{Z_{,\alphap}}.
}
Note that since $\Phi(t,Z(t,\alphap))=e^{i\alphap}$ and $\Phi_{z}$ is non-vanishing,

\Aligns{
{Z_{,\alphap}}=\frac{ie^{i\alphap}}{\Phi_{z}(t,Z)}
}
is holomorphic inside $\D.$ Moreover writing $\zbar_t=F(t,z)$ where $F$ is holomorphic inside $\Omega(t)$ we get

\Aligns{
\zbar_{tt}=F_t(t,z)+F_z(t,z)z_t=F_t(t,z) + \frac{\zbar_{t\alpha}z_t}{z_\alpha},
}
and hence

\Aligns{
\Zbar_{tt}=F_t(t,Z)+\frac{{\Zbar_{t,\alphap}}Z_t}{{Z_{,\alphap}}}.
}
Therefore we can apply $(I-\H)$ to \eqref{A temp 1 RM} to get

\Aligns{
i(I-\H)\A_1=(I-\H)({\Zbar_{t,\alphap}}Z_t)+\pi(I-\H)(\Zbar {Z_{,\alphap}}).
}
Taking imaginary parts of the two sides, keeping in mind that $\A_{1}$ is real, yields

\Align{\label{A temp 2 RM}
\A_1-\Av(\A_1)=\Im[Z_t,\H]{\Zbar_{t,\alphap}}+\pi\Im[\Zbar,\H]{Z_{,\alphap}}.
}
Note that from \eqref{A temp 1 RM}

\Align{\label{A temp 2 p RM}
\Av(\A_1)=-i \Av({\Zbar_{t,\alphap}}Z_t)-\pi i\Av(\Zbar {Z_{,\alphap}}).
}
Also

\Aligns{
[\Zbar,\H]{Z_{,\alphap}}=&\Zbar\tH {Z_{,\alphap}}+\Zbar \Av({Z_{,\alphap}})-\tH(\Zbar {Z_{,\alphap}})-\Av(\Zbar {Z_{,\alphap}})\\
=&[\Zbar,\tH]{Z_{,\alphap}}-\Av(\Zbar {Z_{,\alphap}})
}
and

\Aligns{
[Z_t,\H]{\Zbar_{t,\alphap}}=[Z_t,\tH]{\Zbar_{t,\alphap}}-\Av({\Zbar_{t,\alphap}}Z_t).
}
Using the fact that $\Av(\Zbar {Z_{,\alphap}})$ and $\Av({\Zbar_{t,\alphap}}Z_t)$ are purely imaginary, these computations and \eqref{A temp 2 RM} combine to give 

\Aligns{
\A_1=\Im[Z_t,\tH]{\Zbar_{t,\alphap}}+\pi\Im[\Zbar,\tH]{Z_{,\alphap}}.
}
To see that the right hand side above is positive note that

\begin{align*}
\Im([\overline{Z},\tH]{Z_{,\alphap}})&=-\frac{1}{2\pi}\Re\int_{0}^{2\pi}\left(\overline{Z}(t,\alpha')-\overline{Z}(t,\beta')\right)\cot\left(\frac{\betap-\alphap}{2}\right)\partial_{\beta'}\left(Z(t,\beta')-Z(t,\alpha')\right)d\beta'\\
&=\frac{1}{4\pi}\int_{0}^{2\pi}\cot\left(\frac{\betap-\alphap}{2}\right)\partial_{\beta'}|Z(t,\beta')-Z(t,\alpha')|^{2}d\beta'\\
&=\frac{1}{8\pi}\int_{0}^{2\pi}|Z(t,\beta')-Z(t,\alpha')|^{2}\csc^{2}\left(\frac{\betap-\alphap}{2}\right)d\beta'>0
\end{align*}
and

\begin{align*}
\begin{split}
\Im[Z_t,\tH]{\Zbar_{t,\alphap}}&=-\frac{1}{2\pi}\Re\int_{0}^{2\pi}\left(Z_{t}(t,\alpha')-Z_{t}(t,\beta')\right)\cot\left(\frac{\betap-\alphap}{2}\right)\partial_{\beta'}\left(\overline{Z}_{t}(t,\beta')-\overline{Z}_{t}(t,\alpha')\right)d\beta'\\
&=\frac{1}{4\pi}\int_{0}^{2\pi}\partial_{\beta'}|Z_{t}(t,\beta')-Z_{t}(t,\alpha')|^{2}\cot\left(\frac{\betap-\alphap}{2}\right)d\beta'\\
&=\frac{1}{8\pi}\int_{0}^{2\pi}|Z_{t}(t,\beta')-Z_{t}(t,\alpha')|^{2}\csc^{2}\left(\frac{\betap-\alphap}{2}\right)d\beta'>0.
\end{split}
\end{align*}
\end{proof}
The computation for $\frac{a_t}{a}\circ h^{-1}$ is more involved. In order to state the result we introduce the notation

\Align{\label{D notation}
D_\alpha:=\frac{1}{z_\alpha}\partial_\alpha,\qquad D_{\alphap}:=\frac{1}{{Z_{,\alphap}}}\partial_{\alphap}
}
We also define
\begin{align*}
[Z_{t},Z_{t};D_{\alphap}\Zbar_{t}]:=&-[Z_{t}^{2},\Hp]\partial_{\alphap}D_{\alphap}\Zbar_{t}+2[Z_{t},\Hp]\partial_{\alpha'}((D_{\alphap}\Zbar_{t})Z_{t})\\
&=-\frac{ie^{i\alphap}}{\pi}\int_{0}^{2\pi}\left(\frac{Z_t(t,\betap)-Z_t(t,\alphap)}{e^{i\betap}-e^{i\alphap}}\right)^2\frac{e^{i\betap}}{Z_\betap(t,\betap)}{\Zbar_{t,\betap}}(t,\betap)d\beta'.\\
\end{align*}
With this notation we state our next proposition. 

\begin{proposition}\label{prop: at over a}
\Aligns{
\frac{a_t}{a}\circ h^{-1}-\frac{1}{\A_1}\Av\left(\A_1\frac{a_t}{a}\circ h^{-1}\right)=\frac{1}{\A_1}\Im\Bigg\{&-\frac{\pi}{2}\left[[Z_{t},\cH]\frac{{\Zbar_{,\alphap}}}{{Z_{,\alphap}}},\H\right]{Z_{,\alphap}}-\frac{\pi}{2}[(I+\CHbar)Z,\H]{\Zbar_{t,\alphap}}\\
&+2[Z_{t},\H]{\Zbar_{tt,\alphap}}+2[Z_{tt},\H]{\Zbar_{t,\alphap}}-[Z_t,Z_t;D_\alphap \Zbar_t]\Bigg\},
}
and

\begin{align*}
\Av\left(\A_{1}\frac{a_{t}}{a}\circ h^{-1}\right)=&-2i\Av\left(Z_{t}\partial_{\alphap}(\Zbar_{tt}-(D_{\alphap}\Zbar_{t})Z_{t})\right)-i\Av\left(Z_{tt}\pap\Zbar_t\right)-i\Av\left(Z_t^2\pap D_\alphap \Zbar_t\right)\\
&-i\Av\left({\Zbar_{t,\alphap}}Z_{tt}\right)+\frac{\pi i}{2}\Av\left({\Zbar_{t,\alphap}}(I+\CHbar)Z\right)+\frac{\pi i}{2}\Av\left({Z_{,\alphap}} [Z_t,\CH]\frac{{\Zbar_{,\alphap}}}{{Z_{,\alphap}}}\right).
\end{align*}
\end{proposition}
\begin{proof}
Multiplying \eqref{Zt eq RM} by ${Z_{,\alphap}}$ gives

\Align{\label{A1 temp 1}
{Z_{,\alphap}}\left(\Zbar_{ttt}-i\A{\Zbar_{t,\alphap}}-\frac{\pi}{2}[Z_t,\CH]\frac{{\Zbar_{,\alphap}}}{{Z_{,\alphap}}}\right)=i\A_1\frac{a_{t}}{a}\circ h^{-1}.
}
In order to understand the holomorphicity properties of $\Zbar_{ttt}$ we recall that

\Align{\label{holomorphicity zbart}
&\zbar_t(t,\alpha)=F(t,z(t,\alpha)),\\
&\zbar_{tt}(t,\alpha)=F_t(t,z(t,\alpha))+F_z(t,z(t,\alpha))z_t(t,\alpha),\\
&\zbar_{ttt}=F_{tt}(t,z(t,\alpha))+2F_{tz}(t,z(t,\alpha))z_t(t,\alpha)+F_z(t,z(t,\alpha))z_{tt}(t,\alpha)+F_{zz}(t,z(t,\alpha))z_t^2(t,\alpha),\\
&\zbar_{t\alpha}=F_{z}(t,z(t,\alpha))z_\alpha(t,\alpha).
}
With the notation introduced in \eqref{D notation}, 

\Aligns{
F_t\circ z=\zbar_{tt}-(D_\alpha\zbar_t)z_t, \quad F_z\circ z=D_\alpha \zbar_t,\quad {F_{zz}\circ z=D_\alpha^2\zbar_t,}\quad F_{tz}\circ z =D_\alpha(\zbar_{tt}-(D_\alpha \zbar_t)z_t)
}
where the lasts identity follows form differentiating the first with respect to $\alpha.$
Substituting back into the equation for $\zbar_{ttt}$ we get

\Aligns{
\zbar_{ttt}=F_{tt}\circ z+2z_tD_\alpha(\zbar_{tt}-(D_\alpha\zbar_t)z_t)+z_{tt}D_\alpha \zbar_t+z_t^2D_\alpha^2\zbar_t.
}
We now precompose with $h^{-1}$ to get

\Align{\label{A1 temp 2}
\Zbar_{ttt}=F_{tt}\circ Z+2Z_tD_\alphap(\Zbar_{tt}-(D_\alphap\Zbar_t)Z_t)+Z_{tt}D_\alphap \Zbar_t+Z_t^2D_\alphap^2\Zbar_t.
}
We will substitute this into \eqref{A1 temp 1} and apply $(I-\H).$ To this end we first note that if $f$ is holomorphic then since $Z_{,\alphap}$ is also holomorphic, $(I-\Hp)(Z_{,\alphap} f)=0,$ which allows us to compute

\begin{align*}
&(I-\Hp)({Z_{,\alphap}}F_{tt}\circ Z)=0,\quad (I-\Hp)({Z_{,\alphap}}\Zbar_{t})=0,\\
&(I-\Hp)(Z_{t}\partial_{\alphap}(\Zbar_{tt}-(D_{\alphap}\Zbar_{t})Z_{t})=[Z_{t},\Hp]\partial_{\alphap}(\Zbar_{tt}-(D_{\alphap}\Zbar_{t})Z_{t}),\\
&(I-\Hp)(Z_{tt}\partial_{\alphap}\Zbar_{t})=[Z_{tt},\Hp]\partial_{\alphap}\Zbar_{t},\\
&(I-\Hp)(Z_{t}^{2}\partial_{\alpha'}D_{\alpha'}\Zbar_{t})=[Z_{t}^{2},\Hp]\partial_{\alpha'}D_{\alpha'}\Zbar_{t},
\end{align*}

so
\Aligns{
(I-\H)({Z_{,\alphap}}\Zbar_{ttt})=2[Z_{t},\H]\partial_{\alphap}(\Zbar_{tt}-(D_{\alphap}\Zbar_{t})Z_{t})+[Z_{tt},\H]\partial_{\alphap}\Zbar_{t}+[Z_{t}^{2},\H]\partial_{\alphap}D_{\alphap}\Zbar_{t}.
}
In view of \eqref{Z eq RM} and holomorphicity of $Z\Zbar_{t,\alphap},$

\begin{align*}
-i(I-\H)(\A{\Zbar_{t,\alphap}}{Z_{,\alphap}})=&(I-\H)\left(\Zbar_{t,\alphap}(Z_{tt}+\pi Z-\Gbar)\right)\\
=&[Z_{tt},\H]{\Zbar_{t,\alphap}}-[\Gbar,\H]{\Zbar_{t,\alphap}}.
\end{align*}
Moreover using Lemma \ref{lem: operator H commutator},

\begin{align*}
-\frac{\pi}{2}(I-\H)({Z_{,\alphap}}\partial_{t}(I-\cH)\Zbar)={\frac{\pi}{2}}\left[[Z_{t},\cH]\frac{{\Zbar_{,\alphap}}}{{Z_{,\alphap}}},\H\right]{Z_{,\alphap}}.
\end{align*}
Putting these together and using the notation introduced before the proposition we get

\begin{align*}
i(I-\H)(\A_{1}\frac{a_{t}}{a}\circ h^{-1})=&-\frac{\pi}{2}\left[[Z_{t},\cH]\frac{{\Zbar_{,\alphap}}}{{Z_{,\alphap}}},\H\right]{Z_{,\alphap}}-\frac{\pi}{2}[(I+\CHbar)Z,\H]{\Zbar_{t,\alphap}}\\
&+2[Z_{t},\H]{\Zbar_{tt,\alphap}}+2[Z_{tt},\H]{\Zbar_{t,\alphap}}-[Z_t,Z_t;D_\alphap \Zbar_t].
\end{align*}
The first statement of the proposition now follows by taking imaginary parts on both sides of this equation. The second statement follows from taking the averages of the two sides of \eqref{A1 temp 1} and using \eqref{A1 temp 2} and \eqref{Z eq RM} as well as the facts that by the holomorphicity of $F$ and $F_{tt}$ 

$$\Av({Z_{,\alphap}} F_{tt}\circ Z )==\frac{1}{2\pi}\int_{\partial\Omega}F_{tt} dz=0$$
and

$$\Av({\Zbar_{t,\alphap}}Z)=-\frac{1}{2\pi}\int_0^{2\pi}\Zbar_t Z_{,\alphap}d\alphap=\frac{1}{2\pi}\int_{\partial\Gamma}F dz=0.$$
\end{proof}
Finally we turn to $B:=h_t\circ h^{-1}.$

\begin{proposition}\label{prop: B}
Suppose the Riemann mapping $\Phi$ satisfies $\Phi(t,\bfx_0)=0,$ $\Phi_z(t,\bfx_0)>0$ for all $t\leq T,$ where $T$ is such that $\bfx_0\in\Omega(t)$ for $t\leq T.$  Then $B$ satisfies

\Aligns{
B-\Av(B)= \Re\left(\left[\frac{Z_{t}}{e^{i\alphap}},\H\right]\frac{e^{i\alphap}}{{Z_{,\alphap}}}\right),
}
and 

\Aligns{
\Av(B)=\frac{1}{2\pi}\Re\int_{0}^{2\pi}\frac{Z_{t}}{{Z_{,\alphap}}}d\alphap.
}
\end{proposition}
\begin{proof}
Differentiating the equation $\Phi(t,{Z(t,\alphap)})=e^{i\alphap}$ with respect to $t$ gives

\Aligns{
0=\Phi_t\circ Z +\Phi_z\circ Z (Z_t-B{Z_{,\alphap}})=\Phi_t\circ Z+\frac{ie^{i\alphap}(Z_t-B{Z_{,\alphap}})}{{Z_{,\alphap}}}
}
which can be rearranged as

\Aligns{
B=\frac{\Phi_t\circ Z}{ie^{i\alphap}}+\frac{Z_t}{{Z_{,\alphap}}}.
}
Since $\Phi(t,\bfx_{0})=0$ for all $t\in[0,T]$, applying $(I-\H)$ gives

\Aligns{
(I-\H)B=(I-\H)\left(\frac{Z_{t}}{{Z_{,\alphap}}}\right)=\left[\frac{Z_{t}}{e^{i\alphap}},\H\right]\frac{e^{i\alphap}}{{Z_{,\alphap}}}.
}
Taking the real parts on both sides of above gives

\Aligns{
B-\Av(B)=\Re\left(\left[\frac{Z_{t}}{e^{i\alphap}},\H\right]\frac{e^{i\alphap}}{{Z_{,\alphap}}}\right).
}
Note that

\Aligns{
\frac{1}{2\pi i}\int_{0}^{2\pi}\frac{\Phi_{t}(t,\Phi^{-1}(t,e^{i\alphap}))}{ie^{2i\alphap}}ie^{i\alphap}d\alphap=\frac{1}{i}\Phi_{tz}(t,\bfx_{0}){(\Phi^{-1})_{w}(t,0)}
}
is purely imaginary by our choice of normalization for the Riemann mapping. Since $B$ is real, it follows that

\Aligns{
\Av(B)=\frac{1}{2\pi}\Re\int_{0}^{2\pi}\frac{Z_{t}}{{Z_{,\alphap}}}d\alphap.
}
\end{proof}

Summarizing the computations above, we get the following corollary of \eqref{Z eq RM}, \eqref{Quasi temp},  and Propositions \ref{prop: A1},  \ref{prop: at over a}, and \ref{prop: B}.

\begin{corollary}\label{cor: Quasi}
If $z$ is a solution to \eqref{z eq RM} and the Riemann mapping $\Phi$ is defined according to the normalization above, then $Z:=z\circ h^{-1}$ satisfies
\begin{align}\label{Quasi 1}
\begin{cases}
&(\partial_{t}+B\partial_{\alphap})^{2}\Zbar_{t}-i\A\partial_{\alphap}\Zbar_{t}=i\dfrac{a_t}{a}\circ h^{-1}\dfrac{\A_{1}}{{Z_{,\alphap}}}+\dfrac{\pi}{2}[Z_t,\CH]\dfrac{{\Zbar_{,\alphap}}}{{Z_{,\alphap}}}:=g\\
&\Zbar_{t}=\H\Zbar_{t}
\end{cases},
\end{align}
where

\begin{align}\label{Quasi 2}
\begin{cases}
&\A=\dfrac{|\Zbar_{tt}+\dfrac{\pi}{2}(I-\cH)\Zbar|^{2}}{\A_{1}},\quad \Zbar_{tt}=(\partial_{t}+B\partial_{\alphap})\Zbar_{t}\\
&\dfrac{1}{{Z_{,\alphap}}}=\dfrac{\Zbar_{tt}+\dfrac{\pi}{2}(I-\cH)\Zbar}{i\A_{1}}
\end{cases},
\end{align}
and $\A_{1}, ~\dfrac{a_{t}}{a}\circ h^{-1},$ and $B$ are given in Propositions \ref{prop: A1}, \ref{prop: at over a}, and \ref{prop: B} respectively.

\end{corollary}
\begin{remark}
The significance of \eqref{Quasi 2} is that in proving local well-posedness for \eqref{Quasi 1} we will use \eqref{Quasi 2} as the \emph{definition} of $\A$ and $\frac{1}{{Z_{,\alphap}}}.$ As we will discuss below, we will separately show that the resulting solution is a solution of the original system \eqref{Z eq RM}.
\end{remark}
We have now seen how to go from the original system to 

\begin{align}\label{Actual eq 1}
\begin{cases}
&(\partial_{t}+B\partial_{\alphap})^{2}V+\A|D|V=\dfrac{a_t}{a}\circ h^{-1}L+\dfrac{\pi}{2}[\Vbar,\CH]\dfrac{\Wbar_{\alphap}}{W_\alphap}=:g\\
&(\pt+B\pap)W=\Vbar
\end{cases},
\end{align}
where

\begin{align}\label{Actual eq 2}
\begin{cases}
&B-\Av(B)=\Re\left[\frac{\Vbar}{e^{i\alphap}},\H\right]\frac{e^{i\alphap}L}{i\A_{1}},\\
&\Av(B)=\frac{1}{2\pi}\Re\int_{0}^{2\pi}\frac{\Vbar L}{i\A_{1}}d\alphap,\\
&\A=\dfrac{|(\pt+B\pap)V+\dfrac{\pi}{2}(I-\cH)\Wbar|^{2}}{\A_{1}},\quad \\
&L=(\partial_{t}+B\pap)V+\frac{\pi}{2}(I-\CH)\Wbar,
\end{cases}
\end{align}
and $\A_{1}$ and $\frac{a_{t}}{a}\circ h^{-1}$ are defined in Propositions \ref{prop: A1} and \ref{prop: at over a} with $Z,~\Zbar_{t},$ and $\Zbar_{tt}$ replaced by $W,~V,$ and $(\pt+B\pap)V$ respectively.
Here $W=Z, V=\Zbar_{t}$, $|D|=\sqrt{-\pap^{2}}$ and we have used the fact that if $u$ is the boundary value of a holomorphic function in the disc, then $|D|u=-i\pap \H u=-i\pap u$. We next discuss how to go back to the original water wave system from \eqref{Actual eq 1}--\eqref{Actual eq 2}.

\begin{proposition}\label{prop: going back}
Suppose $(W,V)$ is a solution to \eqref{Actual eq 1} and \eqref{Actual eq 2} on some time interval $J$ extending from $t=0$ such that $\bfx_0\in\Omega(t)$ for all $t\in J.$ Then the following statements hold.

\begin{enumerate} 
\item  $W$ and $V$ are boundary values of holomorphic functions and $L-\frac{i\A_{1}}{W_{\alphap}}=0$, if initially  $W$ and $V$ are boundary values of holomorphic functions and $L-\frac{i\A_{1}}{W_{\alphap}}=0$. 
\item If $h$ is the solution to

\begin{align*}
\frac{dh}{dt}=B(h,t),\quad h(0,\alpha)=\alpha,
\end{align*}
then $z:=W\circ h$ satisfies \eqref{z eq RM}.
\end{enumerate}
\end{proposition}
\begin{proof}
\begin{enumerate}
\item 

We will derive a linear differential system for the quantities $(I-\H)V$, $(I-\H)W$ and $\frac{i\A_{1}}{W_{\alphap}}-L$ for which uniqueness of solutions holds. Since these quantities are zero initially, they must be zero during the evolution. In this process, we will use $\calR$ to denote linear terms in these quantities, whose exact definition may change from line to line. If we want to make the dependence precise, we use expressions such as $\calR((I-\H)V, (I-\H)W,...)$. We start with the equation for $W$. Applying $(I-\H)$ on both sides we get

\begin{align}\label{homo 1}
\begin{split}
\partial_{t}\left((I-\H)W\right)+B\partial_{\alphap}\left((I-\H)W\right)
=&-[B,\H]W_{\alphap}+(I-\H)\Vbar\\
=&-[B,\H]\frac{(I+\H)W_{\alphap}}{2}+(I-\H)\Vbar+\calR_{1}\\
=&-\left[\frac{I-\H}{2}B,\H\right]\frac{(I+\H)}{2}W_{\alphap}+(I-\H)\Vbar+\calR_{1}\\
=&(I-\H)\Vbar-\left[\frac{I-\H}{2}\Re[\Vbar e^{-i\alphap},\H]\frac{e^{i\alphap}}{W_{\alphap}},\H\right]\frac{I+\H}{2}W_{\alphap}+\calR_{1}+\calR_{2}\\
=&(I-\H)\Vbar-\left[\frac{I-\H}{2}\Re[\Vbar e^{-i\alphap},\H]\frac{e^{i\alphap}}{\frac{I+\H}{2}W_{\alphap}},\H\right]\frac{I+\H}{2}W_{\alphap}+\calR_{1}+\calR_{2}+\calR_{3}\\
=&(I-\H)\Vbar-\left[\frac{I-\H}{2}\left(\frac{\Vbar}{\frac{I+\H}{2}W_{\alphap}}\right),\H\right]\frac{I+\H}{2}W_{\alphap}+\calR_{1}+\calR_{2}+\calR_{3}\\
=&(I-\H)\Vbar-\left[\frac{\Vbar}{\frac{I+\H}{2}W_{\alphap}},\H\right]\frac{I+\H}{2}W_{\alphap}+\calR_{1}+\calR_{2}+\calR_{3}\\
=&\calR_{1}+\calR_{2}+\calR_{3}.
\end{split}
\end{align}
where 

\begin{align}\label{W errors}
\begin{split}
&\calR_{1}=-[B,\H]\frac{(I-\H)}{2}W_{\alphap},\\
&\calR_{2}=-\left[\frac{I-\H}{2}\Re[\Vbar e^{-i\alphap},\H]e^{i\alphap}\left(\frac{L}{i\A_{1}}-\frac{1}{W_{\alphap}}\right),\H\right]\frac{I+\H}{2}W_{\alphap},\\
&\calR_{3}=-\left[\frac{I-\H}{2}\Re[\Vbar e^{-i\alphap},\H]\left(\frac{e^{i\alphap}}{W_{\alphap}}-\frac{e^{i\alphap}}{\frac{I+\H}{2}W_{\alphap}}\right),\H\right]\frac{I+\H}{2}W_{\alphap}.
\end{split}
\end{align}
Note that in view of Lemma \ref{lem: Yosihara}, 

\begin{align}\label{esti R1-3}
\|\calR_{j}\|_{H^{s}_{\alphap}}\leq C\left(\left\|L-\frac{i\A_{1}}{W_{\alphap}}\right\|_{L^{2}_{\alphap}}+\|(I-\H)W\|_{L^{2}_{\alphap}}\right).
\end{align}
To derive an equation for $(I-\H)V$, we introduce the notation $\calP:=(\pt+B\pap)^{2}+\A|D|$. Then the first equation in \eqref{Actual eq 1} can be written as

\begin{align}\label{homo 2 pre}
\calP\left(\frac{I-\H}{2}V\right)=-\calP\left(\frac{I+\H}{2}V\right)+e^{-i\alphap}\frac{ie^{i\alphap}}{\frac{I+\H}{2}W_{\alphap}}\A_{1}\frac{a_{t}}{a}\circ h^{-1}+\frac{\pi}{2}[\Vbar,\calH]\frac{\Wbar_{\alphap}}{W_{\alphap}}+\calR_{4}
\end{align}
where

\begin{align}\label{V error}
\begin{split}
\calR_{4}=&\left((\pt+B\pap)\left(\frac{I-\H}{2}V\right)+\pi\frac{I-\overline{\bbH}}{2}\Wbar+\frac{\pi}{2}(\tcalH-\calH)\Wbar\right)\frac{a_{t}}{a}\circ h^{-1}\\
&+e^{-i\alphap}\left(\frac{\tL e^{i\alphap}}{i\A_{1}}-\frac{e^{i\alphap}}{\frac{I+\H}{2}W_{\alphap}}\right)i\A_{1}\left(\frac{a_{t}}{a}\circ h^{-1}\right).
\end{split}
\end{align}
Here

\begin{align*}
&\tL=(\pt+B\pap)\left(\frac{I+\H}{2}V\right)+\frac{\pi}{2}\PHbar\Wbar-\frac{\pi}{2}(I+\tcalH)\Wbar,\\
&(\tcalH f)(\alphap):=\frac{\pv}{\pi i}\int_0^{2\pi}\frac{f(\betap)}{\PH W(\betap)-\PH W(\alphap)}\PH W_\betap(\betap)d\betap.
\end{align*}
Note that

\begin{align*}
&(\tcalH-\calH)f=\calR((I-\H)W, (I-\H)W_{\alphap}),\\
 &L-\tL=\calR((I-\H)V, (I-\H)W, (I-\H)W_{\alphap}, (\pt+B\pap)(I-\H)V).
\end{align*}
\begin{claim}\label{claim: gain regularity}
Given any $f\in H^{s}_{\alphap}$, there is a constant $C=C\left(\|f\|_{H^{s}_{\alphap}}\right)$, such that
\begin{align*}
&\left\|(I-\H)f\left(\frac{\tL e^{i\alphap}}{i\A_{1}}-\frac{e^{i\alphap}}{\frac{I+\H}{2}W_{\alphap}}\right)\right\|_{H^{s}_{\alphap}}\\
\leq &C\left(\left\|L-i\A\Wbar_{\alphap}\right\|_{L^{2}_{\alphap}}+\|(I-\H)V\|_{H^{s}_{\alphap}}+\|(\pt+B\pap)(I-\H)V\|_{L^{2}_{\alphap}}+\|(I-\H)W\|_{H^{s}_{\alphap}}\right).
\end{align*}
\end{claim}
\begin{proof}[Proof of Claim \ref{claim: gain regularity}]
First we compute

\begin{align*}
\A_{1}(I-\H)\left(\frac{\tL e^{i\alphap}}{\A_{1}}\right)=&(I-\H)(e^{i\alphap}\tL)+[\H,\A_{1}]\left(\frac{\tL e^{i\alphap}}{\A_{1}}\right)\\
=&(I-\H)(\pt+B\pap)e^{i\alphap}\left(\frac{I+\H}{2}V\right)-(I-\H)\left(iBe^{i\alphap}\frac{I+\H}{2}V\right)\\
&+\pi(I-\H)e^{i\alphap}\Wbar+[\H,\A_{1}]\left(\frac{ie^{i\alphap}}{\frac{I+\H}{2}W_{\alphap}}\right)+\calR\\
=&[B,\H]\left(e^{i\alphap}\frac{I+\H}{2}V_{\alphap}\right)+\pi(I-\H)e^{i\alphap}\Wbar-\left[\frac{i(I-\H)}{2}\A_{1},\H\right]\left(\frac{e^{i\alphap}}{\frac{I+\H}{2}W_{\alphap}}\right)+\calR\\
=&\left[\frac{I-\H}{2}B,\H\right]\left(\frac{I+\H}{2}V_{\alphap}\cdot e^{i\alphap}\right)+\pi(I-\H)e^{i\alphap}\frac{I+\overline{\H}}{2}\Wbar\\
&-\frac{1}{2}\left[(I-\H)\PHbar\Vbar\PH V_{\alphap},\H\right]\left(\frac{e^{i\alphap}}{\frac{I+\H}{2}W_{\alphap}}\right)\\
&-\frac{\pi}{2}\left[(I-\H)\PHbar\Wbar\PH W_{\alphap},\H\right]\left(\frac{e^{i\alphap}}{\frac{I+\H}{2}W_{\alphap}}\right)+\calR\\
=&\left[\frac{\PHbar\Vbar}{\PH W_{\alphap}},\H\right]\left(e^{i\alphap}\left(\frac{I+\H}{2}\right)V_{\alphap}\right)\\
&-\left[\PHbar\Vbar\PH V_{\alphap},\H\right]\left(\frac{e^{i\alphap}}{\frac{I+\H}{2}W_{\alphap}}\right)+\calR=\calR.
\end{align*}
Therefore since $\A_{1}$ is bounded away from zero

\begin{align*}
(I-\H)f\left(\frac{\tL e^{i\alphap}}{i\A_{1}}-\frac{e^{i\alphap}}{\frac{I+\H}{2}W_{\alphap}}\right)=&[f,\H]\PH\left(\frac{\tL e^{i\alphap}}{i\A_{1}}-\frac{e^{i\alphap}}{\frac{I+\H}{2}W_{\alphap}}\right)+\calR\\
=&[f,\H]\PH\left(\frac{e^{i\alphap}L}{i\A_{1}}-\frac{e^{i\alphap}}{W_{\alphap}}\right)+\calR.
\end{align*}
and the claim follows from Lemma \ref{lem: Yosihara}.
\end{proof}
Applying $(I-\H)$ to both sides of \eqref{homo 2 pre} and with

\begin{align*}
\calS:=i\A_{1}\frac{a_{t}}{a}\circ h^{-1}-\left(\frac{I+\H}{2}\right)W_{\alphap}\left(\calP\left(\frac{I+\H}{2}\right)V-\frac{\pi}{2}[\Vbar,\calH]\frac{\Wbar_{\alphap}}{W_{\alphap}}\right)
\end{align*}
 we obtain

\begin{align}\label{homo V temp 1}
\begin{split}
(I-\H)\calP\left(\frac{I-\H}{2}V\right)=&\left[e^{-i\alphap}\calS,\H\right]\frac{e^{i\alphap}}{\frac{I+\H}{2}W_{\alphap}}+(I-\H)\calR_{4}\\
=&\frac{1}{2}\left[e^{-i\alphap}(I-\H)\calS,\H\right]\frac{e^{i\alphap}}{\frac{I+\H}{2}W_{\alphap}}+\frac{1}{2}\left[[e^{-i\alphap},\H]\calS,\H\right]\frac{e^{i\alphap}}{\frac{I+\H}{2}W_{\alphap}}+(I-\H)\calR_{4}.
\end{split}
\end{align} 
To see that the first two terms in the last line are linear in $(I-\H)V, (I-\H)W,$ and $\frac{i\A_{1}}{W_{\alphap}}-L$, we want to mimic the proof of Proposition \ref{prop: at over a}, for which we need to introduce the Riemann mapping. First let $h$ be the function on $\bbR$ defined by
\begin{align}\label{h construct}
\frac{dh}{dt}=B(h,t),\quad h(\alpha,0)=\alpha.
\end{align}
Since $h$ is a diffeomorphism at $t=0$ and $h_{\alphap}$ satisfies the linear ODE

\begin{align*}
\frac{dh_{\alpha}}{dt}=B_{\alphap}h_{\alpha},
\end{align*}
$h$ is a diffeomorphism at least for a short time
and $\partial_{t}(f\circ h)=(\partial_{t}+B\pap)f\circ h$ for all time.
Let $\TPhi^{-1}(t,\cdot)$ be the holomorphic function with boundary value $\TPhi^{-1}(t,e^{i\alphap})=\left(\frac{I+\H}{2}W\right)(t,\alphap)$. Since $\TPhi^{-1}_{w}(0,\cdot)$ is never zero on the disc $\D$, the same is true for $\TPhi^{-1}_{w}(t,\cdot)$ for small $t$ by the Cauchy integral formula for the derivative of a holomorphic function. Therefore $\TPhi^{-1}(t,\cdot)$ has an inverse, which we denote by $\TPhi(t,\cdot): \TPhi^{-1}(t,\D)\to \D$. Note that with this definition $\TPhi(t,\frac{I+\H}{2}W(t,\alphap))=e^{i\alphap}$. It follows that if $f$ is the boundary value of a holomorphic function on $\D$, i.e., $f(\alphap)=F(e^{i\alphap})$ for a holomorphic function $F$ on $\D$, then $f\circ h$ is the boundary value of the holomorphic function $G=F\circ \TPhi$ on $\TPhi^{-1}(t,\D)$. Introducing the variable

\begin{align*}
\tz:=\frac{I+\H}{2}W\circ h
\end{align*}
we can write 

\begin{align*}
\frac{I+\H}{2}V\circ h=\widetilde{F}(t,\tz).
\end{align*}
Now the same argument as in the proof of Proposition \ref{prop: at over a} implies that $(I-\H)\calS$ and $[e^{-i\alphap},\H]\calS=2e^{-i\alphap}\Av(\calS)$ are linear in $(I-\H)V, (I-\H)W,$ and $\frac{i\A_{1}}{W_{\alphap}}-L$. 
Next we compute the left hand side of \eqref{homo V temp 1}.

\begin{align*}
(I-\H)\calP\left(\frac{I-\H}{2}V\right)=&(\pt+B\pap)^{2}\left((I-\H)V\right)\\
&+(\pt+B\pap)[B,\H]\pap\left(\frac{I-\H}{2}V\right)+[B,\H]\pap(\pt+B\pap)\left(\frac{I-\H}{2}V\right)\\
=&(\pt+B\pap)^{2}\left((I-\H)V\right)+\calR.
\end{align*}
Similarly,

\begin{align*}
(I-\H)\A|D|\left(\frac{I-\H}{2}V\right)=\A|D|(I-\H)V+[\A,\H]|D|\left(\frac{I-\H}{2}V\right)=\A|D|(I-\H)V+\calR.
\end{align*}
Combining these observations with \eqref{homo V temp 1}, we get 

\begin{align}\label{homo 2}
\calP\left(\frac{I-\H}{2}V\right)=\calR.
\end{align}
Note that by Claim \ref{claim: gain regularity} and \eqref{esti R1-3}, to bound, say, the $H^{2}_{\alphap}$ norms of $(I-\H)V$ and $(I-\H)W$, we only need to use the $L^{2}_{\alphap}$ norm of $L-i\frac{\A_{1}}{W_{\alphap}}$. Therefore to derive the equation for $L-i\frac{\A_{1}}{W_{\alphap}}$, we can write terms involving derivatives of $(I-\H)V$ and $(I-\H)W$ as $\calR$. To derive this equation for $L-\frac{i\A_1}{W_\alphap}.$ we first note that

\Aligns{
(\pt+B\pap)(I-\CH)\Wbar=[\Vbar,\CH]\frac{\Wbar_\alphap}{W_\alphap}+(I-\CH)V=[\Vbar,\CH]\frac{\Wbar_\alphap}{W_\alphap}+\calR,
}
where for the last equality we have used the fact that $\CH f-\tcalH f=\calR$.This computation and the fact that $|D|V=-i\pap V+|D|(I-\H)V$ allow us to write the first equation in \eqref{Actual eq 1} as

\Align{\label{homo 3 temp 1}
(\pt+B\pap)\left(L-i\A\pap\Wbar\right)=\frac{i}{W_\alphap}\left(\A_1\frac{a_{t}}{a}\circ h^{-1}-\A_1\frac{\bfa_{t}}{\bfa}\circ h^{-1}\right)-\A|D|(I-\H)V+\calR.
}
Here we have used the notation

\Aligns{
\bfa:=\frac{\A\circ h}{h_{\alpha}}
}
so in particular since $\A=\frac{L\Lbar}{\A_1}$

\Align{\label{homo 3 temp 2}
\A_1\frac{\bfa_{t}}{\bfa}\circ h^{-1}=\frac{\Lbar(\pa+B\pap)L}{\A}+\frac{L(\pt+B\pap)\Lbar}{\A}-(\pt+B\pap)\A_1-\A_1B_\alphap.
}
To write \eqref{homo 3 temp 1} as a homogeneous linear equation in $(I-\H)V, (I-\H)W,$ and $L-\frac{i\A_1}{W_\alphap}$ we need to study the right hand side of \eqref{homo 3 temp 2} more carefully. Since by \eqref{Actual eq 1} and with the notation $L_t:=(\pt+B\pap)L$ the quantity $\Lbar(L_t+\A|D|V)-\frac{\pi}{2}\Lbar(\tcalH-\CH)V$ is purely real,

\Aligns{
\frac{\Lbar L_t}{\A}+\frac{L\Lbar_t}{\A}=&2\frac{\Lbar(L_t+\A|D|V)}{\A}-\Lbar |D|V-L|D|\Vbar+\calR\\
=&2\frac{\Lbar}{\A}(L_t+\A|D|V)-L|D|\Vbar-\Lbar|D| V+\calR\\
=&2\A_1\frac{a_t}{a}\circ h^{-1}-L|D|\Vbar-\Lbar|D| V+\calR,
} 
which means

\begin{align*}
\A_{1}\frac{\bfa_{t}}{\bfa}\circ h^{-1}=2\A_{1}\frac{a_{t}}{a}\circ h^{-1}-L|D|\Vbar-\Lbar|D|V-(\pt+B\pap)\A_{1}-\A_{1}B_{\alphap}.
\end{align*}
Together with \eqref{homo 3 temp 1} and \eqref{homo 3 temp 2} this gives

\Align{\label{homo 3 temp 3}
(\pt+B\pap)\left(L-i\A\pap\Wbar\right)=&\frac{i}{W_\alphap}\left(-\A_1\frac{a_t}{a}\circ h^{-1}+\Lbar|D|V+L|D|\Vbar+(\pt+B\pap)\A_1+\A_1B_\alphap\right)+\calR\\
=:&\frac{i}{W_{\alphap}}\calT+\calR.
}
Since $\calT$ is purely real,

\begin{align*}
\calT=\Im(i\calT)=\Im\left((I-\H)i\calT\right)+\Av(\calT).
\end{align*}
First we compute

\begin{align*}
L|D|\Vbar+\Lbar|D|V=&L|D|\Vbar-i\A W_{\alphap}|D|\left(\frac{I-\H}{2}\right)V-i\A W_{\alphap}|D|\left(\frac{I+\H}{2}\right)V+\calR\\
=&\left((\pt+B\pap)\frac{I-\H}{2}V\right)|D|\Vbar+\left((\pt+B\pap)\frac{I+\H}{2}V+\frac{\pi}{2}(I-\cH)\Wbar\right)|D|\Vbar\\
&-i\A W_{\alphap}|D|\left(\frac{I+\H}{2}\right)V+\calR\\
=&\left((\pt+B\pap)\frac{I+\H}{2}V+\frac{\pi}{2}(I-\cH)\Wbar\right)|D|\Vbar-i\A W_{\alphap}|D|\left(\frac{I+\H}{2}\right)V+\calR.
\end{align*}
Therefore 

\begin{align}\label{homo 3 temp 6}
\begin{split}
\frac{i}{W_{\alphap}}(I-\H)\left(L|D|\Vbar+\Lbar|D|V\right)
=&\frac{i}{W_{\alphap}}(I-\H)\left(\left((\pt+B\pap)\frac{I+\H}{2}V+\frac{\pi}{2}(I-\cH)\Wbar\right)|D|\Vbar\right.\\
&\qquad\qquad\qquad\qquad\left.-i\A \left(\frac{I+\H}{2}\right)W_{\alphap}|D|\left(\frac{I+\H}{2}\right)V\right)+\calR.
\end{split}
\end{align}
With the notation $\calU:=(\pt+B\pap)\frac{I+\H}{2}V+\frac{\pi}{2}(I-\cH)\Wbar$

\begin{align}\label{homo 3 temp 7}
\begin{split}
\calU|D|\Vbar=&\calU|D|\PHbar\Vbar+\calU|D|\OPHbar\Vbar\\
=&\calU i\pap\frac{I-\H}{2}\Vbar+\calR=i\calU \pap\left(\frac{I-\H}{2}(\pt+B\pap)W\right)+\calR\\
=&i\calU\pap(\pt+B\pap)\left(\frac{I-\H}{2}\right)W+\frac{i}{2}\calU\pap[B,\H]W_{\alphap}+\calR\\
=&\frac{1}{2}i\calU\pap\left(\calR_{1}+\calR_{2}+\calR_{3}\right)
+i\calU\pap(\pt+B\pap)\left(\frac{I+\H}{2}\right)W-i\calU\pap\left(\frac{I+\H}{2}\right)\Vbar+\calR\\
=&i\calU B_{\alphap}\left(\frac{I+\H}{2}\right)W_{\alphap}+i(\pt+B\pap)\left(\calU\left(\frac{I+\H}{2}\right)W_{\alphap}\right)-i\left(\frac{I+\H}{2}\right)W_{\alphap}(\pt+B\pap)\calU+\calR.
\end{split}
\end{align}

Combining  \eqref{homo 3 temp 3}--\eqref{homo 3 temp 7} we get 

\begin{align*}
(I-\H)(i\calT)=&i(I-\H)\left(-\A_{1}\frac{a_{t}}{a}\circ h^{-1}+L|D|\Vbar+\Lbar|D|V+(\pt+B\pap)\A_{1}+\A_{1}B_{\alphap}\right)\\
=&-(I-\H)\left(\calU B_{\alphap}\left(\frac{I+\H}{2}\right)W_{\alphap}\right)-(I-\H)\left((\pt+B\pap)
\left(\calU\left(\frac{I+\H}{2}\right)W_{\alphap}\right)\right)\\
&+(I-\H)\left(\left(\frac{I+\H}{2}\right)W_{\alphap}\left((\pt+B\pap)\calU+\A|D|\left(\frac{I+\H}{2}\right)V\right)-i\A_{1}\frac{a_{t}}{a}\circ h^{-1}\right)\\
&+i(I-\H)\left((\pt+B\pap)\A_{1}+\A_{1}B_{\alphap}\right)+\calR\\
=&-(I-\H)\left(\calU B_{\alphap}\left(\frac{I+\H}{2}\right)W_{\alphap}\right)-(I-\H)\left((\pt+B\pap)
\left(\calU\left(\frac{I+\H}{2}\right)W_{\alphap}\right)\right)\\
&+i(I-\H)\left((\pt+B\pap)\A_{1}+\A_{1}B_{\alphap}\right)+\calR.
\end{align*}
To compute the second term, we first note that 

\begin{align*}
\calU=(\pt+B\pap)\PH V+\pi\PHbar \Wbar-\frac{\pi}{2}(I+\tcalH)\Wbar+\calR:=\tcalU+\calR.
\end{align*}
By a computation similar to Proposition \ref{prop: A1} it follows that 
\begin{align*}
-(I-\H)\left((\pt+B\pap)\left(\calU\left(\frac{I+\H}{2}\right)W_{\alphap}\right)\right)
=&-(\pt+B\pap)\left((I-\H)\left(\calU\left(\frac{I+\H}{2}\right)W_{\alphap}\right)\right)\\
&-[B,\H]\partial_{\alphap}\left(\calU\left(\frac{I+\H}{2}\right)W_{\alphap}\right)\\
=&-(\pt+B\pap)\left(\left[\Vbar,\H\right]\PH V_{\alphap}+\pi\left[\Wbar,\H\right]\PH W_{\alphap}\right)\\
&-[B,\H]\partial_{\alphap}\left(\calU\left(\frac{I+\H}{2}\right)W_{\alphap}\right)+\calR.
\end{align*}
Therefore 

\begin{align}\label{homo 3 temp 10}
\begin{split}
\Im((I-\H)i\calT)=&-(\pt+B\pap)\Im \left(\left[\Vbar,\H\right]\PH V_{\alphap}+\pi\left[\Wbar,\H\right]\PH W_{\alphap}\right)\\
&-\Im [B,\H]\partial_{\alphap}\left(\calU\left(\frac{I+\H}{2}\right)W_{\alphap}\right)+(\pt+B\pap)\A_{1}-\pt\Av(\A_{1})
+\A_{1}B_{\alphap}+\calR\\
&-\Im(I-\H)\left(\calU B_{\alphap}\left(\frac{I+\H}{2}\right)W_{\alphap}\right)\\
=&-\Im [B,\H]\partial_{\alphap}\left(\calU\left(\frac{I+\H}{2}\right)W_{\alphap}\right)+\A_{1}B_{\alphap}-\Im(I-\H)\left(\calU B_{\alphap}\left(\frac{I+\H}{2}\right)W_{\alphap}\right)+\calR\\
=&-\Im [B,\H]\partial_{\alphap}\left(\tcalU\left(\frac{I+\H}{2}\right)W_{\alphap}\right)+\A_{1}B_{\alphap}-\Im(I-\H)\left(\tcalU B_{\alphap}\left(\frac{I+\H}{2}\right)W_{\alphap}\right)+\calR.
\end{split}
\end{align}
We compute

\begin{align*}
&-[B,\H]\partial_{\alphap}\left(\tcalU\left(\frac{I+\H}{2}\right)W_{\alphap}\right)-(I-\H)\left(\tcalU B_{\alphap}\left(\frac{I+\H}{2}\right)W_{\alphap}\right)\\
=&-B_{\alphap}(I-\H)\left(\tcalU\PH W_{\alphap}\right)-\partial_{\alphap}[B,\H]\left(\tcalU\PH W_{\alphap}\right).
\end{align*}
Note that 

\begin{align*}
-\Im\left(B_{\alphap}(I-\H)\left(\tcalU\PH W_{\alphap}\right)\right)=-\A_{1} B_{\alphap}+\Av(\A_{1})B_{\alphap}+\calR
\end{align*}
and
\begin{align*}
-\Im\left(\partial_{\alphap}[B,\H]\left(\tcalU\PH W_{\alphap}\right)\right)=&-\Im\left(\partial_{\alphap}[B,\H](i\A_{1})\right)+\calR\\
=&-\Im\left(\partial_{\alphap}[B,\Av](i\A_{1})\right)+\calR=-\Av(\A_{1}) B_{\alphap}+\calR.
\end{align*}
Combining these observations with \eqref{homo 3 temp 10} we obtain

\begin{align*}
(\pt+B\pap)(L-i\A\Wbar_{\alphap})=\calR.
\end{align*}

\item This is a direct consequence of the fact that $L=\frac{i\A_{1}}{W_{\alphap}}$ and the definition of $h$.
\end{enumerate}

\end{proof}

The proof of Theorem \ref{thm: lwp} now follows from local the well-posedness of \eqref{Actual eq 1}--\eqref{Actual eq 2}:

\begin{proof}[Proof of Theorem \ref{thm: lwp}]
By Proposition \ref{prop: going back} it suffices to show local well-posedness of \eqref{Actual eq 1}--\eqref{Actual eq 2}. The proof of local well-posedness for the system \eqref{Actual eq 1}--\eqref{Actual eq 2} is almost identically the same as the proof of Theorem 5.10 in \cite{Wu1} where the vanishing viscosity method us used. In fact the only difference is that unlike in \cite{Wu1}, here we also need to control $W=Z.$ But by \eqref{Actual eq 1} $W$ satisfies a transport equation and therefore control of $W$ follows from control of $V$ by integration. We refer the reader to \cite{Wu1} Section 5, and leave the necessary routine modifications to the reader.
\end{proof}

\appendix


\section{The Hilbert Transform}\label{app: Hilbert transform}
In this appendix we recall some facts about the Hilbert transform. If $\Omega$ is a bounded domain in $\C$ with $C^{2}_{t,\alpha}$ boundary and $f$ is a function defined on $\partial\Omega$ then the Hilbert transform $Hf$ of $f$ with respect to $\Omega$ is defined as

\Aligns{
Hf(z_0):=\lim_{\epsilon\to 0^{+}}\frac{1}{\pi i}\int_{\gamma_\epsilon}\frac{f(w)}{w-z_0}dw,
}
where $\gamma_\epsilon$ is the portion of $\partial\Omega$ obtained by removing a segment of $\partial\Omega$ which lies within a circle of radius $\epsilon$ centered at $z_0\in\partial\Omega.$ Given a $C^{2}_{t,\alpha}$ parametrization $z:[0,2\pi]\to\partial\Omega$ of $\partial\Omega$ we identify $2\pi-$periodic functions on $\R$ with functions on $\partial\Omega,$ and for any such function $f$ we write

\Aligns{
H f(\alpha):=\frac{\pv}{\pi i}\int_{0}^{2\pi}\frac{f(\beta)}{z(\beta)-z(\alpha)}z_{\beta}(\beta)d\beta.
}
The relevant results from this appendix are summarized in the following proposition.

\begin{proposition}\label{prop: hilbert}
Suppose that $\Omega$ is a bounded domain in $\bbC$ with $C^{2}$ boundary $\partial\Omega$. Let $f$ be a Lipschitz continuous function on $\partial\Omega$ and $H f$ be its Hilbert transform. Then $H f=f$ if and only if $f$ is the boundary value of a holomorphic function in $\Omega$ and $H f=-f$ if and only if $f$ is the boundary value of a holomorphic function $F$ in $\Omega^{c}$ satisfying $F(z)\to0$ as $|z|\rightarrow\infty$.
\end{proposition}
\begin{proof}
Suppose that $\Omega$ is a bounded domain in $\bbC$ and $\gamma:=\partial\Omega$ has $C^{2}_{t,\alpha}$. Let $f$ be a continuous function defined on $\partial\Omega$. The following Cauchy integral

\begin{align}\label{Cauchy integral}
C_f(z):=\frac{1}{2\pi i}\int_{\gamma}\frac{f(w)}{w-z}dw
\end{align}
defines a holomorphic function when $z\slashed{\in}\gamma$. In this subsection, we will introduce the Hilbert transforms associated to $\Omega$ and $\Omega^{c}$ by considering the limit of $C_f(z)$ as $z$ approaches $z_{0}$ from $\Omega$ and $\Omega^{c}$ where $z_0$ is a point on $\partial\Omega.$ Here all integrals are understood as counterclockwise, unless otherwise stated. Let us first consider the limit from the inside. 

\begin{align}\label{hilbert temp 1}
\lim_{z\rightarrow z_{0}}\frac{1}{2\pi i}\int_{\gamma}\frac{f(w)}{w-z}dw=\lim_{z\rightarrow z_{0}}\frac{1}{2\pi i}\int_{\gamma_{\epsilon}+\xi_{\epsilon}}\frac{f(w)}{w-z}dw=\lim_{\epsilon\rightarrow0^{+}}\lim_{z\rightarrow z_{0}}\frac{1}{2\pi i}\int_{\gamma_{\epsilon}+\xi_{\epsilon}}\frac{f(w)}{w-z}dw,
\end{align}
where $\gamma_{\epsilon}$ is the portion of $\gamma$ obtained by subtracting the segment $\xi_{\epsilon}$  about $z_{0}$ which lies within the circle of radius $\epsilon$ centered at $z_0.$ We recognize the limit over $\gamma_\epsilon$ as one half of the Hilbert transform of $f$ associated to $\Omega$:

\begin{align}\label{hilbert temp 2}
\frac{1}{2}H f(z_{0})=\lim_{\epsilon\rightarrow 0^{+}}\frac{1}{2\pi i}\int_{\gamma_{\epsilon}}\frac{f(w)}{w-z_{0}}dw=\lim_{\epsilon\rightarrow0^{+}}\lim_{z\rightarrow z_{0}}\frac{1}{2\pi i}\frac{f(w)}{w-z}dw.
\end{align}
On the other hand,

\Aligns{
\lim_{\epsilon\rightarrow0^{+}}\lim_{z\to z_0}\int_{\xi_\epsilon}\frac{f(w)}{w-z}dw&=\lim_{\epsilon\rightarrow0^{+}}\lim_{z\to z_0}\left(\int_{\xi_\epsilon}\frac{f(w)-f(z_0)}{w-z}dw+\int_{\xi_\epsilon}\frac{f(z_0)}{w-z}dw\right)\\
&=\lim_{\epsilon\rightarrow0^{+}}\lim_{z\to z_0}\int_{\xi_0}\frac{f(z_0)}{w-z}dw.
}
Now with $C_\epsilon$ denoting the part of the circle of radius $\epsilon$ centered at $z_0$ which lies within $\Omega$ we have

\Align{\label{hilbert temp 3}
\lim_{\epsilon\rightarrow0^{+}}\lim_{z\to z_0}\int_{\xi_\ep}\frac{dw}{w-z}&=\lim_{\epsilon\rightarrow0^{+}}\lim_{\epsilon\rightarrow0^{+}}\int_{\xi_\epsilon+C_\epsilon}\frac{dw}{w-z}-\lim_{\epsilon\rightarrow0^{+}}\int_{C_\epsilon}\frac{dw}{w-z_0}\\
&=2\pi i-\lim_{\epsilon\rightarrow0^{+}}\int_{\pi+O(\epsilon)}^{2\pi+O(\epsilon)}\frac{i\epsilon e^{i\theta}}{\epsilon e^{i\theta}}=\pi i.
}
Combining this with \eqref{hilbert temp 1} and \eqref{hilbert temp 2} we get

\Aligns{
H f(z_0)=2\lim_{z\to z_0}C_f(z)-f(z_0).
}
Since $C_f$ is a holomorphic function inside $\Omega,$ and $\lim_{z\to z_0}C_f(z)=f(z_0)$ if $f$ can be extended to a holomorphic function inside $\Omega,$ we conclude that $f$ is the boundary value of a holomorphic function inside $\Omega$ if and only if $Hf(z_0)=f(z_0)$ for all $z_0\in\partial\Omega.$

The computation is similar for the case where $z\to z_0$ from the outside (i.e. $z\in \Omega^c$). In this case in \eqref{hilbert temp 3} we define $C_\epsilon$ to be the part of the circle of radius $\epsilon$ centered at $z_0,$ parametrized clockwisely, which lies in $\Omega^c.$ It then follows that

\Aligns{
\int_{\xi_\epsilon+C_\epsilon}\frac{dw}{w-z}=-2\pi i,
} 
and hence

\Aligns{
Hf(z_0)=2\lim_{z\to z_0}C_f(z)-3f(z_0),
}
where now the limit is understood to be from the outside. Now notice that from the definition \eqref{Cauchy integral} of the Cauchy integral that $C_f$ is holomorphic in $\Omega^c$ and decays like $\frac{1}{|z|}$ as $|z|\to\infty.$  Therefore if $f$ if $Hf=-f$ then $f$ is the boundary value of a holomorphic function in $\Omega^c$ decaying like $\frac{1}{|z|}$ as $|z|\to\infty.$ Conversely, if $f$ is the boundary value of such a holomorphic function, then defining $U=\{\frac{1}{z}\mathrm{~s.t.~}z\in\Omega^c\}\subseteq \C,$ we have

\Aligns{
\lim_{\stackrel{z\to z_0}{z\in\Omega^c}} C_f(z)=\lim_{\stackrel{z\to z_0}{z\in\Omega^c}}\int_{\partial\Omega}\frac{f(w)}{w-z}dw=\lim_{\stackrel{z\to 1/z_0}{z\in U}}\frac{1}{z}\int_{\partial U}\frac{\frac{f(1/u)}{u}}{u-\frac{1}{z}}du=f(z_0),
}
and therefore $Hf(z_0)=-f(z_0).$
\end{proof}


\section*{Notations}
For the reader's convenience we give the definitions of some of the symbols used commonly in this work.

\Aligns{
&H f(t,\alpha)=\frac{\pv}{\pi i}\int_0^{2\pi}\frac{f(t,\beta)}{z(t,\beta)-z(t,\alpha)}z_\beta(\beta)d\beta.\\
&\CH f(t,\alpha)=\frac{\pv}{\pi i}\int_0^{2\pi}\frac{f(t,\beta)}{\zed(t,\beta)-\zed(t,\alpha)}\zed_\beta(\beta)d\beta,\quad \zed(t,\cdot)=z(t,j(t,\cdot)),\quad j(t,\cdot):[0,2\pi]\to[0,2\pi] \mathrm{~ a ~ diffeomorphism}.\\
&\H f(t,\alphap)=\frac{\pv}{\pi i}\int_0^{2\pi}\frac{f(t,\betap)}{e^{i\betap}-e^{i\alphap}}ie^{i\betap}d\betap, \qquad \tH f(t,\alpha)=\frac{\pv}{2\pi i}\int_0^{2\pi} f(t,\beta)\cot\left(\frac{\beta-\alpha}{2}\right)d\beta.\\
&\AV(f):=\frac{1}{2\pi i}\int_0^{2\pi}\frac{f(\alpha)}{z(t,\alpha)}z_\alpha(t,\alpha)d\alpha, \qquad \Av(f)=\frac{1}{2\pi}\int_0^{2\pi}f(\alpha)d\alpha.\\
&Kf=\Re H=\frac{1}{2}(H+\Hbar)f,\quad \CK f=\Re \CH f=\frac{1}{2}(\CH+\CHbar)f,\quad f\mathrm{~real~valued}.\\
&K^*f=-\Re\left\{\frac{z_\alpha}{|z_\alpha|}H\frac{|z_\alpha|}{z_\alpha}f\right\},\quad \CK^*f=-\Re\left\{\frac{\zed_\alpha}{|\zed_\alpha|}\CH\frac{|\zed_\alpha|}{\zed_\alpha}f\right\}, \quad f\mathrm{~ real ~ valued}.\\
&a=-\frac{1}{|z_\alpha|}\frac{\partial P}{\partial n},\quad {\bfn \mathrm{~unit~exterior~ normal}}.
}
Let $h$ be as defined in Figure \ref{fig: h} and $k$ as defined in Remark~\ref{prop: k existence}.

\Aligns{\allowdisplaybreaks
&Z(t,\alphap)=z(t,h^{-1}(t,\alphap)), \quad \zeta(t,\alpha)=z(t,k^{-1}(t,\alpha)).\\
&Z_t(t,\alphap)=z_t(t,h^{-1}(t,\alphap)), \quad Z_{tt}(t,\alphap)=z_{tt}(t,h^{-1}(t,\alphap)),\quad Z_{ttt}(t,\alphap)=z_{ttt}(t,h^{-1}(t,\alphap)).\\
&B=h_t\circ h^{-1},\qquad b=k_t\circ k^{-1}.\\
&\A=(ah_\alpha)\circ h^{-1},\quad \A_1=\A|{Z_{,\alphap}}|^2,\quad A=(ak_\alphap)\circ k^{-1}.\\
&G=(I+\CH)\Zbar.\\
&D_\alpha=\frac{1}{|z_\alpha|}\pa,\quad D_\alphap=\frac{1}{|{Z_{,\alphap}}|}\pap.\\
&[Z_{t},Z_{t};D_{\alphap}\Zbar_{t}]=\frac{ie^{i\alphap}}{\pi i}\int_{0}^{2\pi}\left(\frac{Z_t(t,\betap)-Z_t(t,\alphap)}{e^{i\betap}-e^{i\alphap}}\right)^2\frac{e^{i\betap}}{Z_\betap(t,\betap)}{\Zbar_{t,\betap}}(t,\betap)d\beta'.\\
&\P=(\pt+b\pa)^2+ia\pa-\pi.\\
&\ep=|z|^2-1,\quad \mu=\ep\circ k^{-1},\quad \delta=(I-H)\ep,\quad \chi=\delta\circ k^{-1},\quad \eta=\zeta_\alpha-i\zeta.\\
&u=z_t\circ k^{-1},\quad w=z_{tt}\circ k^{-1},\quad v=\delta_{t}\circ k^{-1}.
}

\bibliographystyle{plain}
\bibliography{twobodybib}

 \bigskip

\centerline{\scshape Lydia Bieri, Shuang Miao, Sohrab Shahshahani, Sijue Wu}
\medskip
{\footnotesize
 \centerline{Department of Mathematics, The University of Michigan}
\centerline{2074 East Hall, 530 Church Street
Ann Arbor, MI  48109-1043, U.S.A.}
\centerline{\email{lbieri@umich.edu, shmiao@umich.edu, shahshah@umich.edu, sijue@umich.edu}}

\end{document}